\documentclass[11pt]{article}
\usepackage[utf8x]{inputenc}
\usepackage[margin=1.5in]{geometry}
\usepackage{microtype}
\usepackage[T1]{fontenc}
\usepackage{ae,aecompl}
\usepackage{times}
\usepackage{tikz-cd}
\usepackage{mathrsfs} 
\usepackage{harmony}
\usepackage{verbatim,amsmath,amsthm,amsfonts,amssymb,latexsym,graphicx,mathtools,extpfeil,color}
\usepackage{tikz}
\usepackage{epstopdf,pinlabel}
\usepackage[all]{xy}
\usepackage{enumitem}
\usepackage{graphicx}
\usepackage{booktabs}
\usepackage{caption}
\usepackage{tabularx}
\usepackage{subcaption}
\usepackage{cancel}
\usepackage{slashed}
\DeclareMathAlphabet{\mathpzc}{OT1}{pzc}{m}{it}
\usepackage[colorlinks,pagebackref,hypertexnames=false]{hyperref} \usepackage[alphabetic,backrefs,msc-links]{amsrefs}
\usepackage{aliascnt}
\numberwithin{equation}{section}

\newcommand{\bR}{{\bf R}}

\newcommand{\cM}{\mathcal{M}}

\newcommand{\cS}{\mathcal{S}}

\newcommand{\fa}{{\mathfrak a}}
\newcommand{\fb}{{\mathfrak b}}

\newcommand{\fm}{{\mathfrak m}}

\newcommand{\fz}{{\mathfrak z}}

\newcommand{\fZ}{{\mathfrak Z}}

\newcommand{\R}{\bR}

\newcommand{\su}{\mathfrak{su}}

\DeclareMathOperator{\Hom}{Hom}

\DeclareMathOperator{\im}{im}

\DeclareMathOperator{\tr}{tr}

\newcommand{\id}{{\rm id}}

\renewcommand{\epsilon}{\varepsilon}

\def\({\mathopen{}\left(}
\def\){\right)\mathclose{}}
\def\<{\mathopen{}\left<}
\def\>{\right>\mathclose{}}

\usepackage{multicol, color}

\definecolor{gold}{rgb}{0.85,.66,0}
\definecolor{cherry}{rgb}{0.9,.1,.2}
\definecolor{burgundy}{rgb}{0.8,.2,.2}
\definecolor{orangered}{rgb}{0.85,.3,0}
\definecolor{orange}{rgb}{0.85,.4,0}
\definecolor{olive}{rgb}{.45,.4,0}
\definecolor{lime}{rgb}{.6,.9,0}
\definecolor{green}{rgb}{.2,.7,0}
\definecolor{grey}{rgb}{.4,.4,.2}
\definecolor{brown}{rgb}{.4,.3,.1}

\def\makeautorefname#1#2{\AtBeginDocument{\expandafter\def\csname#1autorefname\endcsname{#2}}}

\newcommand{\mynewtheorem}[2]{
  \newaliascnt{#1}{equation}          
  \newtheorem{#1}[#1]{#2}
  \aliascntresetthe{#1}
  \makeautorefname{#1}{#2}
}
\mynewtheorem{theorem}{Theorem}
\mynewtheorem{prop}{Proposition}
\mynewtheorem{cor}{Corollary}
\mynewtheorem{construction}{Construction}
\mynewtheorem{lemma}{Lemma}
\mynewtheorem{conjecture}{Conjecture}

\numberwithin{substep}{step}
\makeautorefname{step}{Step}
\makeautorefname{substep}{Step}

\numberwithin{subcase}{case}
\makeautorefname{case}{Case}
\makeautorefname{subcase}{case}

\theoremstyle{remark}
\mynewtheorem{remarkx}{Remark}
\newenvironment{remark}
{\pushQED{\qed}\remarkx}
{\popQED\endremarkx}

\theoremstyle{definition}
\mynewtheorem{definitionx}{Definition}
\newenvironment{definition}
{\pushQED{\qed}\definitionx}
{\popQED\enddefinitionx}
\mynewtheorem{examplex}{Example}
\newenvironment{example}
{\pushQED{\qed}\examplex}
{\popQED\endexamplex}
\mynewtheorem{exercise}{Exercise}
\mynewtheorem{convention}{Convention}
\newtheorem*{convention*}{Convention}
\newtheorem*{conventions*}{Conventions}
\mynewtheorem{question}{Question}
        
\makeautorefname{chapter}{Chapter}
\makeautorefname{section}{Section}
\makeautorefname{subsection}{Section}
\makeautorefname{subsubsection}{Section}

\theoremstyle{theorem}
\newtheorem{theorem-intro}{Theorem}
\newtheorem{cor-intro}{Corollary}
\newtheorem{lemma-intro}{Lemma}
\theoremstyle{remark}
\newtheorem{remarkintrox}{Remark}
\makeautorefname{remarkintrox}{Remark}
\newenvironment{remarkintro}
{\pushQED{\qed}\remarkintrox}
{\popQED\endremarkintrox}

\theoremstyle{definition}
\newtheorem{defnintrox}{Definition}
\makeautorefname{defnintrox}{Definition}
\newenvironment{defnintro}
{\pushQED{\qed}\defnintrox}
{\popQED\enddefnintrox}
\usepackage{accents}
\usepackage[normalem]{ulem}

\DeclareFontFamily{U}{mathx}{\hyphenchar\font45}
\DeclareFontShape{U}{mathx}{m}{n}{<-> mathx10}{}
\DeclareSymbolFont{mathx}{U}{mathx}{m}{n}
\DeclareMathAccent{\widecheck}{0}{mathx}{"71}

\DeclareSymbolFont{stmry}{U}{stmry}{m}{n}
\DeclareMathSymbol{\llbracket}{\mathopen}{stmry}{"4A}

\newcommand{\lift}{\mathscr O}

\newcommand{\hp}{q_3}

\newcommand{\F}{{\Bbb F_2}}

\newcommand{\medtilde}[1]{\mkern3mu\widetilde{\mkern-3mu#1\mkern-1mu}\mkern1mu}
\newcommand{\medhat}[1]{\mkern3mu\widehat{\mkern-3mu#1\mkern-1mu}\mkern1mu}
\newcommand{\medcheck}[1]{\mkern3mu\widecheck{\mkern-3mu#1\mkern-1mu}\mkern1mu}
\newcommand{\medbar}[1]{\mkern3mu\overline{\mkern-3mu#1\mkern-1mu}\mkern1mu}
\newcommand{\eqm}{\medtilde M}
\newcommand{\eqb}{\medtilde B}
\newcommand{\eqz}{\medtilde Z}
\newcommand{\eqn}{\medtilde N}

\title{Instantons, indefinite 4-manifolds, and Dehn surgery}
\author{Aliakbar Daemi\thanks{The work of AD was supported by NSF Grants DMS-2609113, DMS-2208181, and a Simons fellowship.} \hspace{1cm} Xingpei Liu\thanks{The work of XL was supported by NSFC 12341105.} \hspace{1cm} Mike Miller Eismeier}
\date{}

\newcommand{\Addresses}{{
  \bigskip
  \footnotesize
  Aliakbar Daemi, \textsc{Department of Mathematics, Washington University in St. Louis, One Brookings drive, Room 212,
  St. Louis, MO 63130}\par\nopagebreak
  \textit{E-mail address}: \texttt{adaemi@wustl.edu}
  \vspace{.4cm}
  
Xingpei Liu, \textsc{Beijing International Center for Mathematical Research, Peking University, Beijing 100871, China}\par\nopagebreak
  \textit{E-mail address}: \texttt{zivpei@pku.edu.cn}
  \vspace{.4cm}
  
Mike Miller Eismeier,
  \textsc{Department of Mathematics and Statistics, University of Vermont, Innovation Hall E231, 82 University Place, Burlington, VT 05404}\par\nopagebreak
  \textit{E-mail address}: \texttt{Mike.Miller-Eismeier@uvm.edu}\par\nopagebreak
  \vspace{.2cm}
  
  \textsc{Institute of Mathematics, Academia Sinica, 635 Astronomy-Mathematics Building, No. 1, Sec. 4, Roosevelt Road, Taipei 106319, Taiwan}\par\nopagebreak
  \textit{E-mail address}: \texttt{mme@as.edu.tw}
}}

\begin{document}

\newgeometry{top=1in}

\maketitle

\begin{abstract}
	We prove that there exist hyperbolic integer homology spheres with arbitrarily large Dehn surgery number. Previously, no integer homology sphere was known to have a surgery number larger than $2$. Our approach uses Fr{\o}yshov's invariant $q_3$ of integer homology spheres, which is defined in terms of mod 2 instanton homology. We show that if $W: Y \to Y'$ is a cobordism between integer homology spheres with no $2$-torsion in its first homology, then $-b^+(W) \le q_3(Y') - q_3(Y) \le b^-(W)$. We also extend both $q_3$ and the inequality to rational homology spheres.
\end{abstract}
\hypersetup{linkcolor=black}
\setcounter{tocdepth}{2}
\tableofcontents

\restoregeometry

\section{Introduction}
Every closed oriented 3-manifold can be obtained by performing surgery on a framed link in $S^3$. The first published proof, due to Wallace \cite{Wallace}, was obtained by establishing the relationship between cobordisms and what is now called surgery. Lickorish later offered a more elementary proof \cite{Lickorish} in terms of what are now called Dehn twists. This motivates the definition of {\it surgery number}, a measure of complexity for any closed oriented 3-manifold \cite{Auckly1}.  The surgery number $S(Y)$ of such a $3$-manifold $Y$ is the minimum number of components of a framed link $L \subset S^3$ such that $Y$ is obtained by surgery on $L$. A framing of a link $L \subset S^3$ is specified by assigning an integer to each connected component of $L$. One may also perform more general Dehn surgeries on $L$, parameterized by a choice of rational surgery coefficient on each component. We call any such link a rationally-framed link. Replacing integral surgery with Dehn surgery in the definition of $S(Y)$ yields another topological invariant of $Y$, called the Dehn surgery number of $Y$ and denoted by $S_D(Y)$. It follows immediately from the definitions that $S(Y)$ is at least as large as $S_D(Y)$. 

There is a trivial algebraic lower bound: $S_D(Y)$ is bounded below by the least number of generators of $H_1(Y)$. A stronger, though less computable, lower bound is given by the weight of the fundamental group. There has been considerable interest in finding nontrivial lower bounds for $S_D(Y)$, especially in the case that $H_1(Y)$ is trivial, but it has been difficult to do so. It was first observed by \cite{GL} that every reducible integer homology sphere has $S_D(Y) \ge 2$. The first irreducible (and in fact atoroidal) example with surgery number at least $2$ was constructed by \cite{Auckly1}. Further integer homology sphere examples were produced in \cite{HKL,HL-surgery,Chen-HF-surgery} using Heegaard Floer homology, and in \cite{NST, IPT} using the Chern--Simons filtration. The first example of an integer homology sphere for which the weight of the fundamental group gives a nontrivial lower bound on the surgery number was given in \cite{ChenLodha}.
	
Relaxing the assumption on first homology, further examples with cyclic $H_1(Y)$ and $S_D(Y) \ge 2$ appeared in \cite{BL-lens-spaces,HKMP,SZ:surgery,LP-10/8,Azarpendar}, using obstructions coming from the Casson--Walker invariant, character varieties, Heegaard Floer homology, and Furuta's $10/8$-inequality. No examples have appeared thus far in the literature for which $H_1(Y)$ is cyclic and $S_D(Y)$ is known to be larger than $2$.

Auckly asked whether there are irreducible, atoroidal homology spheres with arbitrarily large surgery number. This appears as Problem 3.102(B) in Kirby's 1995 problem list \cite{Kirby}, and a variation appears as part of Problem 1.15(a) in the 2026 K3 problem list \cite{K3}. The main result of this paper is an affirmative answer to this question.

\begin{figure}
	\begin{center}
		\begin{tikzpicture}[scale=1]
			\usetikzlibrary{decorations.pathreplacing}
			
			\tikzset{
				knot/.style={
					line width=0.5mm,
					line cap=round,
					line join=round
				}
			}
			\draw[knot] (-0.2,-0.15) .. controls (2.4,0.9) and (2,-3.2) .. (-0.2,-1.4);
			\draw[knot] (0.3,-0.2) .. controls (-0.2,-2.2) and (-2.4,-1.8) .. (-1,-0.7);
			\draw[knot] (-0.45,-1.2) .. controls (-1.75,0.65) and (-0.35,1.3) .. (0, 0.9);
			\draw[knot] (-0.6,0.05) .. controls (-1, -0.25) and (-0.9, 0.16) .. (-0.78,0.23);
			\draw[knot] (0.3,0.55) .. controls (0.35,0.5) and (0.35,0.2) .. (0.35,0.15);
			
			{\draw[knot] (-0.45,0.49) .. controls (0.4, 1.3) and (2.2, 1.25) .. (2.7, 0.9);
			\draw[draw=white, line width=2mm] (1.1, 1.1) -- (2.1, 1.1);
			\draw[knot] (-0.6,0.05) .. controls (0.7, 1.1) and (1.65,1.12) .. (2.5,0.7);
			\draw[draw=white, line width=2mm] (1.1, 0.9) -- (2.1, 0.9);
			\node[scale=1.3] at (1.6, 1) {$\cdots$};

			\draw[decorate, decoration={brace,mirror,raise=1mm,amplitude=8pt}, line 				width=0.3mm] (-5, -2) -- (5,-2)
			node[midway,below=10pt] {$n$ copies};}
	
			\draw[knot] (-0.7,-0.5) -- (-0.45,-0.3);
			\draw[knot] (-0.45,-0.3) -- (-0.55,-0.07);
			\draw[knot] (-0.35, 0.08) -- (-0.2,-0.15);
			\draw[knot] (-0.65, 0.15) -- (-0.82,0.5);
			\draw[knot] (-0.45,0.25) -- (-0.65,0.7);
			\draw[knot] (-0.65,0.35) -- (-0.6, 0.39);
			\draw[knot] (0.16,0.78) -- (0.23,0.7);
			
			\draw[knot] (3.3,-0.15) .. controls (5.9,0.9) and (5.5,-3.2) .. (3.3,-1.4);
			\draw[knot] (3.8,-0.2) .. controls (3.3,-2.2) and (1.1,-1.8) .. (2.5,-0.7);
			\draw[knot] (3.05,-1.2) .. controls (1.75,0.65) and (3.15,1.3) .. (3.5, 0.9);
			\draw[knot] (2.9,0.05) .. controls (2.5, -0.25) and (2.6, 0.16) .. (2.72,0.23);
			\draw[knot] (3.8,0.55) .. controls (3.85,0.5) and (3.85,0.2) .. (3.85,0.15);
			\draw[knot] (3.05,0.49) .. controls (4.55,1.3) and (3.5,1.62) .. (3.5,1.6);
			\draw[knot] (2.9,0.05) .. controls (5.3, 1.4) and (3.55,1.82) .. (3.63,1.8);
			\draw[knot] (2.8,-0.5) -- (3.05,-0.3);
			\draw[knot] (3.05,-0.3) -- (2.95,-0.07);
			\draw[knot] (3.15, 0.08) -- (3.3,-0.15);
			\draw[knot] (2.85, 0.15) -- (2.68,0.5);
			\draw[knot] (3.05,0.25) -- (2.85,0.7);
			\draw[knot] (2.85,0.35) -- (2.9, 0.39);
			\draw[knot] (3.66,0.78) -- (3.73,0.7);
			
			\draw[knot] (-3.7,-0.15) .. controls (-1.1,0.9) and (-1.5,-3.2) .. (-3.7,-1.4);
			\draw[knot] (-3.2,-0.2) .. controls (-3.7,-2.2) and (-5.9,-1.8) .. (-4.5,-0.7);
			\draw[knot] (-3.95,-1.2) .. controls (-5.25,0.65) and (-3.85,1.3) .. (-3.5, 0.9);
			\draw[knot] (-4.1,0.05) .. controls (-4.5, -0.25) and (-4.4, 0.16) .. (-4.28,0.23);
			\draw[knot] (-3.2,0.55) .. controls (-3.15,0.5) and (-3.15,0.2) .. (-3.15,0.15);
			\draw[knot] (-3.95,0.49) .. controls (-3.1, 1.3) and (-1.3, 1.25) .. (-0.8, 0.9);
			\draw[knot] (-4.1,0.05) .. controls (-2.8, 1.1) and (-1.85,1.12) .. (-1,0.7);
			\draw[knot] (-4.25, 0.95) .. controls (-4.5,1.5) and (-4,1.62) .. (-4,1.6);
			\draw[knot] (-4.45,0.75) .. controls (-4.8,1.5) and (-4.3,1.83) .. (-4.1,1.8);
			\draw[knot] (-4.2,-0.5) -- (-3.95,-0.3);
			\draw[knot] (-3.95,-0.3) -- (-4.05,-0.07);
			\draw[knot] (-3.85, 0.08) -- (-3.7,-0.15);
			\draw[knot] (-4.15, 0.15) -- (-4.3,0.45);
			\draw[knot] (-3.95,0.25) -- (-4.15,0.7);
			\draw[knot] (-4.15,0.35) -- (-4.1, 0.39);
			\draw[knot] (-3.34,0.78) -- (-3.27,0.7);
			
			\draw[knot] (-0.7,2.1) -- (-0.7, 1.3);
			\draw[knot] (0.1,2.1) -- (0.1,1.3);
			\draw[knot] (-0.7,1.3) -- (0.1,1.3);
			\draw[knot] (-0.7,2.1) -- (0.1,2.1);
			\draw[knot] (-4.1,1.8) -- (-0.7,1.8);
			\draw[knot] (-4,1.6) -- (-0.7,1.6);
			\draw[knot] (3.5,1.6) -- (0.1,1.6);
			\draw[knot] (3.62,1.8) -- (0.1,1.8);
			
			\node at (-0.3, 1.7) {$k$};
		\end{tikzpicture}
	\end{center}
	\vspace{-5mm}
	\caption{The link $L_{n,k}$ in $S^3$. The box indicates $k$ full twists.} \label{Lnk}
\end{figure}
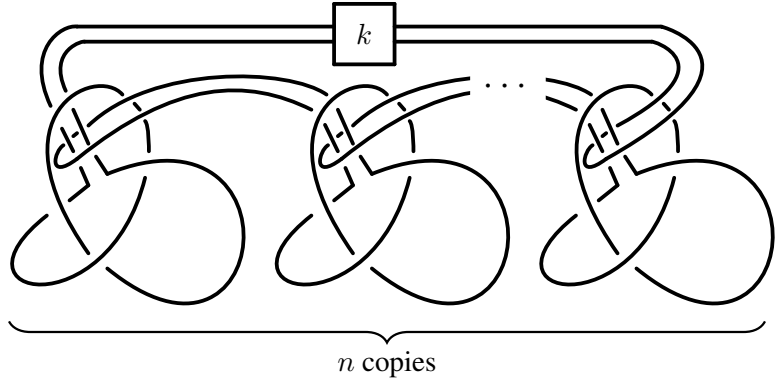

\begin{theorem-intro}\label{intro1}
	Let $\#^n P$ be the connected sum of $n \ge 0$ copies of the Poincar\'e homology sphere $P$. Then 
	\[
	  S(\#^n P) = S_D(\#^n P) = n.
	\]
	Furthermore, for any $n \ge 1$ and $k \in \mathbb Z$, the manifold $Y_{n,k}$ given by $(-1)$-surgery on each component of the link $L_{n,k}$ of Figure \ref{Lnk} is a hyperbolic integer homology sphere with \[S(Y_{n,k}) = S_D(Y_{n,k}) = n.\]
\end{theorem-intro}
Auckly also asked whether the connected sum of $n$ nontrivial integer homology spheres has Dehn surgery number at least $n$ \cite[Problem 3.102A]{Kirby}. The first part of Theorem~\ref{intro1} shows this is at least true for the case that these integer homology spheres are $P$; see also Corollary \ref{surgery-bound-Brieskorn}.

We give two proofs of Theorem~\ref{intro1}, both using the invariant $q_3$ of integer homology spheres, introduced by Kim Fr{\o}yshov in \cite{Fr:q2}. The first proof, presented in Section~\ref{background-proof-1}, is a geodesic route to the result. The second proof is obtained by establishing more general properties of $q_3$. We make the observation that the integer-valued invariant $q_3$ can be generalized to $\Bbb F_2$-homology spheres, i.e. closed $3$-manifolds whose homology with coefficients in the field $\mathbb{F}_2 = \mathbb{Z}/2\mathbb{Z}$ is the same as that of $S^3$. More importantly, we establish the following properties of this extension.

\begin{theorem-intro}\label{main-thm}
	The invariant $q_3(Y)$ of $\Bbb F_2$-homology spheres satisfies the following properties. 
\begin{enumerate}[label=(\roman*), ref=(\roman*)]
	\item\label{main-thm-i} If $-Y$ denotes the orientation-reversal, we have $q_3(-Y) = -q_3(Y)$.
	\item\label{main-thm-ii} If $W:Y\to Y'$ is a cobordism between $\Bbb F_2$-homology spheres such that $H_1(W;\Bbb Z)$ has no 2-torsion, then 
	\begin{equation}\label{b+-ineq}
		-b^+(W) \le q_3(Y')-q_3(Y) \le b^-(W).
	\end{equation}	
	\item\label{main-thm-iii} If $Y$ is an integer homology sphere and $P$ is the Poincar\'e homology sphere, then  \[q_3(Y \# P) = q_3(Y) + 1.\]
\end{enumerate}
\end{theorem-intro}

It follows from \eqref{b+-ineq} that, like Fr{\o}yshov's $h$-invariant \cite{Froyshov} or the $d$-invariant in Heegaard Floer homology \cite{OzSz}, the invariant $q_3$ is nondecreasing under negative-definite cobordisms. In particular, $q_3$ is an $\mathbb F_2$-homology cobordism invariant. However, in contrast to the inequality \eqref{b+-ineq}, the difference $d(Y') - d(Y)$ can be an arbitrary even integer when $b^+(W) = b^-(W) = 1$; a similar statement is true for $h$.

\begin{remarkintro}
Both the invariant $q_3(Y)$ and the properties stated in Theorem~\ref{main-thm}\ref{main-thm-i}-\ref{main-thm-ii} admit generalizations to rational homology spheres. However, the invariant should no longer be understood as an integer, but rather an integer-valued function on nonzero elements of the group ring $\mathbb F_2[H^1(Y;\mathbb F_2)]$. The definition is given in Section \ref{subsec:QHS-Scpx} using the machinery of Section \ref{sec:equivariant-homology-groups}. The generalization of Theorem \ref{main-thm}\ref{main-thm-i} is Proposition \ref{prop:J-3mfd-duality}, and the generalization of Theorem \ref{main-thm}\ref{main-thm-ii} is Theorem \ref{thm:main-ineq-QHS}.
\end{remarkintro}

\begin{remarkintro}\label{rmk:additivity}
A statement more general than Theorem~\ref{main-thm}\ref{main-thm-iii} is true: if $Y, Y'$ are integer homology spheres and $Y'$ is homology cobordant to any connected sum of Seifert spaces and manifolds with Dehn surgery number one, then $q_3(Y \# Y') = q_3(Y) + q_3(Y')$. This generalization will appear in forthcoming joint work by the first and third authors and Christopher Scaduto \cite{DMES}, studying in detail the cellular instanton invariants discussed in this paper and proving a connected-sum theorem. It is not known whether $q_3$ is additive for arbitrary integer homology spheres.
\end{remarkintro}

Because the trace of an integral surgery presentation for $Y$ is a simply-connected $4$-manifold with boundary $Y$ and second Betti number equal to the number of components of the link, it follows immediately from \eqref{b+-ineq} that $|q_3(Y)|$ is a lower bound for the surgery number $S(Y)$. This estimate can be improved.

\begin{cor-intro}\label{surgery-number-bound}
	For any $\Bbb F_2$-homology sphere $Y$, we have
	\[  
	|q_3(Y)| \le S_D(Y).
	\]
\end{cor-intro}

In preparation for the proof, we make the following definitions. First, if $L \subseteq S^3$ is any rationally-framed link, its linking matrix $M_L$ has coefficients $a_{ij} = \text{link}(L_i, L_j)$ for $i \ne j$, and $a_{ii}$ equal to the rational framing of $L_i$.  Second, if $M$ is a symmetric matrix, $n_\pm(M)$ is the number of eigenvalues of $M$ of each sign.

\begin{proof}
	Suppose $L$ is a rationally-framed link whose Dehn surgery is equal to $Y$. We will construct a $2$-handle cobordism $W: S^3 \to Y$ for which $b^-(W) = n_-(M_L)$. Then \eqref{b+-ineq} implies that 
	\[
	q_3(Y) \leq b^-(W) = n_-(M_L) \leq \# \pi_0(L).
	\] 
	Minimizing over all rational surgery presentations for $Y$ gives $q_3(Y) \le S_D(Y)$. The same argument applied to $-Y$ yields the inequality $q_3(-Y) \leq S_D(-Y)$. Because $S_D(-Y) = S_D(Y)$, this inequality together with Theorem \ref{main-thm}\ref{main-thm-i} implies $-q_3(Y) \le S_D(Y)$.
	
	The cobordism $W$ is given by the trace of an integrally-framed link $L'$ obtained from $L$ by a sequence of reverse slam-dunk moves, so in particular $L$ and $L'$ have the same surgery. If $r$ is the surgery coefficient of a component $K$ of the link $L$, let 
	\[  
	r = a_1-\frac{1}{a_2-\dfrac{1}{a_3-\dfrac{1}{\cdots-\dfrac{1}{a_n}}}}
	\]
	be the Hirzebruch--Jung continued fraction for $r$, so $a_1 = \lfloor r\rfloor$ and $a_j \ge 2$ for $2\le j\le n$. Using a sequence of reverse slam-dunk moves, $r$-surgery along $K$ is equivalent to 
	integral surgery along $K$ and a chain of $n-1$ unknots linked with $K$, where the surgery coefficient of $K$ is $a_1$ and the surgery coefficient of the $j^{\rm th}$ unknot in 
	the chain is $a_{j+1}$. Adding such a chain of unknots for each component of $L$ and setting the surgery coefficients as above determines the integrally framed link $L'$. The $4$-manifold $W$, given by gluing 2-handles to $B^4$ along the components of $L'$ with specified framing, has intersection form equal to the linking matrix $M_{L'}$. It remains to show that the number of negative eigenvalues of $M_L$ and $M_{L'}$ agree, for which it suffices to show that this is the case when performing a single reverse slam-dunk move which introduces a meridian of positive framing.
	
	Suppose $J$ is a rationally-framed link whose component $J_1$ has framing $a=b-1/c$, with $b$ an integer and $c>0$. Then the reverse slam-dunk move produces a link $J'=J \cup \mu$, for which the framing of $\mu$ and $J_1$ are $c$ and $b$, respectively. Let $V_\pm$ be maximal definite subspaces for the linking form $M_J$. It is straightforward to compute that $V_-$ and $\text{span}(e_\mu) + V_+$ are complementary definite subspaces of $M_{J'}$, where $e_\mu$ is the basis vector corresponding to the component $\mu$ of $J'$, and in particular each is of maximal dimension. This implies that $n_-(M_J) = n_-(M_{J'})$, as desired.
\end{proof}

The preceding argument shows that $q_3(Y)$ bounds the following more refined quantity.

\begin{defnintro}
	The \textit{signed surgery numbers of $Y$} are defined as \[S_\pm(Y) = \min \{n_\pm(M_L) : L \text{ is a rationally-framed link with Dehn surgery } Y\}.\qedhere\]
\end{defnintro}

To be precise, the proof of Corollary \ref{surgery-number-bound} shows that the signed surgery numbers of $Y$ are the same if we instead minimize over the class of integrally-framed links. Therefore, the inequality \eqref{b+-ineq} implies the following result.

\begin{cor-intro}\label{cor:signed-surgery-bound}
	For any $\Bbb F_2$-homology sphere $Y$, we have 
	\[
	-S_+(Y) \le q_3(Y) \le S_-(Y).
	\]
\end{cor-intro}

The monotonicity of $q_3$ can be used to provide additional examples for which the Dehn surgery of an $n$-fold connected sum is at least $n$, which could be considered as evidence toward a positive resolution to Auckly's question \cite[Problem 3.102A]{Kirby}. 

\begin{cor-intro} \label{surgery-bound-Brieskorn}
	Suppose $Y = Y_1\#\cdots\#Y_n$, where each $Y_i$ is one of the Brieskorn homology spheres $\Sigma(p,q,pqk-1)$. Then $S_-(Y) = S_D(Y) = n$.
\end{cor-intro}
\begin{proof}
	We show that $n \le q_3(Y) \le S_-(Y) \le S_D(Y) \le n$. The second inequality is Corollary~\ref{cor:signed-surgery-bound}, the third inequality is obvious, and the fourth inequality follows because $\Sigma(p,q,pqk-1)$ is $-1/k$ surgery on the left-handed $(p,q)$-torus knot. 
	
	For each coprime pair $(p,q)$ and each integer $k \ge 1$ there exists a negative-definite simply-connected cobordism $W: P \to \Sigma(p,q,pqk-1)$; more details appear in Example \ref{ex:Brieskorn}. It follows that there is a negative-definite simply-connected cobordism $W: \#^n P \to Y$. By Theorem~\ref{intro1} and Theorem~\ref{main-thm}\ref{main-thm-ii}, we have $n \le q_3(Y)$.
\end{proof}

\begin{remarkintro}
	Together with the generalized additivity result discussed in Remark \ref{rmk:additivity}, a stronger result holds with a similar argument: if $Y = Y_1 \# \cdots \# Y_n$ is an integer homology sphere with $q_3(Y_i) = S_D(Y_i) = 1$ for each $i$, then $q_3(Y) = S_-(Y) = S_D(Y) = n$. It is established in \cite{GME} that there is a homomorphism $M: \mathcal C \to 4\mathbb Z$ on the concordance group such that $q_3(S^3_{-1/n}(K)) = 1$ for all $n > 0$ if and only if $M(K) < 0$. In some sense, half of all knots satisfy the stated condition. 
\end{remarkintro}

\begin{remarkintro}
	Any Seifert--fibered integer homology sphere $\Sigma$ satisfies $|q_3(\Sigma)| \le 1$. This is a straightforward consequence of the definition of $q_3$ and is established in Example \ref{ex:SFS-q3-small}. This implies that $q_3$ cannot be used to provide nontrivial lower bounds on the surgery number of Seifert-fibered spaces. In fact, no invariant satisfying the inequality \eqref{b+-ineq} can do so, because every Seifert-fibered space has $S_\pm(Y) \le 1$. To see that $S_-(Y) \le 1$, observe that every Seifert--fibered space can be written as surgery on a core circle with some integer framing $b$ and a collection of meridians framed by the Seifert data $b_i/a_i$, which can all be chosen positive. It is straightforward to verify that this link has $n_-(M_L) \le 1$. 
\end{remarkintro}

\subsubsection*{Outline} 
\indent In Section~\ref{background-proof-1}, we present the first proof of Theorem~\ref{intro1}. Section~\ref{Fro-approach} outlines Fr{\o}yshov's definition of $q_3(Y)$, and Section~\ref{sec:weak-ineq} establishes \eqref{b+-ineq} in the special case that one of $Y$ or $Y'$ is $S^3$, the other is an integer homology sphere, and $W$ is a $2$-handle cobordism. This is enough to prove Corollary~\ref{surgery-number-bound} in the case that $Y$ is an integer homology sphere. Section~\ref{connected-sum-P} establishes that $q_3(Y)$ is \textit{superadditive} with respect to connected sum with the Poincar\'e sphere, from which Corollary~\ref{surgery-bound-Brieskorn} follows. Section~\ref{sec:hyperbolic-example} completes the proof of Theorem~\ref{intro1}.

In Section~\ref{sec:S-complexes} we explain the language of $\mathcal S$-complexes, introduced in \cite{DS1}. These are used to reconceptualize the invariant $q_3(Y)$ and its basic properties.

In Section~\ref{cellular-model} we give the background needed to extend the definition of $q_3(Y)$ to rational homology spheres. The material here is based on part of the forthcoming book \cite{DMES}, but the relevant portions are presented in a self-contained manner. The definition of $q_3(Y)$ for rational homology spheres is given in Section~\ref{subsec:QHS-Scpx}.

In Section~\ref{sec:susp} we present the proof of \eqref{b+-ineq} in the full generality of rational homology spheres, which is used to complete the proof of Theorem~\ref{main-thm}\ref{main-thm-iii}. The proof presented in Section~\ref{sec:susp} is different in character from that of Section~\ref{background-proof-1}, which is based on the surgery exact triangle in instanton Floer homology. Instead, we present a general construction of cobordism maps for cobordisms with $b^+(W) > 0$, using the technique of \textit{suspension} developed in \cite{DME1}. 

\subsubsection*{History}
In 2018, Kim Fr{\o}yshov gave several talks on the invariants $q_2(Y)$ and $q_3(Y)$. Around the same time, the first and third authors developed the notion of \emph{abelian suspension} \cite{DME1} in work to prove the invariance of equivariant instanton homology. This led to the notion of \emph{central suspension} discussed in Section~\ref{sec:susp}, which can be used to prove \eqref{b+-ineq} and resolve the surgery number problem, announced in late 2022. At the same time, they collaborated with Chris Scaduto on the cellular model of Section~\ref{cellular-model}, and the connected-sum theorem used to prove the result in Remark~\ref{rmk:additivity}. Dissemination of this work was delayed until the completion of \cite{DMES}; the first and third authors regret the lengthy wait.

During this period, Fr{\o}yshov's work appeared as \cite{Fr:q2}. Studying this preprint led the second author to an independent proof of Theorem~\ref{intro1}, also using $q_3$. The two approaches offer complementary advantages. The argument of the first and third authors leads to the generalization to rational homology spheres, the inequality \eqref{b+-ineq}, and the general additivity result of Remark~\ref{rmk:additivity}. The second author's approach is more direct and substantially simplifies the proof of the additivity in Theorem~\ref{main-thm}. The authors became aware of each other's work in late 2025.

We have decided to present a variation on both arguments in this paper, in the hope that both are of independent value. This paper would not exist in its present form without the equal contribution of all three authors, nor without the foundational joint work with Chris Scaduto.\\

\noindent\textit{Acknowledgements.} The authors appreciate Kim Fr\o yshov's responses to emails on this subject, in which he expressed that he anticipated the inequality \eqref{b+-ineq} for $Y = S^3$, as well as an additivity result similar to the one mentioned in Remark~\ref{rmk:additivity}. The authors thank Kyle Hayden for suggesting a construction of integer homology spheres similar to the family $Y_{n,k}$. XL would like to thank his advisor Yi Xie for introducing him to the world of gauge theory. AD and MME thank Tye Lidman for bringing this problem to their attention. AD and MME also thank Chris Scaduto for many conversations, which have been influential on their understanding of this subject. The authors thank Chris Scaduto for comments on a draft of the paper. 


\section{Surgery number bounds for integer homology spheres}\label{background-proof-1}

The invariant $q_3$ is defined using instanton homology with coefficients in the field $\F$. There are at least two approaches to setting up this theory. In Section \ref{Fro-approach}, we review Fr{\o}yshov's perspective, as presented in \cite{Fr:q2}, which is restricted to integer homology spheres. 
Section \ref{sec:weak-ineq} is devoted to the proof of \eqref{b+-ineq} for a cobordism $W:Y\to Y'$ under a restrictive assumption on $Y$, which is satisfied when $Y=S^3$. In particular, this version of \eqref{b+-ineq} is still strong enough to prove the surgery number bound in Corollary \ref{surgery-number-bound}. The additivity formula Theorem \ref{main-thm}\ref{main-thm-iii} can be separated into two inequalities, and in Section \ref{connected-sum-P} we establish one of the two inequalities. Using these ingredients, in Section \ref{sec:hyperbolic-example} we complete the proof of Theorem \ref{intro1}.

\subsection{Mod 2 instanton homology \`a la Fr{\o}yshov}\label{Fro-approach}
Let $Y$ be a connected oriented closed 3-manifold. The {\it Chern--Simons functional} of any connection $A$ on the trivial $SU(2)$-bundle over $Y$ is defined as 
\vspace{-5pt}
\begin{equation} \label{CS}
	CS(A)=-\frac{1}{8\pi^2}\int_Y \tr(A\wedge dA+\frac{2}{3}A\wedge A\wedge A).
\end{equation} 
Any map $u:Y\to SU(2)$ determines an automorphism of the trivial $SU(2)$-bundle on $Y$, and we may pull back $A$ to obtain another connection $u^*A$. The difference between values of the Chern--Simons functional at $u^*A$ and $A$ is equal to the degree of $u$ as a map between closed oriented 3-manifolds. In particular, if $\mathcal A(Y)$ denotes the space of all $SU(2)$ connections on $Y$ and $\mathcal G(Y)$ denotes the gauge group of maps $u:Y\to SU(2)$, then $CS$ induces an $\mathbb R/\mathbb Z$-valued functional on the configuration space ${\mathcal B}(Y) = {\mathcal A}(Y)/{\mathcal G}(Y)$, which is still denoted by $CS$. 

For an integer homology sphere $Y$, the instanton homology group of $Y$ is defined by applying Morse homological methods to this functional. For any such $Y$, the trivial connection always represents a critical point of $CS:{\mathcal B}(Y) \to \mathbb R/\mathbb Z$ and any other critical point is irreducible. A similar claim holds even after a small perturbation of $CS$; we assume that a perturbation is fixed such that all critical points of the perturbed Chern--Simons functional are nondegenerate. We write $\mathfrak C(Y)$ (resp. $\mathfrak C^{\rm irr}(Y)$) for the set of critical points (resp. nontrivial critical points) of the perturbed Chern--Simons functional. We also write $\theta$ for the element of $\mathfrak C(Y)$ represented by the trivial connection. 

After fixing a Riemannian metric on $Y$, we may form the gradient of $CS$ with respect to the induced metric on $\mathcal B(Y)$. The gradient flow equation of the perturbed Chern--Simons functional can be identified with a perturbation of the ASD equation for connections on the cylinder $\Bbb R\times Y$. Recall that for $SU(2)$-connections on a Riemannian $4$-manifold, the ASD equation is defined by $* F_A = -F_A$; that is, the Hodge star should take the curvature to its negative. Here the ASD equation is defined with respect to the product metric on $\Bbb R\times Y$ induced by the fixed metric on $Y$ and the standard metric on $\mathbb R$. For any pair $\alpha$, $\beta\in \mathfrak C(Y)$, we write $M(Y;\alpha,\beta)$ for the gauge equivalence classes of all solutions of this perturbed ASD equation that are asymptotic to $\alpha$ and $\beta$ on the incoming and the outgoing ends of the cylinder $\Bbb R\times Y$. We choose the perturbation of the Chern--Simons functional such that the moduli spaces $M(Y;\alpha,\beta)$ are cut out transversely. Translation of any element of $M(Y;\alpha,\beta)$ in the $\Bbb R$ direction represents another element of $M(Y;\alpha,\beta)$, and we write $\breve M(Y;\alpha,\beta)$ for the quotient of the free orbits of this action. This action is free at an element of $M(Y;\alpha,\beta)$ unless $\alpha=\beta$ and this element of the moduli space is the pullback of $\alpha=\beta$ to the cylinder. There is a $\Bbb Z/8\Bbb Z$-grading on  $\mathfrak C(Y)$ that determines the mod $8$ values of the dimension of the moduli spaces $\breve M(Y;\alpha,\beta)$:
\begin{equation}\label{dim-mod-8}
  \dim\left(\breve M(Y;\alpha,\beta)\right)\equiv \left\{\begin{array}{ll}|\alpha|-|\beta|-1 & \text{$\alpha$ is irreducible}\\
  |\alpha|-|\beta|-4 & \alpha=\theta \end{array}\right. \hspace{1cm} \mod 8.
\end{equation}
Here $|\alpha|$ denotes the grading of $\alpha$. We write $\breve M(Y;\alpha,\beta)_d$ for the $d$-dimensional component of $\breve M(Y;\alpha,\beta)$, which is nonempty only if the mod $8$ value of $d$ satisfies the relation in \eqref{dim-mod-8}. A similar convention is used for the parameterized moduli spaces $M(Y; \alpha, \beta)_d$, such that $M(Y; \alpha, \beta)_d/\mathbb R = \breve M(Y; \alpha, \beta)_{d-1}$ for $d > 0$.

The instanton homology group $I(Y)$ is the homology of a chain complex $(C(Y),d_1)$, where $C(Y)$ is the vector space over $\F$ generated by $\mathfrak C^{\rm irr}(Y)$, and the differential $d_1$ acts on a generator $\alpha$ of $C(Y)$ as 
\begin{equation}\label{diff}
  d_1\alpha=\sum_{\beta\in \mathfrak C^{\rm irr}(Y)}\#\breve M(Y;\alpha,\beta)_0 \cdot \beta,
\end{equation}
where the sum is over all $\beta\in \mathfrak C^{\rm irr}(Y)$ with $|\beta|=|\alpha|-1$, and $\#\breve M(Y;\alpha,\beta)_0$ is the number of the elements of the $0$-dimensional space $\breve M(Y;\alpha,\beta)_0$. This defines a differential of degree $-1$, i.e., $d_1$ maps the summand $C_i(Y)$ of $C(Y)$ generated by the elements of degree $i$ in $\mathfrak C^{\rm irr}(Y)$ to $C_{i-1}(Y)$. In particular, we adopt the homological convention for instanton homology, which differs from that in \cite{Fr:q2}, where the cohomological convention is used. We use the convention that an operator with subscript $k$ has degree $-k$. 

There are two homomorphisms that are defined using moduli spaces involving the trivial connection. Define
\[
  \delta_1:C(Y) \to \F,\hspace{1cm}\delta_4: \F \to C(Y),
\]
as
\[
  \delta_1(\alpha)=\#\breve M(Y;\alpha,\theta)_0,\hspace{1cm}\delta_4(1)=\sum_{\beta}\#\breve M(Y;\theta,\beta)_0\cdot \beta
\]
These operators satisfy $\delta_1d_1=0$ and $d_1\delta_4=0$. In particular, they induce homomorphisms $\delta_1:I(Y)\to \F$, which is nontrivial only in degree $1$, and $\delta_4:\F\to I(Y)$, which takes values in degree $4$.

There are several other operators acting on $C(Y)$ that are defined using the {\it basepoint fibration}. Given a basepoint $y\in Y$, we obtain an action of $\mathcal G(Y)$ on the Lie algebra $\su(2)$ of $SU(2)$ by first restricting a gauge transformation $u:Y\to SU(2)$ to $y$ and then applying the adjoint action of $SU(2)$ on its Lie algebra. Let $\mathcal A^*(Y)$ be the subspace of $\mathcal A(Y)$ given by irreducible connections. The group $\mathcal G(Y)$ acts on the trivial vector bundle $\su(2)\times \mathcal A^*(Y)$ over $\mathcal A^*(Y)$ and the stabilizer of any point in this bundle or its base is $\{\pm 1\}$. Passing to the quotient determines a rank three vector bundle $\mathbb E$ over $\mathcal B^*(Y)$, the quotient of $\mathcal A^*(Y)$ with respect to the action of $\mathcal G(Y)$. 
A chain level version of taking cap product with cohomology classes $w_2(\mathbb E)$, $w_3(\mathbb E)$ of $\mathcal B^*(Y)$ gives rise to the operators 
\[
  d_2:C(Y) \to C(Y),\hspace{1cm} d_3:C(Y) \to C(Y).
\] 
More precisely, fix generic sections $s_0$, $s_1$ and $s_2$ of $\mathbb E$, and for any $\alpha,\beta\in \mathfrak C^{\rm irr}(Y)$ form the following subspaces of $M(Y;\alpha,\beta)$:
\begin{align*}
  M^{(2)}(Y;\alpha,\beta)_d&=\{[A] \in M(Y;\alpha,\beta)_{d+2} \mid s_1(A|_{\{0\}\times Y})\wedge s_2(A|_{\{0\}\times Y})=0\},\\
  M^{(3)}(Y;\alpha,\beta)_d&=\{[A] \in M(Y;\alpha,\beta)_{d+3} \mid s_0(A|_{\{0\}\times Y})=0\},
\end{align*}
where the notation $v\wedge w=0$ means that $v$ and $w$ are linearly dependent vectors of a vector space.
For $i=2, 3$, let $d_i:C(Y) \to C(Y)$ be the linear operator defined as
\begin{equation}\label{def-di}
  d_i\alpha=\sum_{\beta\in \mathfrak C^{\rm irr}(Y)}\# M^{(i)}(Y;\alpha,\beta)_0 \cdot \beta
\end{equation}
These operators satisfy
\begin{equation}\label{di-rel}
	d_id_1+d_1d_i=0,
\end{equation}
as established in \cite[Proposition 6.1]{Fr:q2}. This relation follows from analyzing the ends of the $1$-dimensional manifold $M^{(i)}(Y;\alpha,\beta)_1$. In fact, this space can be compactified into a manifold with boundary components
\[
  \breve M(Y;\alpha,\gamma)_0\times M^{(i)}(Y;\gamma,\beta)_0, \hspace{1cm}  M^{(i)}(Y;\alpha,\gamma)_0\times \breve M(Y;\gamma,\beta)_0.
\]
The number of boundary components of the first type (resp. second type) is equal to the coefficient of $\beta$ in $d_id_1(\alpha)$ (resp. $d_1d_i(\alpha)$). This implies \eqref{di-rel}. The chain map relation \eqref{di-rel} implies that $d_i$ induces an operator acting on $I(Y)$, which is denoted by $u_i$ following \cite{Fr:q2}. 

There is yet another operator $d_4:C(Y) \to C(Y)$ that measures the extent to which the operators $u_2$ and $u_3$ commute. First define 
\begin{align*}
	M^{(2,3)}(Y;\alpha,\beta)_d=\{(t,[A]) \in \mathbb R\times M(Y;\alpha,\beta)_{d+4}&\mid  s_0(A|_{\{-t\}\times Y}) = 0,\\
	&\,\,\,s_1(A|_{\{t\}\times Y})\wedge s_2(A|_{\{t\}\times Y})=0\}.
\end{align*}
The operator $d_4$ is obtained by replacing $M^{(i)}(Y;\alpha,\beta)_0$ with $M^{(2,3)}(Y;\alpha,\beta)_0$ in \eqref{def-di}. This operator satisfies  
\begin{equation}\label{d4-rel}
  d_1d_4+d_4d_1+d_2d_3+d_3d_2+\delta_4\delta_1=0,
\end{equation}
as established in \cite[Proposition 6.7]{Fr:q2}. In particular, this implies that the operators $u_2$, $u_3$ commute up to the term given by $\delta_4\delta_1$. As with the previous relations, \eqref{d4-rel} is proved by analyzing the ends of a 1-dimensional moduli space. The 1-dimensional moduli space $M^{(2,3)}(Y;\alpha,\beta)_1$ has ends of the form 
\[
  \breve M(Y;\alpha,\gamma)_0\times M^{(2,3)}(Y;\gamma,\beta)_0, \hspace{1cm}  M^{(2,3)}(Y;\alpha,\gamma)_0\times \breve M(Y;\gamma,\beta)_0,
\]
giving the terms $d_1d_4$ and $d_4d_1$, the ends of the form 
\[
  M^{(2)}(Y;\alpha,\gamma)_0\times M^{(3)}(Y;\gamma,\beta)_0, \hspace{1cm}  M^{(3)}(Y;\alpha,\gamma)_0\times M^{(2)}(Y;\gamma,\beta)_0,
\]
giving the terms $d_2d_3$ and $d_3d_2$. The final term of \eqref{d4-rel} is somewhat less transparent. The remaining ends of $\mathbb R \times M(Y; \alpha, \beta)_5$ can be identified with
\[\mathbb R^2 \times SO(3)\times\breve M(Y;\alpha,\theta)_0\times \breve M(Y;\theta,\beta)_0.\]
A model computation for $\mathbb R^2 \times SO(3)$ \cite[Claim 11.1]{Fr:q2} shows that there is an odd number of points for which $s_0 = 0$ and $s_1 \wedge s_2 = 0$, such that the corresponding end of $M^{(2,3)}(Y; \alpha, \beta)_1$ has the same number of points modulo $2$ as $\breve M(Y;\alpha,\theta)_0\times \breve M(Y;\theta,\beta)_0$, giving the term $\delta_4 \delta_1$. 

Now we are ready to define the integer valued invariant $q_3(Y)$, which will play the central role in this paper. 
\begin{definition}If $Y$ is an integer homology sphere, $q_3(Y)$ is the unique integer satisfying the following. For a positive integer $k$, we have $q_3(Y)\geq k$ if and only if there are $\zeta, \eta \in C(Y)$ such that 
\begin{equation}\label{q3-pos-def-ch-cx}
  d_1\zeta=0, \hspace{.5cm} d_2\zeta=d_1\eta, \hspace{.5 cm}\delta_1d_3^{k-1}\zeta \neq 0,\hspace{.5cm } \delta_1d_3^i\zeta = 0 \ \ \text{for }0\le i\le k-2, 
\end{equation}
and for a nonpositive integer $k$, we have $q_3(Y)\geq k$ if and only if there are $a_0, a_1, \dots, a_{-k}\in\mathbb F_2$ and $\zeta, \eta \in C(Y)$ satisfying 
\begin{equation}\label{q3-neg-def-ch-cx}
  \sum_{i=0}^{-k}d_3^i\delta_4(a_i) = d_1\zeta+d_2\eta,\hspace{.5 cm} d_1 \eta = 0, \hspace{.5cm} a_{-k}\neq 0.
\end{equation}
It is straightforward to check that these inequalities are compatible. 
\end{definition}
Since $C(Y)$ is finite dimensional, $q_3(Y)$ is a finite integer, determined uniquely by these inequalities. 
The relations in \eqref{q3-pos-def-ch-cx} are equivalent to the existence of $[\zeta]\in I(Y)$ such that 
\[
  u_2([\zeta])=0, \hspace{.5cm} \delta_1u_3^{k-1}([\zeta]) \neq 0,\hspace{.5cm } \delta_1u_3^i([\zeta]) = 0 \ \ \text{for }0\le i\le k-2, 
\]
and the relations in \eqref{q3-neg-def-ch-cx} are equivalent to the existence of $a_0, a_1, \dots, a_{-k}\in\mathbb F_2$ satisfying
\[
  \sum_{i=0}^{-k}u_3^i\delta_4(a_i) \in \im(u_2), \hspace{0.5cm} a_{-k}\neq 0.
\]
A slightly more conceptual treatment of $q_3$ will be given in Section \ref{sec:equivariant-homology-groups}. Switching the roles of $d_2$ and $d_3$ in the definition of $q_3$ gives the definition of the invariant $q_2$, which is studied extensively in \cite{Fr:q2}. 

\begin{example}\label{ex:SFS-q3-small}
		If $\Sigma$ is a Seifert-fibered integer homology sphere, its Chern--Simons functional is Morse--Bott, and its critical manifolds carry Morse functions with only even critical points \cite{SFS-1, SFS-2}. Therefore, the instanton Floer homology $I_k(Y)$ is supported in only one degree mod $2$. Because $u_3$ has odd degree, the characterization of $q_3$ given above implies $|q_3(\Sigma)| \le 1$. 
\end{example}

It is straightforward to check how changing the orientation of $Y$ affects the instanton Floer homology of $Y$. The chain group $C_i(-Y)$ is isomorphic to \[C_{i}(Y)^{\dagger}={\rm Hom}(C_{-i-3}(Y), \F).\]
With respect to this identification, the operator $d_i$ for $-Y$ is equal to the adjoint of the operator $d_i$ for $Y$, the operator $\delta_1$ for $-Y$ is equal to the adjoint of the operator $\delta_4$ for $Y$, and the operator $\delta_4$ for $-Y$ is equal to the adjoint of the operator $\delta_1$ for $Y$. It follows that $q_3(-Y) = -q_3(Y)$. 

\begin{example}\label{ex:Poincare}
	The 3-dimensional sphere $S^3$ does not admit any irreducible flat $SU(2)$ connection, so $I(S^3)$ is trivial and $q_3(S^3) = 0$.
	The Poincar\'e homology sphere $P$ has two nondegenerate irreducible flat $SU(2)$-connections $\beta_0$, $\beta_1$ with gradings $1$ and $5$. 
	In particular, the operators $d_1$, $d_2$ and $d_3$ are trivial. Furthermore, we have $\delta_1(\beta_0)=1$. It follows that $q_3(P)=1$, and the above discussion implies that $q_3(-P)=-1$.
\end{example}

It is observed in \cite[Proposition 6.5]{Fr:q2} that $u_2$ and $u_3$ are both nilpotent operators acting on $I(Y)$. In particular, if $n_3(Y)$ is the nilpotency degree of $u_3$, then it is clear from the definition of $q_3$ that
\begin{equation}\label{nilp-q3-ineq}
  -n_3(Y)\leq q_3(Y)\leq n_3(Y).
\end{equation}
Here we use the convention that $n_3(Y)=0$ if $I(Y)$ is trivial.
\begin{definition}\label{q3-minimal}
	An integer homology sphere $Y$ is {\it $q_3$-minimal} (resp. {\it $q_3$-maximal}) if $q_3(Y)=-n_3(Y)$ (resp. $q_3(Y)=n_3(Y)$).
\end{definition} 
\noindent
This definition will be important in the next section. Clearly $Y$ is $q_3$-minimal only if $q_3(Y)\le 0$ and $q_3$-maximal only if $q_3(Y)\geq 0$. An integer homology sphere $Y$ is $q_3$-minimal if and only if $-Y$ is $q_3$-maximal. The Poincar\'e sphere is $q_3$-maximal, and $-P$ is $q_3$-minimal, while $Y$ is both $q_3$-maximal and $q_3$-minimal if and only if $I(Y) = 0$. This holds for $Y = S^3$, and no other homology sphere is expected to have this property; compare \cite[Problem 3.50]{K3}.

Instanton Floer homology $I(Y)$ is functorial with respect to cobordisms and this functoriality is compatible with the operators $u_2$ and $u_3$. We review this additional structure in the special case of cobordisms $W:Y\to Y'$ with $b^+(W) = 0$ and $H_1(W;\mathbb Z) = 0$, in which some of the subtleties that arise in the more general setting do not appear. First fix Riemannian metrics on $Y$, $Y'$, perturbations of the associated Chern--Simons functionals, and sections of the associated basepoint fibrations. We may then form the instanton Floer complexes $(C(Y),d_1)$ and $(C(Y'),d_1')$ and the morphisms $d_i$, $\delta_1$ and $\delta_4$ for $Y$ and the morphisms $d_i'$, $\delta_1'$ and $\delta_4'$ for $Y'$.

Next, fix a Riemannian metric on $W$ which is the product metric in a tubular neighborhood of the boundary components of $W$ corresponding to fixed metrics on $Y$, $Y'$. Add cylindrical ends $(-\infty,-1]\times Y$ and $[1,\infty)\times Y'$ with the product metrics to obtain a noncompact complete Riemannian manifold denoted by $W^+$. For any $\alpha \in \mathfrak C(Y)$ and $\alpha' \in \mathfrak C(Y')$, write $M(W; \alpha, \alpha')$ for the set of gauge equivalence classes of solutions to a perturbed ASD equation on $W^+$ that are asymptotic to $\alpha$ and $\alpha'$ on the incoming and outgoing ends, respectively. We assume that this perturbed ASD equation on $W^+$ restricts on the ends to the perturbed ASD equations on the cylinders associated with $Y$, $Y'$, which are used in forming the chain complexes $(C(Y),d_1)$ and $(C(Y'),d_1')$. The dimensions of these moduli spaces associated with $W$ satisfy \[\dim M(W; \alpha, \alpha') \equiv \begin{cases} |\alpha| - |\alpha'| & \alpha \text{ is irreducible} \\ -|\alpha'| - 3 & \alpha = \theta \end{cases} \mod 8.\]
As in the case of cylinders, we write $M(W; \alpha, \alpha')_d$ for the $d$-dimensional component of $M(W; \alpha, \alpha') $.

We define a degree zero map $\lambda_0: C(Y) \to C(Y')$ by the formula \[\lambda_0(\alpha) = \sum_{\alpha' \in \mathfrak C^{\text{irr}}(Y')} \# M(W;\alpha, \alpha')_0 \cdot \alpha'.\] We may also define the operators 
\[
 \Delta_0:C_0(Y) \to \F,\hspace{1cm}\Delta_3: \F \to C_{-3}(Y'),
\]
as
\[
  \Delta_0(\alpha)=\# M(W;\alpha,\theta')_0,\hspace{1cm}\Delta_3(1)=\sum_{\alpha'}\#M(W;\theta,\alpha')_0\cdot \alpha'
\]
where $\theta'$ denotes the trivial connection on $Y'$.

To define the cobordism version of the operators $d_i$ for $i\geq 2$, we need to extend the definition of the basepoint fibration to the case of cobordisms. For $\alpha \in \mathfrak C^{\rm irr}(Y)$ and $\alpha' \in \mathfrak C^{\rm irr}(Y')$, the moduli space $M(W; \alpha, \alpha')$ is a subspace of $\mathcal B(W; \alpha, \alpha')=\mathcal A(W; \alpha, \alpha')/\mathcal G(W)$ with $\mathcal A(W; \alpha, \alpha')$ being the space of connections asymptotic to (fixed representatives of) $\alpha$, $\alpha'$ on the incoming and the outgoing ends, and $\mathcal G(W)$ is the space of maps $u:W^+\to SU(2)$ that are asymptotic to $\{\pm 1\}$ on the ends. In fact, a standard unique continuation result implies that $M(W; \alpha, \alpha')$ is contained in the subspace $\mathcal B^*(W; \alpha, \alpha')$ of $\mathcal B(W; \alpha, \alpha')$ consisting of those connections $A\in \mathcal A(W; \alpha, \alpha')$ whose restrictions to $\{-t\}\times Y$ and $\{t\}\times Y'$ are irreducible for any $t\in [1,\infty)$. In particular, there are well-defined restriction maps 
\begin{equation}\label{restriction-maps}
	(-\infty,-1]\times \mathcal B^*(W; \alpha, \alpha')\to \mathcal B^*(Y), \hspace{1cm}[1,\infty) \times \mathcal B^*(W; \alpha, \alpha')\to \mathcal B^*(Y') 
\end{equation}
given by mapping $(t,[A])\in (-\infty, -1] \times \mathcal B^*(W; \alpha, \alpha')$ to the restriction of $[A]$ to $\{t\}\times Y$ and $(t,[A]) \in [1,\infty) \times \mathcal B^*(W; \alpha, \alpha')$ to the restriction of $[A]$ to $\{t\}\times Y'$.

Now let $\gamma:[-1,1]\to W$ be a path connecting the basepoints $y\in Y$, $y'\in Y'$. Extend $\gamma$ to a path $\gamma^+:\mathbb R\to W^+$ by setting $\gamma^+(t) = (t, y)$ for $t \le -1$ and $\gamma^+(t) = (t, y')$ for $t \ge 1$. As in the 3-dimensional case, define a rank three vector bundle $\mathbb E(\gamma)$ over $\mathbb R \times \mathcal B^*(W; \alpha, \alpha')$ as the quotient of $\su(2)\times \mathbb R \times  \mathcal A(W; \alpha, \alpha')$ with respect to the action of $\mathcal G(W)$ defined by
\[
  u\cdot (\zeta,t,A)=(\text{ad}_{u(\gamma^+(t))}(\zeta), t,(u^{-1})^*A)
\]
with $u:W^+\to SU(2)$ and $(\zeta,t,A)\in \su(2)\times \mathbb R\times \mathcal A(W; \alpha, \alpha')$. The restrictions of $\mathbb E(\gamma)$ to $(-\infty,-1]\times \mathcal B^*(W; \alpha, \alpha')$ and $[1,\infty) \times  \mathcal B^*(W; \alpha, \alpha')$ are identified with the pullbacks of the basepoint fibrations over $\mathcal B^*(Y)$ and $\mathcal B^*(Y')$ with respect to the maps in \eqref{restriction-maps}. Fix generic sections $\sigma_i$ of $\mathbb E(\gamma)$ for $i \in \{0,1,2\}$ whose restrictions to  $(-\infty,-1]\times \mathcal B^*(W; \alpha, \alpha')$ and $[1,\infty) \times  \mathcal B^*(W; \alpha, \alpha')$ are the pullbacks of $s_i$ and $s_i'$ with respect to the maps in \eqref{restriction-maps}.

We then define the cutout moduli spaces 
\begin{align*}M^{(2)}(W; \alpha, \alpha')_d &= \{(t,[A]) \in \mathbb R \times M(W; \alpha, \alpha')_{d+1} \mid \sigma_1(t,A) \wedge \sigma_2(t,A) = 0\}, \\
M^{(3)}(W; \alpha, \alpha')_{d} &= \{(t,[A]) \in \mathbb R \times M(W; \alpha, \alpha')_{d+2} \mid \sigma_0(t,A) = 0\},\\
	M^{(2,3)}(W;\alpha,\alpha')_d&=\{(t,t',[A]) \in \mathbb R \times \mathbb R \times M(W;\alpha,\alpha')_{d+3} \mid  \sigma_0(t,[A])=0,\\
	&\hspace{6cm} \,\,\,\sigma_1(t',[A])\wedge \sigma_2(t',[A])=0\}.
\end{align*}
We now define operators $\lambda_i$ for $i \in \{1,2\}$ by \[\lambda_i(\alpha) = \sum_{\alpha' \in \mathfrak C^{\text{irr}}(Y')} \# M^{(i+1)}(W;\alpha, \alpha')_0 \cdot \alpha'.\] 
Similarly, we define $\lambda_{3}$ by replacing $M^{(i+1)}(W;\alpha, \alpha')_0$ in the above expression with $M^{(2,3)}(W;\alpha,\alpha')_0$.

The same analysis which establishes that $d_1^2, \delta_1 d_1, d_1 \delta_4$ are all zero and verifies equations \eqref{di-rel}-\eqref{d4-rel} now gives the following relations: 
\begin{align}
	d_1'\lambda_0 + \lambda_0 d_1&= 0 \label{lambda0-bdry-relation}\\
	\delta_1'\lambda_0+\Delta_0 d_1 +\delta_1  &=0  \label{Delta-0-bdry-relation}\\
	\lambda_0 \delta_4+d_1\Delta_3+\delta_4' & =0  \label{Delta-3-bdry-relation}\\
	d_1'\lambda_1 + \lambda_1 d_1+d_2'\lambda_0 + \lambda_0 d_2&= 0 \label{lambda1-bdry-relation}\\
	d_1'\lambda_2 + \lambda_2 d_1+d_3'\lambda_0 + \lambda_0 d_3&= 0 \label{lambda2-bdry-relation}\\
	d_1'\lambda_3 + \lambda_3 d_1+d_2'\lambda_2 + \lambda_2 d_2+d_3'\lambda_1 + \lambda_1 d_3+d_4'\lambda_0 + \lambda_0 d_4+\delta'_4\Delta_0+\Delta_3\delta_1&= 0 \label{lambda3-bdry-relation}
\end{align}

We remark that \eqref{Delta-0-bdry-relation} is verified by studying the boundary of $1$-dimensional moduli spaces $M(W;\alpha,\theta')_1$ with $\alpha\in \mathfrak C^{\rm irr}(Y)$, and the boundary points obtained by gluing an element of $\breve M(Y;\alpha,\theta)$ to the trivial connection on $W^+$, which is cut out transversely, are responsible for the term $\delta_1$ in this identity. A similar comment applies to \eqref{Delta-3-bdry-relation}. See \cite[Section 6]{Fr:q2} for more details. 

Functoriality of instanton Floer homology with respect to negative-definite cobordisms has the following {\it monotonicity} consequence for $q_3$, which is a special case of inequality \eqref{b+-ineq} in Theorem \ref{main-thm}. A similar inequality holds for Fr{\o}yshov's monopole and instanton $h$-invariants \cite{Froy:SW1,Froyshov}, the Heegaard Floer $d$-invariant \cite{OzSz}, and the invariant $q_2$ \cite{Fr:q2}. 

\begin{prop}\label{lemma:monotonic}
Suppose that $W: Y \to Y'$ is a cobordism between integer homology spheres with $b^+(W) = 0$ and $H_1(W;\mathbb Z) = 0$. Then $q_3(Y) \le q_3(Y')$.
\end{prop}

\begin{proof}
Suppose $\zeta,\eta \in C(Y)$ are as in \eqref{q3-pos-def-ch-cx} or \eqref{q3-neg-def-ch-cx}, depending on whether $q_3(Y)$ is positive, and define
\[
  \zeta'=\lambda_0(\zeta),\,\eta'=\lambda_1(\zeta)+\lambda_0(\eta)\in C(Y')
\]
Then it is straightforward to see that the identities in \eqref{lambda0-bdry-relation}-\eqref{lambda2-bdry-relation} imply that $\zeta'$ and $\eta'$ satisfy the same properties as in \eqref{q3-pos-def-ch-cx} or \eqref{q3-neg-def-ch-cx}, respectively.
\end{proof}

\begin{cor}
The integer $q_3(Y)$ is a $\mathbb Z$-homology cobordism invariant of $Y$. 
\end{cor}	

\begin{proof}
Let $W: Y \to Y'$ be an integer homology cobordism. Applying Proposition \ref{lemma:monotonic} to $W$ gives the inequality $q_3(Y) \le q_3(Y')$, while applying it to the reversed cobordism $\overline W: Y' \to Y$ gives $q_3(Y') \le q_3(Y)$. 
\end{proof}
\begin{remark}
	Strictly speaking, prior to this corollary, we had not yet shown that $q_3(Y)$ was a topological invariant of $Y$ independent of the auxiliary data involved in the construction of instanton homology. This follows because Proposition \ref{lemma:monotonic} holds regardless of the choice of auxiliary data for $Y$ and $Y'$. 
\end{remark}

\begin{example}\label{ex:Brieskorn}
	Let $Y = \Sigma(p,q,pqk-1)$ be the $-1/k$ surgery on a nontrivial left-handed torus knot $T_{p,q}$. By Example \ref{ex:SFS-q3-small}, we have $q_3(Y) \le 1$. We will argue that there is a negative-definite $2$-handle cobordism $W: P \to Y$, so by Example \ref{ex:Poincare} we have $1 = q_3(P) \le q_3(Y)$ as well. The knot $T_{p,q}$ can be taken to the left-handed trefoil $T_{2,3}$ by changing a sequence of negative crossings; running in reverse, $T_{2,3}$ can be taken to $T_{p,q}$ by changing a sequence of positive crossings. These crossing changes can be effected by performing $(-1)$-surgery along a crossing circle. Furthermore, for any knot $K$, the surgery $S^3_{-1/(k+1)}(K)$ can be obtained by performing $(-1)$-surgery on $S^3_{-1/k}(K)$. Thus, the torus knot $T_{2,3}$ with framing $-1$ can be taken to the torus knot $T_{p,q}$ with framing $-1/k$ by a sequence of $(-1)$-surgeries. The trace of this sequence of surgeries gives the desired cobordism $P \to Y$.
\end{example}
Next, we will discuss the extension of instanton homology to another family of 3-manifolds. A pair $(Y,w)$ of a 3-manifold $Y$ and $w\in H^2(Y;\mathbb Z)$ is admissible if there is an oriented embedded surface $S$ in $Y$ such that the pairing of $w$ and the homology class of $S$ is odd. There is a unique isomorphism class of a $U(2)$-bundle $E$ on $Y$ with $c_1(E)=w$. We may form $\mathcal A(Y,w)$ as the space of all connections on $E$ which induce a fixed connection on $\Lambda^2E$. There is an action on $\mathcal A(Y,w)$ by the gauge group $\mathcal G(Y,w)$ of automorphisms of $E$ with fiberwise determinant one, and we write $\mathcal B(Y,w)$ for the quotient. We may define again a Chern--Simons functional on $\mathcal B(Y,w)$, and the admissibility assumption on $(Y,w)$ implies that all of its critical points are irreducible. As in the case of integer homology spheres, and in fact somewhat simpler due to the irreducibility of critical points of the Chern--Simons functional, we may define the instanton chain complex $(C(Y,w),d_1)$ together with the operators $d_i$ for $2\leq i \leq 4$ associated with the basepoint fibration on $\mathcal B^*(Y,w)$. Then $d_1$ is a differential, \eqref{di-rel} holds for $i=2$ or $3$, and the following simpler version of \eqref{d4-rel} holds:
\begin{equation}\label{d4-rel-amidssible}
  d_1d_4+d_4d_1+d_2d_3+d_3d_2=0.
\end{equation}
The homology $I(Y,w)$ of the chain complex $C(Y,w)$ is a topological invariant of $(Y,w)$, which in fact depends only on the mod $2$ value of $w$, and $d_2$, $d_3$ induce commuting operators $u_2$, $u_3$ acting on $I(Y,w)$. In fact, the operator $u_3$ vanishes because the basepoint fibration on $\mathcal B^*(Y,w)$ lifts to a $U(2)$-bundle for an admissible pair \cite[Proposition 6.3]{Fr:q2}. Analogously to the case of integer homology spheres, this version of instanton homology is functorial with respect to cobordisms between admissible pairs and cobordisms between integer homology spheres and admissible pairs, and analogues of \eqref{lambda0-bdry-relation}-\eqref{lambda3-bdry-relation} hold for such cobordism maps. In particular, the cobordism maps commute with the actions of the operators $u_2$ and $u_3$.

The surgery exact triangle provides a useful tool in the study of instanton Floer homology \cite{Floer:surgery, BD:surgery, Scaduto:odd}. For a knot $K$ in an integer homology sphere $Y$, let $(Y_0(K),w)$ be the admissible pair given by the $0$-surgery on $K$ and a cohomology class $w$, which generates $H^2(Y_0(K))$. Then there is an exact triangle of instanton homology groups
\[
  \cdots \to I(Y_0(K),w) \to I(Y_1(K)) \to I(Y) \to I(Y_0(K),w) \to \cdots 
\]
The arrows in these exact triangles are given by cobordism maps associated with elementary cobordisms between the 3-manifolds involved in this exact triangle. In particular, they commute with the action of the operators $u_2$ and $u_3$. Fr{\o}yshov uses this fact to make the following observation \cite[Proposition 6.4]{Fr:q2}. Since $(Y_0(K),\omega_0)$ is an admissible pair, the action of $u_3$ on $I(Y_0(K),\omega_0)$ is trivial. Therefore, the above exact triangle implies that 
\begin{equation}\label{eqn:n3-surgery-bound}
  |n_3(Y_1(K))-n_3(Y)|\leq 1.
\end{equation}
This inequality will be essential for us in the next section.

\subsection{Integer homology spheres and indefinite manifolds}\label{sec:weak-ineq}
We begin this section by proving a weaker version of the inequality \eqref{b+-ineq} of Theorem \ref{main-thm}. In particular, this weaker version of \eqref{b+-ineq} is sufficient to prove Corollary \ref{surgery-number-bound}. We conclude by computing $q_3$ of the family of Brieskorn spheres $\Sigma(p,q,pqk-1)$. 

\begin{prop}\label{q3_min_lemma}
	Let $Y$ be an integer homology sphere and $L\subseteq Y$ be an integrally-framed link whose linking matrix has determinant $\pm 1$. 
	Let $W\colon Y \to Y'$ denote the trace of $L$.
	Then if $Y$ is $q_3$-minimal, we have
    \begin{equation}\label{main-ineq-weak}
        q_3(Y') - q_3(Y) \ge -b^+(W),
    \end{equation}
    and if $Y$ is $q_3$-maximal, we have
      \begin{equation}\label{main-ineq-weak-maximal}
        q_3(Y') - q_3(Y) \le b^-(W).
    \end{equation}
    In particular, both inequalities hold if $Y=S^3$.
\end{prop}

First we establish that after blowing up, $W$ can be decomposed as a composite of simpler cobordisms.
	
\begin{lemma}\label{blowup_trick}
	Let $W:Y\to Y'$ be a 2-handle cobordism between integer homology spheres associated with an integrally framed link $L$, and let $n=b^+(W)$. Then there are cobordisms 
	$W_i:Z_{i-1}\to Z_{i}$ for $1\leq i\leq n+1$ such that 
	$Z_0=Y$, $Z_{n+1}=Y'$, $W_i$ for $1\leq i\leq n$ is the trace cobordism of a $1$-framed knot in $Z_{i-1}$, $W_{n+1}$ is negative-definite, and 
	\[
	W_1\cup_{Z_1}W_2\cup_{Z_2}\cdots\cup_{Z_{n}}W_{n+1} \cong W\# \overline{\mathbb {CP}}^2.
	\]
\end{lemma}
\begin{proof}
	The claim is tautological for $n = 0$. If $n \ge 1$, let $L'$ be the result of adding a disjoint $(-1)$-framed unknot to $L$. The trace cobordism associated with $L'$ is $W \# \overline{\mathbb{CP}}^2$. Because the linking form of $L'$ is an indefinite odd unimodular form, it is equivalent to a diagonal form. We may perform a sequence of handleslides on the components of $L'$ to obtain a link $L''$ with linking form $\text{diag}\{1,\cdots,1,-1,\cdots,-1\}$ and trace cobordism diffeomorphic to $W \# \overline{\mathbb{CP}}^2$. The manifold $Z_i$ may be defined as the result of surgery on $L''_1 \cup \dots \cup L''_i$, with the cobordism $W_i$ for $1 \le i \le n$ equal to the trace cobordism associated with surgery on $L''_i \subset Z_{i-1}$. 
\end{proof}

\begin{proof}[Proof of Proposition \ref{q3_min_lemma}]
	It suffices to establish \eqref{main-ineq-weak}, because \eqref{main-ineq-weak-maximal} follows from \eqref{main-ineq-weak} by changing the orientation of 
	$W$ and using that $q_3$ negates under orientation-reversal, as well as the fact that the orientation reversal of a $q_3$-maximal integer homology sphere is $q_3$-minimal.
	Since \eqref{main-ineq-weak} is insensitive to blowing up $W$, we may assume that $W$ is diffeomorphic to the decomposition 
	\[
	  W_1\cup_{Z_1}W_2\cup_{Z_2}\cdots\cup_{Z_{n}}W_{n+1}
	\]
	where $W_i$ and $Z_i$ are as in the statement of Lemma \ref{blowup_trick}. In particular, the 3-manifold $Z_{i+1}$ is obtained from $Z_{i}$ by performing $1$-surgery 
	along a knot, and hence \eqref{eqn:n3-surgery-bound} implies that
	\[
	  n_3(Z_{i+1})\leq  n_3(Z_{i})+1
	\]
	for $0 \le i \le n-1$. Since $Z_0=Y$ is $q_3$-minimal, we obtain from the above inequalities and \eqref{nilp-q3-ineq} that 
	\begin{align}
	  q_3(Z_n)-q_3(Y)&\geq -n_{3}(Z_n)+n_3(Z_{0})\nonumber\\
	  &\geq -n = -b^+(W).\label{ineq-Z-n}
	\end{align}
	The cobordism $W_{n+1}:Z_n\to Y'$ is negative-definite, and Proposition \ref{lemma:monotonic} implies that $q_3(Y')$ is at least as large as $q_3(Z_n)$. Combining this 
	and \eqref{ineq-Z-n} gives the desired claim.
\end{proof}

This is all that is needed for the proof of Corollary \ref{surgery-number-bound} presented in the introduction. We conclude this section by using it to present some additional examples of manifolds with $q_3(Y) = S_D(Y) = 1$. 

\subsection{Connected sum with the Poincar\'e homology sphere and $q_3$}\label{connected-sum-P}
The purpose of this section is to prove the following claim, inspired by arguments for $q_2$ presented in \cite[Section 9]{Fr:q2}.

\begin{prop}\label{proposition:superadditivity}
	If $P$ is the Poincar\'e sphere and $Y$ is an integer homology sphere, then $q_3(Y \# P) \ge q_3(Y) + 1$. 
\end{prop}

We will prove this by a small case analysis: the most difficult case is $q_3(Y) \ge 1$; the argument for $q_3(Y) = 0$ is a mild variation, and $q_3(Y) \le -1$ follows purely formally from these two cases.

Given any pair of 3-manifolds $Y$ and $Y'$, there is a {\it pair-of-pants} cobordism $Y\sqcup Y'\to Y\#Y'$, and we denote this cobordism by $V$ in the special case $Y'=P$. Equip $V$ with three paths: $\gamma_{01}$ connects the basepoints of $Y$ and $P$, $\gamma_{02}$ connects the basepoints $y, y_\#$ of $Y$ and $Y \# P$, and $\gamma_{12}$ connects the basepoints $y_P, y_\#$ of $P$ and $Y \# P$. We extend these to maps $\gamma_{ij}^+:\mathbb R \to V^+$ by demanding that $\gamma^+_{02}(t) = (t, y)$ for $t \le -1$ and $\gamma^+_{02}(t) = (t, y_\#)$ for $t \ge 1$, and similarly for the other two paths. The method of proof uses operators which count instantons on $V$ with prescribed limits on the Poincar\'e sphere, cut down by sections over the bundles associated with $\gamma_{ij}$. As in Section \ref{Fro-approach}, for $\alpha \in \mathfrak C(Y)$, $\alpha' \in \mathfrak C(P)$, $\alpha_\# \in \mathfrak C(Y \# P)$, we may form the configuration space $\mathcal B^*(V;\alpha,\alpha',\alpha_\#)$, and three basepoint fibrations 
\[
  \mathbb E(\gamma_{02}),\, \mathbb E(\gamma_{01}),\,\mathbb E(\gamma_{12}) \to \mathbb R \times \mathcal B^*(V;\alpha,\alpha',\alpha_\#).
\]
Choose generic sections $s^{ij}_k$ for $k \in \{0,1,2\}$ of these bundles which restrict to the generic sections of the basepoint fibrations over $Y$, $P$, and $Y \# P$. Using these generic sections, define the moduli spaces 
\begin{align*}
M^{(2)}_{ij}(V; \alpha, \alpha', \alpha_\#) &= \{(t, A) \in \mathbb R \times M(V; \alpha, \alpha', \alpha_\#) \mid s^{ij}_1(t, A)\wedge s^{ij}_2(t,A)=0\},\\
M^{(3)}_{ij}(V;\alpha, \alpha', \alpha_\#) &= \{(t,A)\in\Bbb R\times M(V;\alpha, \alpha', \alpha_\#)\mid s^{ij}_0(t,A)=0\}.
\end{align*}
We also introduce additional moduli spaces, subspaces of $\Bbb R^2\times M(V;\alpha, \alpha', \alpha_\#)$:
\begin{align*}
    M^{(2,2)}_{01,12}(V; \alpha, \alpha', \alpha_\#) &= \{(t_1, t_2, A) \mid s^{01}_1(t_1, A)\wedge s^{01}_2(t_1,A) = 0,\\
    &\hspace{4cm}s^{12}_1(t_2,A) \wedge s^{12}_2(t_2,A) = 0\},\\
    M^{(2,3)}_{01,02}(V;\alpha, \alpha', \alpha_\#) &= \{(t_1,t_2,A) \mid s^{01}_1(t_1, A)\wedge s^{01}_2(t_1,A) =0,\, s^{02}_0(t_2, A) = 0\}.
\end{align*}
In each case, if the path $\gamma_{ij}$ is involved in the definition of the moduli space, then we assume that the limiting perturbed flat connections on the ends of $V$ interacting with the path are irreducible. Using notation similar to the previous section, we denote the $d$-dimensional part of each of these moduli spaces with the subscript $d$.

Recall from Example \ref{ex:Poincare} that $\beta_0$ is the generator of degree $1$ in $C(P)$. Define operators $\lambda_{ij}: C(Y) \to C(Y \# P)$ for $i,j\in\{0,1\}$ as follows:
\[
    \langle\lambda_{ij}\alpha, \alpha^\#\rangle = \left\{
    \begin{array}{ll}
        \#M(V;\alpha,\beta_0,\alpha^\#)_0 & (i,j)=(0,0),\\
        \#M^{(2)}_{01}(V;\alpha,\beta_0,\alpha^\#)_0 & (i,j)=(1,0),\\
        \#M^{(2)}_{12}(V;\alpha,\beta_0,\alpha^\#)_0 & (i,j)=(0,1),\\
        \#M^{(2,2)}_{01,12}(V;\alpha,\beta_0,\alpha^\#)_0 & (i,j)=(1,1).
    \end{array}\right.
\]

\begin{lemma}\label{lemma:d2-after-lambda}
    The operators $\lambda_{ij}$ above satisfy the relations
  \begin{align}
	d_1^\#\lambda_{00}+\lambda_{00}d_1&= 0 \label{lambda00-bdry-relation}\\
	 d_1^\#\lambda_{10}+\lambda_{10}d_1 + \lambda_{00}d_2&= 0 \label{lambda10-bdry-relation}\\
	d_1^\#\lambda_{01}+d_2^\#\lambda_{00} + \lambda_{01}d_1&= 0 \label{lambda01-bdry-relation}\\
	d_1^\#\lambda_{11} +d_2^\#\lambda_{10} + \lambda_{11}d_1 + \lambda_{01}d_2&= 0 \label{lambda11-bdry-relation}
\end{align}
\end{lemma}
\begin{proof}
    These four relations can be verified by analyzing the ends of $1$-dimensional spaces respectively used in the definition of the operators $\lambda_{00}$, $\lambda_{10}$, 
    $\lambda_{01}$ and $\lambda_{11}$.
    We argue \eqref{lambda11-bdry-relation} in detail, which is obtained by counting ends of $M^{(2,2)}_{01,12}(V;\alpha,\beta_0,\alpha^\#)_1$. This moduli space is a subspace of 
    $\mathbb R^2 \times M(V; \alpha, \beta_0, \alpha^\#)_3$, and one should interpret elements of the space $\mathbb R^2 \times M(V; \alpha, \beta_0, \alpha^\#)$ 
    as a choice of point on each of $\gamma_{01}^+$ and $\gamma_{12}^+$, as well as an ASD connection on $V^+$. This moduli space admits a compactification in terms of broken solutions; for any broken solution, in addition to a broken trajectory on $V^+$, one chooses a point on each of $\gamma_{01}^+$ and $\gamma_{12}^+$, where the basepoint may be in $V^+$ or a cylinder on one of the boundary components of $V$. First we observe that there is no room for a broken solution where breaking happens along $P$. 
    For that to happen, after forgetting the basepoints, the involved ASD connections should converge to an element of the form 
    \[\breve M(P;\beta_0,\beta)_d\times M(V;\alpha,\beta,\alpha^\#)_{2-d}\]
    for $0<d\leq 2$. But there is no irreducible flat connection $\beta$ on $P$ with $0 < |\beta_0| - |\beta| \le 3$.
    
    Next, we analyze the ends of the moduli space $M^{(2,2)}_{01,12}(V;\alpha,\beta_0,\alpha^\#)_1$, depending on the behavior of the basepoints. The ends for which neither basepoint goes to infinity take the forms
    \[
      \breve M(Y;\alpha,\beta)_0\times M^{(2,2)}_{01,12}(V;\beta,\beta_0,\alpha^\#)_0,\hspace{0.8cm}
      M^{(2,2)}_{01,12}(V;\alpha,\beta_0,\beta^\#)_0\times \breve M(Y \# P;\beta^\#,\alpha^\#)_0.
    \]
    Counting these two types of ends for all possible choices of $\beta$ and 
    $\beta^\#$ respectively gives $\langle \lambda_{11}d_1 (\alpha),\alpha^\#\rangle$ and $\langle d_1^\#\lambda_{11}(\alpha),\alpha^\#\rangle$.
    
    The ends for which one of the basepoints goes to infinity off one of the two ends $Y$ and $Y\#P$ are of the form:
    \[
      M^{(2)}(Y;\alpha,\beta)_0\times M^{(2)}_{12}(V;\beta,\beta_0,\alpha^\#)_0,\hspace{0.8cm}
      M^{(2)}_{01}(V;\alpha,\beta_0,\beta^\#)_0\times M^{(2)}(Y \# P;\beta^\#,\alpha^\#)_0.
    \]
    Counting these two types of ends for all possible choices of $\beta$ and 
    $\beta^\#$ respectively gives $\langle \lambda_{01}d_2 (\alpha),\alpha^\#\rangle$ and $\langle d_2^\#\lambda_{10}(\alpha),\alpha^\#\rangle$.
\end{proof}

The relations of Lemma \ref{lemma:d2-after-lambda} can be written in the following compact matrix form, viewed as an identity of linear maps from $C(Y)\oplus C(Y)$ to $C(Y\#P)\oplus C(Y\#P)$:
\begin{equation}\label{lambdaij-matrix-rel}
	\left[\begin{array}{cc}d_1^\#&0\\d_2^\#&d_1^\#\end{array}\right]\left[\begin{array}{cc}\lambda_{10}&\lambda_{00}\\\lambda_{11}&\lambda_{01}\end{array}\right]=
	\left[\begin{array}{cc}\lambda_{10}&\lambda_{00}\\\lambda_{11}&\lambda_{01}\end{array}\right]\left[\begin{array}{cc}d_1&0\\d_2&d_1\end{array}\right]
\end{equation}
Towards proving Proposition \ref{proposition:superadditivity}, fix $\zeta,\eta\in C(Y)$ satisfying \eqref{q3-pos-def-ch-cx}. In particular, $d_1\zeta=0$ and  $d_2\zeta=d_1\eta$. Define $\zeta^\#,\eta^\#\in C(Y\#P)$ as follows:
\begin{equation}\label{zeta-eta-sharp}
  \left[\begin{array}{c}\zeta^\#\\\eta^\#\end{array}\right]=
  \left[\begin{array}{cc}\lambda_{10}&\lambda_{00}\\\lambda_{11}&\lambda_{01}\end{array}\right]\left[\begin{array}{c}\zeta\\\eta\end{array}\right]
\end{equation}
Then \eqref{lambdaij-matrix-rel} implies that $d_1^\# \zeta_\# = 0$ and $d_2^\# \zeta^\# = d_1^\# \eta^\#$.

To relate $\delta_1^\#(d_3^\#)^i\zeta^\#$ to $\delta_1d_3^i\zeta$, we need to introduce another set of operators. For this purpose, we define $\mu_{i0}\colon C(Y)\to C(Y\#P)$ for $i\in\{0,1\}$:
\[
    \langle\mu_{i0}\alpha,\alpha^\#\rangle = \left\{
        \begin{array}{ll}
            \#M^{(3)}_{02}(V;\alpha,\beta_0,\alpha^\#)_0 & i = 0,\\
            \#M^{(2,3)}_{01,02}(V;\alpha,\beta_0,\alpha^\#)_0 & i = 1.
        \end{array}
    \right.
\]
We also define $\Delta_{0},\Delta_{1}\colon C(Y)\to \Bbb F_2$ as
    \[
        \Delta_{0}(\alpha) = \#M(V;\alpha,\beta_0,\theta^\#)_0, \quad \Delta_{1}(\alpha) = \#M^{(2)}_{01}(V;\alpha,\beta_0,\theta^\#)_0,
    \]
and the operators $ \Delta'_{0}, \Delta'_1 \colon \Bbb F_2\to C(Y\#P)$ as 
\[
   \langle\Delta'_{0}(1),\alpha^\#\rangle = \#M(V;\theta,\beta_0,\alpha^\#)_0, \quad \langle\Delta_{1}'(1), \alpha^\#\rangle = \#M^{(2)}_{12}(V;\theta,\beta_0,\alpha^\#)_0.
\]
\begin{lemma}\label{lemma:d3-after-lambda}
    The operators $\mu_{i0}$, $\Delta_{i}$, and $\Delta'_{i}$ described above satisfy the relations
  \begin{align}
	d_1^\#\mu_{00}+d_3^\#\lambda_{00}+ \mu_{00}d_1 + \lambda_{00}d_3&= 0 \label{mu00-bdry-relation}\\
	  d_1^\#\mu_{10}+ d_3^\#\lambda_{10}+\mu_{10}d_1 + \mu_{00}d_2 + \lambda_{10}d_3 + \lambda_{00}d_4 + \Delta'_{0}\delta_1&= 0 \label{mu10-bdry-relation}\\
	 \Delta_{0}d_1+\delta_1^\#\lambda_{00}&= 0 \label{Delta0-bdry-relation}\\
	 \Delta_{1}d_1+ \Delta_{0}d_2 +\delta_1^\#\lambda_{10}&= 0 \label{Delta1-bdry-relation}\\
	 d_1^\#\Delta'_{0}+ \lambda_{00}\delta_{4}&= 0 \label{Delta0'-bdry-relation}\\
	 d_1^\#\Delta'_{1}+d_2^\#\Delta'_{0}+ \lambda_{01}\delta_{4} &= 0 \label{Delta1'-bdry-relation}\\
	  \delta_1^\#\Delta'_{0}+\Delta_{0}\delta_{4}&= 1 \label{height1-bdry-relation}
\end{align}
\end{lemma}

\begin{proof}
	The relations \eqref{mu00-bdry-relation}-\eqref{Delta1'-bdry-relation} can be verified by analyzing the ends of the $1$-dimensional counterpart of the moduli spaces involved in the definition of the maps 
	$\mu_{00}$, $\mu_{10}$, $\Delta_{0}$, $\Delta_{1}$, $\Delta'_{0}$ and $\Delta'_1$. The proof of \eqref{mu00-bdry-relation} is similar to the proofs of 
	\eqref{di-rel} and \eqref{lambda2-bdry-relation}. The proof of \eqref{mu10-bdry-relation} is analogous to the proofs of \eqref{d4-rel} and \eqref{lambda3-bdry-relation}; that there are no contributions from solutions broken along $P$ follows from an index argument akin to that used in the proof of Lemma \ref{lemma:d2-after-lambda}. The proofs of \eqref{Delta0-bdry-relation} and \eqref{Delta0'-bdry-relation} are similar to the proofs of \eqref{Delta-0-bdry-relation} and \eqref{Delta-3-bdry-relation}, and a slight variation gives \eqref{Delta1-bdry-relation} and \eqref{Delta1'-bdry-relation}.
	To prove \eqref{height1-bdry-relation}, consider the ends of the moduli space $M(V;\theta,\beta_0,\theta^\#)_1$. In addition to the ends of the form 
	\[
	\breve M(Y;\theta,\alpha)_0\times M(V;\alpha,\beta_0,\theta^\#)_0,\hspace{1cm}M(V;\theta,\beta_0,\alpha^\#)_0\times \breve M(Y\#P;\alpha^\#,\theta^\#)_0,
	\]
	which give the terms $\Delta_0 \delta_4$ and $\delta_1^\# \Delta_0'$, there is another type of end obtained by gluing the elements of $\breve M(P;\beta_0,\theta)$ to the trivial connection on $V$. Because $\breve M(P; \beta_0, \theta)$ is a singleton, we have established \eqref{height1-bdry-relation}.	
\end{proof}

We will handle the easiest case of Proposition \ref{proposition:superadditivity} first.
	
\begin{lemma}\label{lemma:case0}
	If $q_3(Y) \ge 0$, then $q_3(Y \# P) \ge 1$. 
\end{lemma}

\begin{proof} 
Because $q_3(Y) \ge 0$, there exist $\zeta, \eta \in C(Y)$ satisfying $\delta_4(1) = d_1 \zeta + d_2 \eta$ and $d_1 \eta = 0$. Define $\zeta^\#, \eta^\# \in C(Y \# P)$ by \[\zeta^\# = \Delta'_0(1) + \lambda_{00} \zeta + \lambda_{10}\eta, \quad \eta^\# = \Delta_1'(1) + \lambda_{01}\zeta + \lambda_{11}\eta.\] That $d_1^\# \zeta^\# = 0$ follows from the relations \eqref{lambda00-bdry-relation}, \eqref{lambda10-bdry-relation}, and \eqref{Delta0'-bdry-relation}. That $\delta_1^\# \zeta^\# = 1$ follows from \eqref{Delta0-bdry-relation}, \eqref{Delta1-bdry-relation}, and \eqref{height1-bdry-relation}. Finally, that $d_2^\# \zeta^\# = d_1^\# \eta^\#$ follows from \eqref{Delta1'-bdry-relation}, \eqref{lambda01-bdry-relation}, and \eqref{lambda11-bdry-relation}. 
\end{proof}

\begin{lemma}\label{simplify-q3-eq-zeta-ind}
Fix $k > 0$. Let $\zeta$, $\eta$ satisfy $d_1 \zeta = 0$ and $d_2 \zeta = d_1 \eta$, and suppose $\delta_1 d_3^i \zeta = 0$ for $i < k-1$. Let $\zeta^\#$ be defined as in \eqref{zeta-eta-sharp}. Then for any $0\leq j \le i \le k$ with $j < k$, we have
	\begin{equation}\label{q3-eq-zeta-ind}
	  \delta_1^\#(d_3^\#)^{i}\zeta^\#= \delta_1^\#(d_3^\#)^{i-j}\left(\lambda_{10}d_3^j\zeta+\lambda_{00}\left(d_3^j\eta+\sum_{0\leq l\leq j-1}d_3^ld_4d_3^{j-1-l}\zeta\right)\right).
	\end{equation}
	If $i = j = k$, a similar identity holds with the inclusion of the additional term $\delta_1^\#\Delta'_0\delta_1d_3^{k-1}\zeta$.
\end{lemma}

\begin{proof}
For a fixed $i \le k$, we argue by induction on $j$. The base case $j = 0$ follows from applying $\delta_1^\# (d_3^\#)^i$ to the defining equation \eqref{zeta-eta-sharp} for $\zeta^\#$. Suppose now that \eqref{q3-eq-zeta-ind} holds for a given $j < i$; we seek to establish this identity for $j+1$. Rewriting $(d_3^\#)^{i-j} = (d_3^\#)^{i-j-1} d_3^\#$ and applying \eqref{mu10-bdry-relation} and \eqref{mu00-bdry-relation} to expand the terms $d_3^\# \lambda_{10}$ and $d_3^\# \lambda_{00}$ in \eqref{q3-eq-zeta-ind}, we obtain
	\begin{align*}
		\delta_1^\#(d_3^\#)^{i}\zeta^\# =&\; \delta_1^\#(d_3^\#)^{i-j}\left(\lambda_{10}d_3^j\zeta+\lambda_{00}\left(d_3^j\eta+\sum_{\ell}d_3^\ell d_4d_3^{j-1-\ell}\zeta\right)\right)\\
		=&\;\delta_1^\#(d_3^\#)^{i-j-1}\left(d_1^\#\mu_{10}+ \mu_{10}d_1 + \mu_{00}d_2 + \lambda_{10}d_3 + \lambda_{00}d_4 + \Delta'_{0}\delta_1\right)d_3^j\zeta\\
		&+\delta_1^\#(d_3^\#)^{i-j-1}\left(d_1^\#\mu_{00}+ \mu_{00}d_1 + \lambda_{00}d_3\right)\left(d_3^j\eta+\sum_{\ell}d_3^\ell d_4d_3^{j-1-\ell}\zeta\right)
	\end{align*}
	Because $\delta_1^\# (d_3^\#)^\ell d_1^\# = \delta_1^\# d_1^\# (d_3^\#)^\ell = 0$, the terms involving $d_1^\# \mu_{10}$ and $d_1^\# \mu_{00}$ are zero. Similarly, the term $\mu_{10} d_1 d_3^j \zeta = 0$ because $d_1 \zeta = 0$ and $d_1$ commutes with $d_3$. The remaining terms involving $\mu_{00}$ can be seen to sum to zero by applying \eqref{d4-rel} as well as the relation $d_2 \zeta = d_1 \eta$. This leaves the relation \eqref{q3-eq-zeta-ind} with an additional term $\delta_1^\#(d_3^\#)^{i-j-1}\Delta'_{0}\delta_1d_3^j\zeta$, which vanishes by hypothesis on $\zeta$ unless $j = i-1 = k-1$. This completes the induction.
\end{proof}

This is enough to establish Proposition \ref{proposition:superadditivity} in the most intricate case.

\begin{lemma}\label{lemma:case1}
If $q_3(Y) \ge 1$, then $q_3(Y \# P) \ge q_3(Y) + 1$. 
\end{lemma}

\begin{proof}
Let $k = q_3(Y) > 0$ and suppose $\zeta, \eta$ are chosen such that $\delta_1 d_3^i \zeta = 0$ for $i < k-1$, but $\delta_1 d_3^{k-1} \zeta = 1$.	Taking $j = i \le k-1$ and substituting \eqref{Delta0-bdry-relation} and \eqref{Delta1-bdry-relation} into \eqref{q3-eq-zeta-ind} gives
	\begin{align*}
		\delta_1^\#(d_3^\#)^{i}\zeta^\#&= \delta_1^\#\lambda_{10}d_3^i\zeta+\delta_1^\#\lambda_{00}(d_3^i\eta+\sum_{0\leq \ell\leq i-1}d_3^\ell d_4d_3^{i-1-\ell}\zeta)\\
		&= ( \Delta_{1}d_1+ \Delta_{0}d_2 )d_3^i\zeta+\Delta_{0}d_1(d_3^i\eta+\sum_{0\leq \ell\leq i-1}d_3^\ell d_4d_3^{i-1-\ell}\zeta)\\
		&=  \Delta_{0}\sum_{0\leq \ell\leq i-1}d_3^\ell\delta_4\delta_1d_3^{i-1-\ell}\zeta = 0.
	\end{align*}
	The third equality follows from a combination of the relations $d_1 \zeta = 0, d_2 \zeta = d_1 \eta$, and an iterated application of \eqref{d4-rel}. The final equality follows because $\delta_1 d_3^m \zeta = 0$ for $m < k-1$. A similar computation when $j=i=k$ gives the relation  
	\[
	\delta_1^\#(d_3^\#)^{k}\zeta^\#=(\delta_1^\# \Delta_0' + \Delta_0 \delta_4) \delta_1 d_3^{k-1} \zeta,
	\]
	which is equal to $\delta_1 d_3^{k-1}\zeta = 1$ by \eqref{height1-bdry-relation}. Therefore, $q_3(Y\#P)\geq k+1$. 
\end{proof}

The remaining case follows purely formally.
\begin{proof}[Proof of Proposition \ref{proposition:superadditivity}.] Set \[f(Y) = q_3(Y \# P) - q_3(Y);\] we aim to show $f(Y) \ge 1$ for all $Y$. Note that $f(Y) = f(-Y \# -P)$ because $q_3(-Y) = -q_3(Y)$ and $q_3$ is a homology cobordism invariant. If $q_3(Y) \ge 0$ or $q_3(-Y \# -P) \ge 0$, then $f(Y) \ge 1$ by Lemma \ref{lemma:case0} and Lemma \ref{lemma:case1}. If both quantities are negative, then $q_3(Y \# P) > 0 > q_3(Y)$, so $f(Y) \ge 2$. 
\end{proof}	

\begin{remark}
This proof of Proposition \ref{proposition:superadditivity} is reconceptualized in Example \ref{example:height-one}, such that the argument is uniform in $q_3(Y)$. We delay the proof of the reverse inequality, and thus Theorem \ref{main-thm}\ref{main-thm-iii} in full, until Section~\ref{sec:susp}.
\end{remark}

\subsection{Proof of Theorem \ref{intro1}}\label{sec:hyperbolic-example}
The first part of Theorem \ref{intro1} is now straightforward. Because $\#^n P$ is surgery on a disjoint union of $n$ trefoil knots, it follows from Example \ref{ex:Poincare}, Proposition \ref{proposition:superadditivity}, and Corollary \ref{surgery-number-bound} that for $n \ge 0$ \[n \le \hp(\#^n P) \le S_D(\#^n P) \le S(\#^n P) \le n.\] 

Myers proved that every closed oriented $3$-manifold is integer homology cobordant to a hyperbolic $3$-manifold \cite[Theorem 5.1]{Myers}. Because $q_3$ is a homology cobordism invariant, it follows immediately that the Dehn surgery number (indeed, the signed surgery number) of hyperbolic integer homology spheres is unbounded. Myers also proved that every link is concordant to a hyperbolic link \cite[Theorem 7.2]{Myers}. Choose a hyperbolic link $L$ concordant to $\sqcup_n T_{2,3}$. By Thurston's hyperbolic Dehn surgery theorem \cite[Theorem 5.8.2]{Thurston4}, performing $-1/k$ surgery on each component gives a hyperbolic manifold $M_k$ for large $k$. This manifold is homology cobordant to $\#^n \Sigma(2,3,6k-1)$, and an argument similar to the above gives that $q_3(M_k) = n$ for all $k$. Thus, there exist hyperbolic homology spheres with prescribed Dehn surgery number as well. However, in addition to being inexplicit, we lose control of the integer surgery number.

\begin{figure}
	\centering
	\begin{minipage}{.5\textwidth}
		\centering
		\begin{tikzpicture}[scale=1]
			
			\tikzset{
				knot/.style={
					line width=0.5mm,
					line cap=round,
					line join=round
				}
			}
			
			\draw[knot] (-0.2,-0.15) .. controls (2.4,0.9) and (2,-3.2) .. (-0.2,-1.4);
			\draw[knot] (0.3,-0.2) .. controls (-0.2,-2.2) and (-2.4,-1.8) .. (-1,-0.7);
			\draw[knot] (-0.45,-1.2) .. controls (-1.75,0.65) and (-0.35,1.3) .. (0, 0.9);
			
			\draw[knot] (-1,0.75) .. controls (-2.4,3.15) and (3.35,2.3) .. (-0.6, 0.05);
			\draw[line width=1.5mm, draw=white] (-0.5, 2.2) -- (-0.32, 2.2);
			
			\draw[knot] (-0.8,0.95) .. controls (-1.7,2.6) and (2.1,2) .. (-0.38,0.48);
			\draw[line width=1.5mm, draw=white] (-0.5, 1.95) -- (-0.32, 1.95);
			
			\draw[knot] (-0.6,0.05) .. controls (-1, -0.25) and (-0.9, 0.25) .. (-0.78,0.25);
			\draw[knot] (0.3,0.55) .. controls (0.35,0.5) and (0.35,0.2) .. (0.35,0.15);
			\draw[knot] (-0.7,-0.5) -- (-0.45,-0.3);
			\draw[knot] (-0.45,-0.3) -- (-0.55,-0.07);
			\draw[knot] (-0.35, 0.08) -- (-0.2,-0.15);
			\draw[knot] (-0.65, 0.15) -- (-0.8,0.45);
			\draw[knot] (-0.45,0.25) -- (-0.65,0.7);
			\draw[knot] (-0.65,0.35) -- (-0.6, 0.38);
			\draw[knot] (0.16,0.78) -- (0.23,0.7);
			
			{\color{blue}\node at (-0.22, 2.52) {$\bullet$};
				
				\draw[line width=0.3mm] (-0.04, 2.25)
				arc[start angle=25,end angle=335,
				x radius=0.2,y radius=0.5];
				
				\draw[line width=0.3mm] (-0.03, 1.98)
				arc[start angle=-5,end angle=7,
				x radius=0.2,y radius=0.5];}
			
		\end{tikzpicture}
	\end{minipage}
	\caption{The link $L$ is the union of a knotted component depicted in black and an unknotted component, depicted in blue with a dot.} \label{K1}
\end{figure}
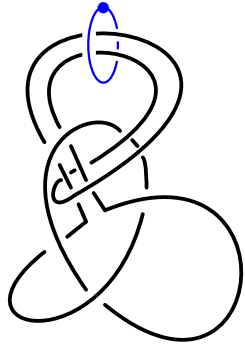

We now move to showing that the manifolds of Theorem \ref{intro1} are hyperbolic homology spheres with integer and Dehn surgery numbers both equal to $n$. The link $L_{n,k}$ is concordant to the disjoint union of $n$ left-handed trefoils, and $q_3(Y)$ is a homology cobordism invariant. Thus, \[n = q_3(Y_{n,k}) \le S_D(Y_{n,k}) \le S(Y_{n,k}) \le n\] by construction of $Y_{n,k}$. It remains to show that each $Y_{n,k}$ is hyperbolic. To do so, we use an alternate description of $Y_{n,k}$. Take the complement of the unknotted component of the link $L$ depicted in Figure \ref{K1}, and perform $(-1)$-surgery on its knotted component; denote the resulting manifold with torus boundary by $Z_1$, and its $n$-fold cyclic cover by $Z_n$. The manifold $Y_{n,k}$ is obtained by Dehn filling $Z_n$ along the slope $-1/k$. 

The SnapPy code below, executed in Sage, verifies that $Z_1$ is hyperbolic, so its covers $Z_n$ are as well. In general, given a one-cusped hyperbolic manifold $Z$, there is a function $L_Z: H_1(\partial Z) \to [0,\infty)$, the length of a geodesic in the given homology class using the metric on the boundary of a maximal horoball in the cusp neighborhood \cite{BleilerHodgson}. It is clear that $L_Z(mx) = |m| L_Z(x).$ If $f: Z' \to Z$ is an $n$-fold cyclic covering with $Z'$ also one-cusped, then $L_{Z'}(x) \ge L_Z(f_* x)$, because the maximal horoball in $Z'$ contains the preimage of the maximal horoball in $Z$. 

A theorem of Agol \cite{AgolDehn} and Lackenby \cite{LackenbyDehn} states that if $x$ is a primitive element of $H_1(\partial Z)$ with $L_Z(x) > 6$, then Dehn filling on $x$ produces a hyperbolic manifold. The same SnapPy code verifies that the only slopes of $Z_1$ of length less than $6$ are $0, \pm 1, \infty$, but that the Dehn fillings along slopes $\pm 1$ and $\infty$ are nevertheless hyperbolic, so $Y_{1,k}$ is hyperbolic for all $k$. Furthermore, SnapPy verifies that $\pm 1$ and $\infty$ have length larger than $5.5$. The map $H_1(\partial Z_n) \to H_1(\partial Z_1)$ sends the slope $-1/k$ to the homology class $(-n, k)$. It follows that $L_{Z_n}(-1/k) \ge \gcd(k,n) L_{Z_1}(-n/k)$, which is at least $6$ for all $n \ge 2$ and all $k \in \mathbb Z$. Thus, $Y_{n,k}$ is also hyperbolic for $n \ge 2$ and $k \in \mathbb Z$, completing the proof.

{\tiny
\begin{verbatim}
import snappy
L = snappy.Link('DT: [(-20,10,-22,14,2,4,18,-28,6,-26,12,30,-8),(-16,24)], [0,1,1,1,0,0,0,0,0,0,1,0,0,1,1]')
M = L.exterior()
M.dehn_fill((-1,1),0)
print('Z hyperbolic?', M.verify_hyperbolicity())
	
Z = M.filled_triangulation().canonical_retriangulation(verified=True)
print()
print('Z slopes of length <= 6:', Z.short_slopes(verified=True, bits_prec=100))
print('Z slopes of length <= 5.5:', Z.short_slopes(verified=True, length=5.5, bits_prec=100))
print()
	
M.dehn_fill((1,1),1)
T_pos = snappy.Manifold(M.filled_triangulation())
M.dehn_fill((1,0),1)
T_inf = snappy.Manifold(M.filled_triangulation())
M.dehn_fill((-1,1),1)
T_neg = snappy.Manifold(M.filled_triangulation())

# The triangulations are not always geometric.
# Randomize to find a geometric triangulation before verifying hyperbolicity.

def make_geometric(T):
	for i in range(100):
		if T.solution_type() == 'all tetrahedra positively oriented':
			break
		T.randomize()
	else:
		print("Could not find geometric triangulation for {T}.")
		return False
	
make_geometric(T_pos)
make_geometric(T_inf)
make_geometric(T_neg)

print('-1 filling on Z hyperbolic?', T_neg.verify_hyperbolicity())
print()
print('infty filling on Z hyperbolic?', T_inf.verify_hyperbolicity())
print()
print('+1 filling on Z hyperbolic?', T_pos.verify_hyperbolicity())
\end{verbatim}
}


\section{Fr{\o}yshov's invariant and $\cS$-complexes}\label{sec:S-complexes}
The goal of this section is to place the constructions and results of the previous section into a more conceptual framework. We begin by reviewing the definition of $\cS$-complexes and morphisms between them from \cites{DS1} in slightly greater generality. We then define the equivariant homology groups of an $\cS$-complex. The definition of the abstract $h$-invariant of an $\cS$-complex is reviewed using these constructions. We also mention examples of $\mathcal S$-complexes and morphisms between them that naturally arise from the previous section. In particular, $q_3$ of an integer homology sphere $Y$ is the $h$-invariant of an $\cS$-complex associated with $Y$ using instanton homology with coefficients in $\F$. This observation is due to Fr{\o}yshov. This also sets the stage for the next section, where similar algebraic objects are associated with more general 3-manifolds $Y$ and the proof of Theorem~\ref{main-thm}\ref{main-thm-ii} is given in greater generality.

\subsection{$\cS$-complexes and morphisms}
The following is the main algebraic object of this section.
\begin{definition}\label{def:S-complex}
    An (ungraded) \textit{$\cS$-complex} is an $\Bbb F_2$-vector space $\widetilde C$ of finite dimension, equipped with operators $\widetilde d, \chi\colon \widetilde C\to \widetilde C$ satisfying
    \[
        \widetilde d^2=\chi^2=\widetilde d\chi+\chi\widetilde d=0.
    \]
    In other words, $\widetilde C$ is a dg-module over $\Lambda(\chi)=\Bbb F_2[\chi]/(\chi^2)$.
\end{definition}

Since $\chi$ defines a differential on $\widetilde C$, we may take the homology $Z(\widetilde C)$ of $\widetilde C$ with respect to this differential. The homomorphism $\widetilde d$ defines a chain map of the chain complex $(\widetilde C,\chi)$, and we write $c_0$ for the induced homomorphism on $Z(\widetilde C)$. If $c_0$ is trivial, then for any $[x_0]\in Z(\widetilde C)$, there is $x_1\in \widetilde C$ satisfying $\chi x_1=\widetilde dx_0$. We also have
\[
  \chi\widetilde d x_1 = \widetilde d\chi x_1=\widetilde d^2x_0 = 0
\]  
Thus, $[\widetilde dx_1]$ defines an element of $Z(\widetilde C)$, and it can be easily checked that it depends only on $[x_0]$. In particular, we have a homomorphism $c_1\colon Z(\widetilde C)\to Z(\widetilde C)$, sending $[x_0]$ to $[\widetilde dx_1]$. In a similar way, we may inductively define the operator $c_i$ on $Z(\widetilde C)$ so long as $c_j = 0$ for $0\le j<i$; for $[x_0]\in Z(\widetilde C)$, define a sequence $\{x_j\}_{0\leq j\leq i}$ by requiring $\chi x_j = \widetilde dx_{j-1}$, and then set $c_i[x_0] = [\widetilde dx_i]$. 

\begin{definition}\label{def:perfectness}
    An $\cS$-complex $\widetilde C$ is called \emph{perfect} if it satisfies $c_i = 0$ for all $i\ge 0$.
\end{definition}

For an $\cS$-complex $\widetilde C$, choose a complement $Z\subset\ker\chi$ to $\im\chi$ and set $C=\im\chi$, so that we have the short exact sequence
\[
    0 \longrightarrow C\oplus Z \longrightarrow \widetilde C \xrightarrow{\;\;\chi\;\;} C \longrightarrow 0.
\]
A splitting of this exact sequence expresses $\widetilde C$ as $C\oplus C\oplus Z$, and with respect to this splitting, we have:
\begin{equation}\label{chi-dtilde-special-form}
    \chi = \begin{bmatrix}
        0 & 0 & 0\\
        1 & 0 & 0\\
        0 & 0 & 0
    \end{bmatrix},\quad
    \widetilde d = \begin{bmatrix}
        d & 0 & 0\\
        v & d & \delta'\\
        \delta & 0 & r
    \end{bmatrix},
\end{equation}
where $r$ can be identified with $c_0$. The special form of $\widetilde d$ follows from the assumption that it commutes with $\chi$. Now $\widetilde C$ is perfect if and only if $c_0=r=0$ and $c_i=\delta v^{i-1}\delta' = 0$ for each $i\ge 1$.

\begin{example}\label{S-cplx:integer-sphere}
    For an integer homology sphere $Y$, let $C(Y)$ and the associated homomorphisms be defined as in the previous section. Define an $\cS$-complex 
    $(\widetilde C(Y),\widetilde d,\chi_2)$ in terms of the instanton complex of $Y$ with $\widetilde C(Y)=C^{(3)}(Y)\oplus C^{(3)}(Y)\oplus Z$ where
    \begin{equation}\label{C3}
      C^{(3)}(Y)= C(Y) \oplus C(Y), \qquad Z = \F,
    \end{equation}
    and the operators 
    \begin{align}\label{chi-tilded}
        &\chi_2 = \left[
            \begin{array}{cc|cc|c}
                0 & 0 & 0 & 0 & 0\\
                0 & 0 & 0 & 0 & 0\\\hline
                1 & 0 & 0 & 0 & 0\\
                0 & 1 & 0 & 0 & 0\\\hline
                0 & 0 & 0 & 0 & 0\\
            \end{array}
        \right],\quad
        \widetilde d = \left[
            \begin{array}{cc|cc|c}
                d_1 & 0 & 0 & 0 & 0\\
                d_2 & d_1 & 0 & 0 & 0\\\hline
                d_3 & 0 & d_1 & 0 & 0\\
                d_4 & d_3 & d_2 & d_1 & \delta_4\\\hline
                \delta_1 & 0 & 0 & 0 & 0
            \end{array}
        \right].
    \end{align}
    In particular, in terms of the given splitting of $\widetilde d$, the components of $\widetilde d$ as in \eqref{chi-dtilde-special-form} are given by 
    \begin{equation}\label{eqn:differential-components}
        d = \begin{bmatrix}
            d_1 & 0 \\
            d_2 & d_1
        \end{bmatrix},\quad
        v = \begin{bmatrix}
            d_3 & 0 \\
            d_4 & d_3
        \end{bmatrix},\quad
        \delta = \begin{bmatrix}
            \delta_1 & 0
        \end{bmatrix},\quad
        \delta' = \begin{bmatrix}
            0\\\delta_4
        \end{bmatrix}.
    \end{equation}
    The relations satisfied by the components of $\widetilde d$, discussed in Section \ref{Fro-approach}, imply that $\widetilde d$ is a differential and $\widetilde C(Y)$ is an $\cS$-complex.
    In fact, $\widetilde C(Y)$ is perfect. The condition $r=0$ is obviously satisfied, and $\delta v^{j-1}\delta'=0$ because $\delta'$ has image in the second summand $C(Y)$ of $C$, 
    which is preserved by $v$ and annihilated by $\delta$.
\end{example} 

\begin{remark}
While $Z(\widetilde C(Y))$ is $1$-dimensional for an integer homology sphere $Y$, we will later apply the theory of perfect $\cS$-complexes to study rational homology spheres, for which the dimension of $Z(\widetilde C(Y))$ is equal to $|H^1(Y;\mathbb F_2)|$. 
\end{remark}

\begin{remark}\label{rmk:Scpx-q2}
	There is another $\cS$-complex that one can associate to an integer homology sphere $Y$ which has the same underlying chain complex $(\widetilde C(Y),\widetilde d)$, 
	but for which the operator $\chi_2$ is replaced by
	    \begin{align}\label{chip}
        &\chi_1= \left[
            \begin{array}{cc|cc|c}
                0 & 0 & 0 & 0 & 0\\
                1 & 0 & 0 & 0 & 0\\\hline
                0 & 0 & 0 & 0 & 0\\
                0 &0& 1 & 0 & 0\\\hline
                0 & 0 & 0 & 0 & 0\\
            \end{array}
        \right].
    \end{align}
	Then a splitting of $\widetilde C(Y)$ compatible with $\chi_1$ is given by $\widetilde C(Y)=C^{(2)}(Y)\oplus C^{(2)}(Y)\oplus Z$, where $C^{(2)}$ is isomorphic to $C^{(3)}$, 
	but the operators 
	in \eqref{eqn:differential-components} are modified by switching the roles of $d_2$ and $d_3$. We will explain the justification for the notation for these two
	$\cS$-complexes in the next section, where we also explain why 
	the $\cS$-complex in Example \ref{S-cplx:integer-sphere} is more important for us than the one in this remark. See also Remark \ref{QHS^3-to-ZHS^3}.
\end{remark}

Next, we consider {\it morphisms of $\cS$-complexes}. For a pair of $\cS$-complexes $(\widetilde C_{\pm},\widetilde d_{\pm},\chi_{\pm})$, a linear map $\widetilde\lambda\colon \widetilde C_-\to \widetilde C_+$ is a morphism of $\cS$-complexes if it satisfies
\[
  \widetilde\lambda\widetilde d_-=\widetilde d_+\widetilde\lambda,\qquad \widetilde\lambda\chi_-=\chi_+\widetilde\lambda.
\]
The mapping cone of this morphism is the following $\cS$-complex:
\[
    \text{Cone}(\widetilde\lambda)=\widetilde C_-\oplus\widetilde C_+,\qquad \widetilde d_{\text{Cone}(\widetilde\lambda)}=\begin{bmatrix}
        \widetilde d_- & 0\\ \widetilde\lambda & \widetilde d_+
    \end{bmatrix},\qquad
    \chi_{\text{Cone}(\widetilde\lambda)}=\begin{bmatrix}
        \chi_- & 0\\ 0& \chi_+
    \end{bmatrix} .
\]
In particular, $Z(\text{Cone}(\widetilde\lambda))=Z(\widetilde C_-)\oplus Z(\widetilde C_+)$. For $i\ge 0$, the operator $c_i(\text{Cone}(\widetilde\lambda))$ has the following form with respect to this splitting of $Z(\text{Cone}(\widetilde\lambda))$:
\[
    c_i(\text{Cone}(\widetilde\lambda)) = \begin{bmatrix}
        c_i(\widetilde C_-) & 0\\ \tau_i & c_i(\widetilde C_+)
    \end{bmatrix},
\]
where $\tau_i\colon Z(\widetilde C_-)\to Z(\widetilde C_+)$. In particular, the operator $\tau_i$ associated with a morphism $\widetilde \lambda$ is well-defined so long as $c_j(\widetilde C_-)$, $c_j(\widetilde C_+)$, $\tau_j$ are trivial for $j < i$. 

For the remainder of this section, we assume all $\cS$-complexes are perfect. 

\begin{definition}
    For $i\in\Bbb Z_{\ge 0}$, a \textit{height $i$ morphism} is a morphism $\widetilde\lambda\colon\widetilde C_-\to\widetilde C_+$ of perfect $\cS$-complexes satisfying $\tau_j = 0$ 
    for all $0\le j<i$. The morphism is further called \textit{strong height $i$} if in addition $\tau_i\colon Z(\widetilde C_-)\to Z(\widetilde C_+)$ is an isomorphism.
\end{definition}

With respect to our chosen decompositions of $\widetilde C_-, \widetilde C_+$, a morphism takes the form
\[
    \widetilde\lambda = \begin{bmatrix}
        \lambda & 0 & 0\\
        \mu & \lambda & \Delta'\\
        \Delta & 0 & \varepsilon
    \end{bmatrix}.
\]
It can be checked easily that
\begin{equation}\label{tau-i-formulas}
    \tau_0 = \varepsilon,\qquad \tau_{i+1} = \delta_+v_+^i\Delta'+\Delta v_-^i\delta'_-+\sum_{j=0}^{i-1}\delta_+v_+^j\mu v_-^{i-1-j}\delta_-'\;\;(i\ge 0).
\end{equation}

\begin{example}\label{example:height-zero}
    Given a negative-definite cobordism $W\colon Y\to Y'$ satisfying $H_1(W;\Bbb Z)=0$, we can construct a strong height $0$ morphism 
    $\widetilde\lambda\colon\widetilde C(Y)\to\widetilde C(Y')$ between the $\cS$-complexes associated with $Y$, $Y'$ following Example \ref{S-cplx:integer-sphere}:
    \[
        \widetilde\lambda = \left[\begin{array}{cc|cc|c}
            \lambda_0 & 0 & 0 & 0 & 0\\
            \lambda_1 & \lambda_0 & 0 & 0 & 0\\\hline
            \lambda_2 & 0 & \lambda_0 & 0 & 0\\
            \lambda_3 & \lambda_2 & \lambda_1 & \lambda_0 & \Delta_3\\\hline
            \Delta_0 & 0 & 0 & 0 & 1
        \end{array}\right].
    \]
    Identities \eqref{lambda0-bdry-relation}-\eqref{lambda3-bdry-relation} may be used to verify that $\widetilde\lambda$ is a morphism of $\cS$-complexes. 
\end{example}

\begin{example}\label{example:height-one}
	We define a morphism $\widetilde\lambda\colon (\widetilde C(Y), \chi_2)\to(\widetilde C(Y\#P), \chi_2)$ by the following matrix, using the cobordism $V\colon Y\sqcup P\to Y\#P$ introduced in Section \ref{connected-sum-P}:
    \begin{equation}\label{eqn:height-1-matrix}
        \widetilde\lambda = 
        \left[
            \begin{array}{cc|cc|c}
                \lambda_{10} & \lambda_{00} & 0 & 0 & 0\\
                \lambda_{11} & \lambda_{01} & 0 & 0 & 0\\
                \hline
                \mu_{10} & \mu_{00} & \lambda_{10} & \lambda_{00} & \Delta_{0}'\\
                {\mu_{11}} & {\mu_{01}} & \lambda_{11} & \lambda_{01} & {\Delta'_{1}}\\
                \hline
                \Delta_{1} & \Delta_{0} & 0 & 0 & 0
            \end{array}
        \right].
    \end{equation}
    All the operators are defined in Section \ref{connected-sum-P} except $ \mu_{01}$ and $ \mu_{11}$, which are given as 
    \begin{align*}
        \langle\mu_{01}\alpha,\alpha^\#\rangle &= \#M^{(2,3)}_{12,02}(V;\alpha,\beta_0,\alpha^\#)_0,\\
        \langle\mu_{11}\alpha,\alpha^\#\rangle &= \#M^{(2,2,3)}_{01,12,02}(V;\alpha,\beta_0,\alpha^\#)_0.
    \end{align*}
    Here the moduli space involved in the definition of $\mu_{11}$ is the subspace of $\Bbb R^3\times M(V;\alpha,\beta_0,\alpha^\#)$ defined by
    \begin{align*}
        M^{(2,2,3)}_{01,12,02}(V;\alpha,\beta_0,\alpha^\#) = \{(t_1,t_2,t_3,A)\mid &s_1^{01}(t_1,A)\wedge s_2^{01}(t_1,A)=0\\  &s_1^{12}(t_2,A)\wedge s_2^{12}(t_2,A)=0,\,\,s_0^{02}(t_3,A)=0\}.
    \end{align*}
    Then $\widetilde\lambda$ is a strong height $1$ morphism. The verification of $\widetilde d\widetilde\lambda=\widetilde\lambda\widetilde d$ uses 
    \eqref{lambdaij-matrix-rel}, Lemma \ref{lemma:d3-after-lambda} and the following relations that can be proved as in Lemma \ref{lemma:d3-after-lambda}:
      \begin{align*}
	d_1^\#\mu_{01}+d_2^\#\mu_{00}+d_3^\#\lambda_{01}+d_4^\#\lambda_{00}+ \mu_{01}d_1 + \lambda_{01}d_3+\delta_4^\#\Delta_{0}&= 0\\
	  d_1^\#\mu_{11}\hspace{-1pt}+\hspace{-1pt}d_2^\#\mu_{10}\hspace{-1pt}+\hspace{-1pt}d_3^\#\lambda_{11}\hspace{-1pt}+ \hspace{-1pt}d_4^\#\lambda_{10}\hspace{-1pt}+\hspace{-1pt}\mu_{11}d_1\hspace{-1pt} +\hspace{-1pt} \mu_{01}d_2 \hspace{-1pt}+\hspace{-1pt} \lambda_{11}d_3\hspace{-1pt} +\hspace{-1pt} \lambda_{01}d_4 \hspace{-1pt}+\hspace{-1pt} \delta_4^\#\Delta_{1}\hspace{-1pt}+\hspace{-1pt}\Delta'_{1}\delta_1&= 0.
	\end{align*}
    From the definition it is clear that $\tau_0=0$ and we conclude from \eqref{height1-bdry-relation} that
    \[
        \tau_1 =
        \begin{bmatrix}
            \delta_1^\# & 0
        \end{bmatrix}
        \begin{bmatrix}
            \Delta_{0}'\\\Delta_{1}'
        \end{bmatrix}+
        \begin{bmatrix}
            \Delta_{1} & \Delta_{0}
        \end{bmatrix}
        \begin{bmatrix}
            0\\\delta_4
        \end{bmatrix}=\delta_1^\#\Delta_{0}'+\Delta_{0}\delta_4=1.\qedhere
    \]
\end{example}

\begin{remark}
    We note that the same cobordism $V$ also induces a strong height-$1$ morphism between $\cS$-complexes $(\widetilde C(Y),\chi_1)$ and $(\widetilde C(Y\#P),\chi_1)$, with $\chi_1$-action defined in equation \eqref{chip}. This is achieved by switching the roles of $w_2$ and $w_3$ in the definitions of the cut-down moduli spaces used to define the operators which appear in matrix \eqref{eqn:height-1-matrix}. 
A closely related construction appears in \cite[Section 9.2]{Fr:q2}.
\end{remark}

\subsection{Equivariant homology and $h$-invariant}\label{sec:equivariant-homology-groups}
Given a perfect $\cS$-complex $\widetilde C$, we define its equivariant homology groups $\medhat H(\widetilde C), \medcheck H(\widetilde C), \medbar H(\widetilde C)$ to be the homology of the complexes defined by the formulas
\[
    \begin{array}{ll}
        \widehat C =\widetilde C\otimes_{\F} \F[x], & \widehat d=\widetilde d+x\chi\\
        \widecheck C = \widetilde C\otimes_{\F} (\F\llbracket x^{-1},x]/\F[x]), & \widecheck d=\widetilde d+x\chi\\
        \medbar C = \widetilde C\otimes_{\F} \F\llbracket x^{-1},x], & \overline d=\widetilde d+x\chi.
    \end{array}
\]
These homology groups are $\F[x]$-modules, and the first is finitely generated as such. The inclusion map $i\colon \widehat C\to \medbar C$ has quotient $\widecheck C$, so the equivariant homology groups fit in a natural exact triangle: 
\begin{equation}\label{exact-tri-equiv}
	\begin{tikzcd}[column sep=1ex, row sep=5ex, fill=none, /tikz/baseline=-10pt]
\medhat H(\widetilde C) \arrow[rr, "\mathfrak i"] & & \medbar H(\widetilde C) \arrow{dl} \\
& \medcheck H(\widetilde C) \arrow{ul} & 
\end{tikzcd}
\end{equation}
The {\it bar} version $\medbar H(\widetilde C)$ of equivariant homology can be identified explicitly. The following is proved in \cite[Corollary 4.12]{DS1}; though the $\cS$-complexes in that reference are defined more strictly than they are here, the only conditions used in the proof of the cited result are the relations $r=0, \delta v^{i-1}\delta'=0$ for $i\ge 1$. That is, the proof only uses that $\widetilde C$ is a perfect $\cS$-complex.

\begin{lemma}\label{iso-I-bar}
    There exists a privileged isomorphism of $\Bbb F_2[x]$-modules \[\Phi\colon\medbar H(\widetilde C)\xrightarrow{\cong} Z(\widetilde C)\llbracket x^{-1},x],\] which is well-defined up to multiplication by power series of the form $1+\sum_{i=1}^\infty a_ix^{-i}$, where $a_i\colon Z(\widetilde C)\to Z(\widetilde C)$ are linear maps. In particular, the leading term of an element in $\medbar H(\widetilde C)$ is well-defined.
\end{lemma}

Following \cite[Section 7.4]{DS:-unori-skein-tr}, we use Lemma \ref{iso-I-bar} to define the following subspace of $Z(\widetilde C)$:
\begin{equation}\label{Jn}
	J_n(\widetilde C)=\{z\in Z(\widetilde C)\mid\exists\, \hat c\in\medhat H(\widetilde C)\colon (\Phi\mathfrak i)(\hat c)=zx^{-n}+O(x^{-n-1})\}.
\end{equation}
By $O(x^{-n-1})$, we mean an element  of $Z(\widetilde C)\llbracket x^{-1},x]$ that includes powers of $x$ of degree at most $-n-1$. Since $\Phi$ is well-defined  up to multiplication  by terms of the form $1+\sum_{i=1}^\infty a_ix^{-i}$, $J_n(\widetilde C)$ is well-defined. Furthermore, we may use the module structure over $\F[x]$ to see that $\{J_n(\widetilde C)\}_{n\in\Bbb Z}$ defines a decreasing filtration on $Z(\widetilde C)$:
\begin{equation}\label{Jn-filtration}
  \cdots \subseteq J_{n+1}(\widetilde C)\subseteq J_{n}(\widetilde C)\subseteq J_{n-1}(\widetilde C) \subseteq \cdots \subseteq Z(\widetilde C)
\end{equation}
Finite dimensionality of $\widetilde C$ implies that $J_n(\widetilde C)$ is trivial if $n$ is large enough and $J_n(\widetilde C)=Z(\widetilde C)$ if $n$ is small enough. In particular, if $\vartheta \in Z(\widetilde C)$, we may define 
\begin{equation}\label{h-inv}
	h(\widetilde C; \vartheta) = \max\{n\in\Bbb Z\mid \vartheta \in J_n(\widetilde C)\}.
\end{equation}
This is an integer when $\vartheta$ is nonzero, and $\infty$ when $\vartheta = 0$. If $Z(\widetilde C)$ is $1$-dimensional, there is a unique nonzero element $\vartheta \in Z(\widetilde C)$, and we write $h(\widetilde C) = h(\widetilde C; \vartheta)$. 

\begin{remark}
The filtration $J_n(\widetilde C)$ and the function $h(\widetilde C; \vartheta)$ contain precisely the same information. The fact that $J_n(\widetilde C)$ is a vector space corresponds to the statement that \[\text{min}(h(\widetilde C; \vartheta), h(\widetilde C; \vartheta')) \le h(\widetilde C; \vartheta + \vartheta').\qedhere\]
\end{remark}

The $h$-invariant of an $\cS$-complex can be computed explicitly in terms of the components of $\widetilde d$ with respect to a chosen splitting $\widetilde C = C \oplus C \oplus Z$ (see \cite[Proposition 4.15]{DS1}):
\begin{lemma}\label{lemma:explicit-h-copy}
    The $h$-invariant of a perfect $\cS$-complex $\widetilde C$ is uniquely determined as follows. For a positive integer $k$, we have $h(\widetilde C;\vartheta)\geq k$ if and only if 
    there exists $\alpha\in C$ satisfying the following properties:
    \[
        d\alpha=0, \quad \delta v^{k-1}(\alpha) = \vartheta, \quad \delta v^i(\alpha)=0\;\; \text{ for } 0\le i\le k-2.
    \]
    For a nonpositive integer $k$, we have $h(\widetilde C; \vartheta)\geq k$ if and only if there are elements $a_0,\dots, a_{-k}\in Z$ and $\alpha\in C$ such that 
    $a_{-k} = \vartheta$ and 
    \[
       d\alpha + \sum_{i=0}^{-k}v^{i}\delta'(a_i)=0.
    \]
\end{lemma}

When $Y$ is an integer homology sphere, applying the characterization in Lemma \ref{lemma:explicit-h-copy} to the $\cS$-complex $(\widetilde C(Y), \chi_2)$ from Example \ref{S-cplx:integer-sphere} gives precisely the definition of $q_3(Y)$ from Section \ref{Fro-approach}. Thus, we obtain an alternate definition of $q_3(Y)$: 

\begin{cor}\label{cor:q3-determines-filtration}
	For an integer homology sphere $Y$, we have $q_3(Y)=h(\widetilde C(Y), \chi_2)$.
\end{cor}

\begin{remark}\label{rmk:q2}
In \cite{Fr:q2}, Fr{\o}yshov defines and extensively studies an invariant $q_2(Y)$ of integer homology spheres. This is equal to $h(\widetilde C(Y), \chi_1)$, as seen by applying Lemma \ref{lemma:explicit-h-copy} to the $\cS$-complex $(\widetilde C(Y), \chi_1)$ and comparing to the definition of \cite[Section 7]{Fr:q2}.
\end{remark}

Morphisms of $\cS$-complexes induce maps between the associated equivariant theories. More specifically, if $\widetilde\lambda\colon \widetilde C_-\to \widetilde C_+$ is a morphism of $\cS$-complexes, then 
\[\widehat\lambda=\widetilde\lambda\otimes 1_{\F[x]}:\widetilde C_-\otimes \F[x]\to \widetilde C_+\otimes \F[x]\]
defines a chain map $\widehat C_-\to \widehat C_+$, and it induces an $\F[x]$-module homomorphism of equivariant homology groups 
\[\widehat\lambda:\medhat H(\widetilde C_-) \to \medhat H(\widetilde C_+).\]
Similarly, we may define $\F[x]$-module homomorphisms 
\[\widecheck\lambda:\medcheck H(\widetilde C_-) \to \medcheck H(\widetilde C_+),\hspace{1cm}\overline\lambda:\medbar H(\widetilde C_-) \to \medbar H(\widetilde C_+),\]
which together define a homomorphism of exact triangles:
\begin{equation}\label{morphism-exact-tri}
\begin{array}{cccccccccc}
\cdots\xrightarrow{}&\medcheck H(\widetilde C_-)
& \xrightarrow{}
& \medhat H(\widetilde C_-)
& \xrightarrow{\,\mathfrak i_-\,}
& \medbar H(\widetilde C_-)
& \longrightarrow 
& \cdots \\[6pt]
& \downarrow \widecheck \lambda
& 
& \downarrow \widehat \lambda
& 
& \downarrow \overline \lambda
& \\[6pt]
\cdots\xrightarrow{}&\medcheck H(\widetilde C_+)
& \xrightarrow{}
& \medhat H(\widetilde C_+)
& \xrightarrow{\,\mathfrak i_+\,}
& \medbar H(\widetilde C_+)
& \longrightarrow 
& \cdots
\end{array}
\end{equation}
With respect to the identifications of $\medbar H(\widetilde C_{\pm})$ provided by Lemma \ref{iso-I-bar}, the leading term of $\overline\lambda$ for a height $n$ morphism $\widetilde\lambda$ is computed in \cite[Corollary 4.12]{DS1}:

\begin{lemma}\label{lemma:filtered-map}
    Suppose $\widetilde\lambda\colon\widetilde C_-\to\widetilde C_+$ is a height $i$ morphism of $\cS$-complexes for some nonnegative $i$. Then there are homomorphisms
     $b_j\colon Z(\widetilde C_-)\to Z(\widetilde C_+)$ for $j\geq i+1$ such that 
    \[
        \Phi_+\circ \overline\lambda\circ (\Phi_-)^{-1}=\tau_i x^{-i}+\sum_{j=i+1}^\infty b_jx^{-j}.
    \]
\end{lemma}

We will apply this in the following form: 

\begin{prop}\label{prop:tau-filtered}
	Suppose $\widetilde \lambda \colon \widetilde C_- \to \widetilde C_+$ is a height $i$ morphism of $\cS$-complexes for some nonnegative $i$. Then the map $\tau_i: Z(\widetilde C_-) \to Z(\widetilde C_+)$ satisfies $\tau_i(J_n(\widetilde C_-)) \subset J_{n+i}(\widetilde C_+)$. If $\widetilde \lambda$ has strong height $i$ and $\widehat \lambda$ is an isomorphism, this containment is an equality. 
\end{prop}

\begin{proof}
Consider the commutative diagram \eqref{morphism-exact-tri}. The vector space $J_n(\widetilde C_-)$ is defined to be the set of elements $z \in Z(\widetilde C_-)$ for which there is an element in the image of $\Phi_-\mathfrak i_-$ with leading term $zx^{-n}$. By Lemma \ref{lemma:filtered-map}, the leading term of $\overline \lambda$ is given by $\tau_i x^{-i}$, which immediately gives the stated containment. When $\widehat \lambda$ is an isomorphism, the image of $\Phi_+\mathfrak i_+$ is the same as the image of \[\Phi_+\mathfrak i_+ \widehat \lambda = (\Phi_+\overline \lambda \Phi_-^{-1})(\Phi_-\mathfrak i_-).\] Furthermore, if $\widetilde \lambda$ is a strong height $i$ morphism, then the map $\overline \Phi_+\overline \lambda \Phi_-^{-1}$ has the leading term $\tau_i x^{-i}$, which is an isomorphism. Therefore, $\tau_i$ maps $J_n(\widetilde C_-)$ onto $J_{n+i}(\widetilde C_+)$.
\end{proof}

It follows from Proposition \ref{prop:tau-filtered} that a height $i$ morphism $\widetilde\lambda\colon\widetilde C_-\to\widetilde C_+$ gives 
\begin{equation}\label{h-ineq}
	h(\widetilde C_-;\vartheta)+i \leq h(\widetilde C_+; \tau_i(\vartheta)),
\end{equation} 
which is an equality when  $\tau_i$ and $\widehat \lambda$ are isomorphisms. Applying \eqref{h-ineq} to the morphisms of $\cS$-complexes introduced in Examples \ref{example:height-zero} and \ref{example:height-one} gives a more conceptual explanation for Proposition \ref{lemma:monotonic} and Proposition \ref{proposition:superadditivity}.

Given an $\cS$-complex $\widetilde C$, we may define the dual $\cS$-complex $\widetilde C^\dagger = \Hom(\widetilde C,\Bbb F_2)$ with dual differential and $\chi$-action. For example, if $\widetilde C(Y)$ is the $\cS$-complex of an integer homology sphere $Y$ provided by Example \ref{S-cplx:integer-sphere}, then the behavior of instanton Floer homology with respect to orientation reversal discussed in Section \ref{Fro-approach} shows that $\widetilde C(Y)^\dagger$ is the $\cS$-complex of $-Y$. Identifying $Z(\widetilde C^\dagger)$ with $Z(\widetilde C)^\dagger$, the argument in \cite[Proposition 4.21]{DS1} can be adapted to prove the following.

\begin{prop}\label{prop:filtration-under-duality}
    We have the following identity for the subspace $J_n(\widetilde C^\dagger)$ of $Z(\widetilde C^\dagger)$:
    \[
        J_n(\widetilde C^\dagger) =\{\psi\in Z(\widetilde C)^\dagger\mid \psi(J_{-n+1}(\widetilde C))=0\}.
    \]
\end{prop}

This gives the following characterization of the $h$-invariant of the dual complex. 

\begin{cor}\label{cor:h-under-duality}
	For each $\vartheta^\dagger\neq 0\in Z(\widetilde C)^\dagger$, we have
	\begin{equation}\label{eqn:h-under-duality}
		h(\widetilde C^\dagger;\vartheta^\dagger) = -\max\{h(\widetilde C;\vartheta')\mid\langle\vartheta^\dagger, \vartheta'\rangle = 1\}.
	\end{equation}
\end{cor}
If $Z(\widetilde C)$ is $1$-dimensional (e.g. when $\widetilde C$ is the $\cS$-complex of an integer homology sphere), we obtain $h(\widetilde C^\dagger)=-h(\widetilde C)$. In particular, this recovers the relation $q_3(-Y) = -q_3(Y)$. 


\section{Cellular model for instanton Floer homology}\label{cellular-model}
In the previous section, we introduced the notion of $\mathcal S$-complexes as a natural home for Fr{\o}yshov-type invariants. The most natural sources of $\mathcal S$-complexes are $S^1$-manifolds equipped with a Morse function. An example is given by the Chern--Simons functional for singular $SU(2)$ connections on a pair $(Y, K)$, and the language of $\mathcal S$-complexes was developed in that context \cite{DS1,DS2,DISST}. The space of nonsingular $SU(2)$ connections on $Y$ admits instead an $SO(3)$-action, and our next goal is to obtain a dg-module from this $SO(3)$-equivariant setting.

In Section \ref{subsec:flowcat}, we explain the material of \cite{DMES} used to define a finite-dimensional dg-module $\widetilde C(Y,\mathfrak a)$ associated with a rational homology sphere $Y$ equipped with certain auxiliary data. In Section \ref{subsec:bimod}, we carry out a similar discussion for the maps induced by ``nice'' cobordisms, a notion introduced in \cite[Chapter 8]{DMES}. While our presentation is self-contained, additional details are provided in that reference. Together with the main theorem of \cite{DME1}, these are used to make precise the sense in which $\widetilde C(Y,\mathfrak a)$ is independent of the auxiliary data $\mathfrak a$. Finally, in Section \ref{subsec:QHS-Scpx}, we use this material to define the invariant $q_3(Y)$ for $\mathbb F_2$-homology spheres and give an extension to the more general case of rational homology spheres. 

\subsection{$SO(3)$-equivariant flow categories and dg-modules}\label{subsec:flowcat}
The dg-module $\widetilde C(Y,\mathfrak a)$ is defined by first producing a \emph{geometric} structure, which is an {\it $SO(3)$-equivariant flow category}. To give the definition of $SO(3)$-equivariant flow categories, we need an appropriate notion of manifolds with corners. Many competing definitions are in use; compare \cite{Joyce}. The definition used here follows \cite[Definition 3.16]{DME1}.

We begin by fixing notation for the standard stratification of $[0,\infty)^k \times \mathbb R^\ell$. For $\sigma \in \{0,1\}^k$, the face $P_\sigma$ of $[0,\infty)^k \times \mathbb R^{\ell}$ is the set of elements $x$ with $x_i = 0$ if and only if $\sigma(i) = 0$. For instance, $(0,1,0)$ corresponds to the face $\{0\} \times (0,\infty) \times \{0\} \times \mathbb R^\ell$ of $[0,\infty)^3 \times \mathbb R^\ell$. Given a partially-ordered set $\Delta$ and an element $\sigma \in \Delta$, we define $\Delta_{\ge \sigma} = \{\tau \in \Delta : \tau \ge \sigma\}$.

\begin{definition}
	A \textit{stratified-smooth manifold} $(P,\Delta)$ (or simply $P$) of dimension $n$ is the following data: 
	\begin{itemize}
		\item a partially ordered set $\Delta$ equipped with a dimension function $d: \Delta \to \mathbb N$ with maximum value $n$, 
		\item a compact topological space $P$ equipped with a partition into locally-closed subsets $P_\sigma$ for $\sigma \in \Delta$, satisfying $\overline P_\sigma = \bigcup_{\tau \le \sigma} P_\tau$, 
		\item the structure on each stratum $P_\sigma$ of a smooth manifold of dimension $d(\sigma)$.
	\end{itemize}
	We demand that the following {\it locally standard} hypothesis holds: for each $p \in P_\sigma$, there exists an isomorphism of posets $\bar \phi: \Delta_{\ge \sigma} \cong \{0,1\}^k$ for some integer $k$. This isomorphism may be chosen so that for an appropriate neighborhood $p \in U \subset P$, there exists a homeomorphism $\phi: U \to [0,\infty)^k \times \mathbb R^{n-k}$ so that $\phi$ sends the stratum $P_\tau$ onto the stratum corresponding to $\bar \phi(\tau)$. Finally, we demand that $\phi$ is a diffeomorphism on each stratum.
	
	A \textit{stratified-smooth $SO(3)$-manifold} is a stratified-smooth manifold $P$ equipped with a continuous $SO(3)$-action, for which each $P_\sigma$ is $SO(3)$-invariant, and the action is smooth on each stratum.

	A \textit{diffeomorphism} of stratified-smooth manifolds $(\phi,\bar \phi): (P,\Delta)\to (P',\Delta')$  is the pair of an isomorphism of partially-ordered sets 
	$\bar \phi:\Delta \to \Delta'$ and a homeomorphism $\phi:P\to P'$ that restricts to a diffeomorphism from $P_\sigma$ to $P'_{\bar \phi(\sigma)}$ for each 
	$\sigma\in \Delta$.
\end{definition}

Moduli spaces of Morse flowlines are naturally provided with a smooth structure on each stratum, and this structure allows us to carry out certain constructions of smooth topology in Section \ref{sec:susp}. 

\begin{definition}\label{defn:bdry-stratsmooth}
	Let $P$ be a stratified-smooth manifold of dimension $n$. The boundary of $P$ is defined as \[\partial P = \bigsqcup_{d(\sigma) = n-1} \overline{P}_\sigma,\] equipped with its natural structure as a stratified-smooth manifold.
\end{definition}

The fact that $\partial P$ is indeed a stratified-smooth manifold uses the locally standard hypothesis; a more general claim is established in \cite[Lemma A.13]{DME1}. Then $P$ admits a natural filtration by its corner strata: let $P^{(i)}$ denote the union of all strata $P_\sigma$ with $d(\sigma) \le i$. The underlying space of $P$ is an $n$-dimensional topological manifold with boundary $P^{(n-1)}$. Note that the topological boundary $P^{(n-1)}$ is not quite the same as $\partial P$, but there is an obvious quotient map $\pi:\partial P\to P^{(n-1)}$. We will use the filtration of $P$ by the subspaces $P^{(i)}$ for certain inductive arguments, as well as to perform cellular constructions.

We now use these to define certain cellular chains. 

\begin{definition}\label{cellular-fundamental-class}
Let $(P,\Delta)$ be a stratified-smooth manifold of dimension $n$ and $X$ be a CW-complex. For each nonnegative integer $i$, let $X^{(i)}$ denote the $i$-skeleton of $X$.
A continuous map $f: P \to X$ is said to be \textit{cellular} if $f(P^{(i)}) \subset X^{(i)}$ for any $i$. With coefficients in $\Bbb F_2$, the cellular fundamental class $f_*[P]$ is the element of $C_n^{\rm cell}(X)$ defined as 
\begin{equation}\label{degree-str-smooth-space-map}
	\langle f_*[P], e_\alpha\rangle = \deg(f: (P, P^{(n-1)}) \to (e_\alpha/ \partial e_\alpha,*)) \mod 2,
\end{equation}
where $e_\alpha$ denotes an $n$-cell in $X$.
\end{definition}

To be more precise, the map in \eqref{degree-str-smooth-space-map} is the composition of $f$ and the quotient map that collapses to a point the complement of $e_\alpha$ in $X^{(n)}$. In particular, $f_*[P]$ is given by mapping $[P]\in H_n(P,P^{(n-1)})$ into $H_n(X^{(n)},X^{(n-1)})$ by the map $f$. If $X$ is equipped with a smooth structure for which the characteristic maps of the CW-structure are smooth, the degree in \eqref{degree-str-smooth-space-map} has the following alternative description. Suppose $f: P \to X$ is smooth, in the sense that $f$ is continuous and restricts to a smooth map $P_\sigma\to X$ for each $\sigma\in\Delta$. Let $D_\alpha$ be a small disc intersecting $e_\alpha$ once transversely. Then $\langle f_* [P], e_\alpha \rangle$ is the intersection number of $f: P \to X$ with $D_\alpha$. Precisely, we may choose a homotopy $H$ from $f$ to a nearby smooth map $\tilde f$ so that $\tilde f$ is transverse to $D_\alpha$ and the homotopy is small enough that $H(I \times P^{(n-1)})$ does not intersect $D_\alpha$. Then $\tilde f^{-1}(D_\alpha)$ is a finite set, and the mod $2$ number of points in this set, which is independent of the choice of $D_\alpha$ and the homotopy, is equal to \eqref{degree-str-smooth-space-map}.

Because the restriction of $f$ to $\partial P$ is again cellular, the boundary also has a cellular fundamental class. 
\begin{lemma}\label{Lemma:bdry-relation}
	The cellular fundamental classes of $f$ and its restriction to $\partial P$ are related as 
	\begin{equation}\label{eqn:cellular-fundamental-class}f_*[\partial P]=\partial^{\textup{cell}} f_*[P].\end{equation}
\end{lemma} 
\begin{proof}
We present two proofs. The first relies only on the assumption that $P$ is a topological manifold with corners. The second uses an additional smoothness assumption on the strata of $X$. For our purposes, the advantage of the second proof is that it adapts easily to the more general setup of \textit{stratified-smooth spaces} that is employed in Section \ref{sec:susp}. The reader might want to see the definition of stratified-smooth spaces and their properties developed in \cite{DME1} before reading the second proof (see also Section \ref{subsec:susp-chains}).

We have the following commutative diagram:  
\begin{equation}
	\begin{array}{ccccc}
		H_n(P^{(n)},P^{(n-1)}) & \xrightarrow{i_P} & H_{n-1}(P^{(n-1)}) & \xrightarrow{j_P} & H_{n-1}(P^{(n-1)},P^{(n-2)})\\
		\downarrow f_* & & \downarrow f_* & & \downarrow f_*\\
		H_n(X^{(n)},X^{(n-1)}) & \xrightarrow{i_X} & H_{n-1}(X^{(n-1)}) & \xrightarrow{j_X} & H_{n-1}(X^{(n-1)},X^{(n-2)})
	\end{array}\label{eqn:digram-relating-fundamental-class}
\end{equation}
The cellular fundamental class $f_*[P]$ is the image of $[P]$ under the first vertical map, while the composite of the bottom row is precisely the cellular boundary map $\partial^{\text{cell}}$. The projection map $\pi\colon(\partial P, (\partial P)^{(n-2)}) \to (P^{(n-1)}, P^{(n-2)})$ satisfies $\pi_* [\partial P] = j_P i_P [P],$ so the result follows from the commutativity of the given diagram.

For the second argument, suppose that $X$ is equipped with a smooth structure for which the characteristic map of each cell is smooth. For a cell $e_\alpha$ with $\dim e_\alpha=n-1$, pick a small disk $D_\alpha$ as above. Since $f$ is cellular, the intersection of $f(P^{(n-2)})$ and $D_\alpha$ is empty, and $f(P^{(n-1)})$ intersects $D_\alpha$ only in the point $D_\alpha\cap e_\alpha$. As in the discussion preceding the lemma, we may choose a homotopy from $f$ to a smooth map $\tilde f$ so that $H(I \times P^{(n-2)})$ is disjoint from $D_\alpha$ and $\tilde f$ is transverse to $D_\alpha$. In particular, $\tilde f^{-1}(D_\alpha)\cap P^{(n-1)}$ is a finite subset of $P^{(n-1)}$, and the number of points in this set is equal to $\langle  f_*[\partial P],e_\alpha\rangle$. In addition, $\tilde f^{-1}(D_\alpha)$ is a compact $1$-dimensional stratified-smooth space \cite[Proposition A.17]{DME1}. Because the boundary of a compact $1$-dimensional stratified-smooth space consists of an even number of points \cite[Proposition A.15]{DME1}, we have 
\begin{align*}\langle f_* [\partial P], e_\alpha\rangle + \langle \partial^{\text{cell}} f_*[P], e_\alpha\rangle&= \#(\partial P \cap D_\alpha) + \#(P \cap \partial D_\alpha)\\ 
&= \#\partial(P \cap D_\alpha)=0.\qedhere
\end{align*}
\end{proof}

\begin{remark}
		The preferred definition for a manifold with corners in \cite{DMES} is based on the language of $\langle n\rangle$-manifolds, where we demand $\Delta = \{0,1\}^n$. There are also weaker definitions of manifolds with corners \cite[Remark 2.11]{Joyce}. We prefer the definition above for two reasons. First, spaces which appear in Section \ref{sec:susp} are stratified by posets more general than $\{0,1\}^n$, so $\langle n\rangle$-manifolds are too restrictive. Second, weaker definitions do not permit a definition of the boundary as simple as Definition \ref{defn:bdry-stratsmooth}; compare \cite[Definition 2.6]{Joyce}. 
\end{remark}

We now move on to the definition of flow categories. The following definition is reproduced from \cite[Definition 8.1.1]{DMES}. In its statement, it is convenient to pass between a set of $SO(3)$-orbits $\mathsf{Ob}$ and the corresponding discrete set $\mathsf{orb} = \mathsf{Ob}/SO(3)$. Elements of the latter will be indicated by Greek letters $\alpha, \beta, \dots$, and the corresponding orbit in $\mathsf{Ob}$ will be indicated by $\lift(\alpha)$. Fiber products over $\lift(\alpha)$ will be written $\times_\alpha$. 

\begin{definition}
An \textit{instanton $SO(3)$-flow category} $\mathcal M$ consists of a set of $SO(3)$-orbits $\mathsf{Ob}(\mathcal M)$ and a grading function $| \cdot | : \mathsf{orb}(\mathcal M) \to \mathbb Z$ with finitely many orbits in each grading. For each $ \alpha, \beta \in \mathsf{orb}(\mathcal M)$ with \begin{equation}\label{eqn:dim-restriction} 1\leq |\alpha| - |\beta| \le \dim \lift(\beta) + 2,\end{equation}
there is a compact stratified-smooth $SO(3)$-manifold of dimension $|\alpha| - |\beta|+\dim \lift(\alpha)-1$, denoted $\eqm(\alpha, \beta)$. These spaces are equipped with equivariant maps \[\lift(\alpha) \xleftarrow{s_{\alpha \beta}} \eqm(\alpha, \beta) \xrightarrow{t_{\alpha \beta}} \lift(\beta)\]
called the {\it source} and {\it target} maps. The stabilizer of any point in $\lift(\alpha)$ or $\eqm(\alpha, \beta)$ is required to be conjugate to one of $\{I\}, SO(2)$, or $SO(3)$. The orbits $\lift(\alpha)$ with trivial stabilizer are called \emph{irreducible}. The orbits $\lift(\alpha) \cong S^2$, with stabilizer conjugate to $SO(2)$, are called {\it abelian}, and the orbits $\lift(\alpha) = \{\ast\}$ with stabilizer $SO(3)$ are called {\it central}.\\ 

\noindent Whenever $\alpha, \beta$ satisfy \eqref{eqn:dim-restriction}, these spaces are also equipped with composition maps \[\circ: \eqm(\alpha, \gamma) \times_{\gamma} \eqm(\gamma, \beta) \to \partial \eqm(\alpha, \beta)\] which are $SO(3)$-equivariant, associative, and preserve the source and target maps. We require that these maps induce a diffeomorphism 
\begin{equation}\label{eqn:flowcat-boundary-relation}\partial \eqm(\alpha, \beta) \cong \bigsqcup_{\gamma \in \mathsf{orb}(\mathcal M)} \eqm(\alpha, \gamma) \times_{\gamma} \eqm(\gamma, \beta).\end{equation}
Finally, there is an equivariant periodicity homeomorphism $U: \mathsf{Ob}(\mathcal M)_i \to \mathsf{Ob}(\mathcal M)_{i+8}$ and $U: \eqm(\alpha, \beta) \to \eqm(U\alpha, U\beta)$ preserving the source, target, and composition maps.
\end{definition}

\begin{remark}
These are called \emph{instanton} flow categories because of the degree-$8$ periodicity isomorphism $U$, which is the structure that appears in flow categories arising from the instanton theory to be discussed later in this section. Even though we will ultimately work modulo $U$, it is useful to include it in the definition so that indices of critical orbits are integers, and spaces $\eqm(\alpha, \beta)$ are of a constant dimension.
\end{remark}

The approach of \cite{DME1} associates with an $SO(3)$-flow category a large chain complex $\widetilde C^{gm}(\mathcal M;\mathbb F_2)$. In this article we prefer a finite-dimensional chain complex, and we use the cellular approach of \cite[Chapters 6-8]{DMES}. To do so, we fix the standard cell structure on $SO(3)$, with $0$-cell given by the identity, $1$-cell given by the nontrivial rotations around the $z$-axis, $2$-cell given by the set of rotations around some axis in the $xz$-plane, less the rotations around the $z$-axis, and $3$-cell given by the remainder. For this cell structure, both multiplication and inversion are cellular maps. We also fix the standard cell structure on $S^2$ with two cells. For this cell structure, the standard action $\mu: SO(3)\times S^2 \to S^2$ is a cellular map. In particular, the $0$-cell is $SO(2)$-invariant. Thus the cellular chain complex of $SO(3)$ is a dg-algebra and the cellular chain complex of $S^2$ a dg-module. With coefficients in $\mathbb F_2$, these are identified with the exterior algebra $\Lambda(\chi_1, \chi_2)$ and its module $\Lambda(\chi_2)$. 

\begin{definition}
	A \textit{pointed instanton $SO(3)$-flow category} $(\cM, p)$ is an instanton $SO(3)$-flow category for which each orbit is equipped with a choice of basepoint $p_\alpha\in\lift(\alpha)$ satisfying the following conditions:
	\begin{itemize}
		\item $Up_\alpha = p_{U\alpha}$.
		\item If $\lift(\alpha)\cong S^2$, the basepoint $p_\alpha$ is $SO(2)$-invariant.
	\end{itemize}
The pointing $p$ of $\cM$ gives rise to an equivariant identification $\lift(\alpha)\cong SO(3)/\Gamma, \Gamma\in\{1, SO(2), SO(3)\}$ for each orbit, and hence a privileged cell structure on $\lift(\alpha)$ for which $p_\alpha$ is its unique $0$-cell.
\end{definition} 

Given a pointed instanton $SO(3)$-flow category $(\cM,p)$,
for $\lift(\alpha) \leftarrow \eqm(\alpha, \beta) \to \lift(\beta)$ to induce a map on cellular chains, we need to impose an additional condition on this correspondence. Suppose $X$ is a finite-dimensional CW-complex, $W$ is a space equipped with a filtration by subspaces, and $X \xleftarrow{s} W$ is a continuous map. For each cell $e: D^n \to X$ of codimension $j$, define $W_e = D^n \times_X W$ and equip it with the filtration $W_e^{(k)} = D^n \times_X W^{(k+j)} \cup S^{n-1} \times_X W^{(k+j+1)}$.
	
The space $\eqm(\alpha, \beta)$ is filtered by $\eqm(\alpha, \beta)^{(k)}$. Because the map $s_{\alpha \beta}$ is a smooth fiber bundle projection, for each cell $e$ of $\lift(\alpha)$ the space $\eqm(\alpha, \beta)_e$ is again a stratified-smooth manifold. 
	
\begin{definition}
A \textit{cellular correspondence} consists of the data of CW-complexes $X$ and $Y$, a filtered space $W$, and a correspondence $X \xleftarrow{s} W \xrightarrow{t} Y$ with the following property: for each cell $e$ of $X$, the map $t: W_e \to Y$ is filtered.
\end{definition}

When $W$ is a stratified-smooth manifold and $s$ is a smooth fiber bundle projection, $W$ induces a map on cellular chains by the formula \[W_*(e) = t_*[W_e] \in C_*^{\text{cell}}(Y).\] It is a consequence of Lemma \ref{Lemma:bdry-relation} that $W_*$ satisfies the boundary relation 

\begin{equation}\label{eqn:bdry-reln-corresp}
	\partial^{\text{cell}} W_* + W_* \partial^{\text{cell}} = (\partial W)_*.
\end{equation}

If $X \leftarrow V \to Y$ and $Y \leftarrow W \to Z$ are cellular correspondences, and $V \times_Y W$ is equipped with the filtration $(V \times_Y W)^{(k)} = \bigcup_{i+j = k + \dim Y} V^{(i)} \times_Y W^{(j)}$, it is straightforward to verify that $V \times_Y W$ is also a cellular correspondence. When $V, W$ are stratified-smooth manifolds with $V \to X$ and $W \to Y$ smooth fiber bundle projections, the fiber product $X \leftarrow V \times_Y W \to Z$ is a cellular correspondence of stratified-smooth manifolds, with source map a smooth fiber bundle projection, and the induced map satisfies 
\begin{equation}\label{eqn:fiber-product-correspondence}
	(V \times_Y W)_* = W_* \circ V_*.
\end{equation}

We will only consider the case that $X, Y$ are $SO(3)$-orbits equipped with the standard cell structure and $W$ is a stratified-smooth $SO(3)$-manifold equipped with its corner filtration, with $s$ and $t$ equivariant. In this case, the hypothesis that the correspondence is cellular can be stated more simply: if $p \in X$ is the unique $0$-cell, then $(W,s,t)$ defines a cellular correspondence if and only if the map $t: W_p \to Y$ is cellular. For further details on the definition of cellular correspondences and their induced maps, see \cite[\S6.2]{DMES}.

\begin{definition}
A pointed instanton $SO(3)$-flow category $(\cM, p)$ is said to be \textit{cellular} if each correspondence $\lift(\alpha) \xleftarrow{s_{\alpha \beta}} \eqm(\alpha, \beta) \xrightarrow{t_{\alpha \beta}} \lift(\beta)$ is cellular.
\end{definition}

We will write $F(e_\alpha, \beta)$ for the fiber product $e_\alpha \times_\alpha \eqm(\alpha, \beta)$. The preceding construction gives rise to a map on cellular chains:
\begin{align*}
{\mathfrak m}_{\alpha \beta}: C_*^{\text{cell}}(\lift(\alpha)) &\to C_*^{\text{cell}}(\lift(\beta)) \\
{\mathfrak m}_{\alpha \beta}(e_\alpha) &= (t_{\alpha \beta})_*[F(e_\alpha, \beta)].
\end{align*}
These maps are equivariant under the action of $C_*^{\text{cell}}(SO(3))$ and satisfy the relations 
\begin{equation}\label{eqn:cell-diff-bdry} 
\partial^{\text{cell}} {\mathfrak m}_{\alpha \beta} + {\mathfrak m}_{\alpha \beta} \partial^{\text{cell}} = \sum_{\gamma \in \mathsf{orb}(\mathcal M)}  {\mathfrak m}_{\gamma \beta} {\mathfrak m}_{\alpha \gamma},
\end{equation} 
owing to the relations \eqref{eqn:flowcat-boundary-relation}, \eqref{eqn:bdry-reln-corresp} and \eqref{eqn:fiber-product-correspondence}; see \cite[Lemmas 6.2.4 and 6.2.5]{DMES} for more details. In the present article we take coefficients in $\mathbb F_2$, so that the left-hand side of this equation is identically zero. We also remark that the inequalities in \eqref{eqn:dim-restriction} are imposed because this is the dimension range in which $\fm_{\alpha \beta}$ and $\fm_{\gamma \beta} \fm_{\alpha \gamma}$ are not vacuously zero: for stratified-smooth manifolds of larger dimension, these maps take values in $C_k^{\text{cell}}(\lift(\beta))$ for $k > \dim \lift(\beta)$.

Thus, given a cellular instanton $SO(3)$-flow category $\mathcal M$, we may define a chain complex over $\mathbb F_2[U, U^{-1}]$ 
\begin{align*}
\widetilde C(\mathcal M; \Delta_{\mathbb F_2}) &= \bigoplus_{\alpha \in \mathsf{orb}(\mathcal M)} C_*^{\text{cell}}(\lift(\alpha))[-|\alpha|] \\ 
\widetilde d(e_\alpha) &= \partial^{\text{cell}}(e_\alpha) + \sum_{\beta \in \mathsf{orb}(\mathcal M)} {\mathfrak m}_{\alpha \beta}(e_\alpha).
\end{align*}
Here we use the convention that $C_i[k]=C_{i+k}$ for a graded chain complex $C_*$.
Again, the cellular differential is identically zero because we take coefficients in $\mathbb F_2$. That this differential squares to zero is an immediate consequence of \eqref{eqn:cell-diff-bdry}, and that it gives $\widetilde C$ the structure of a dg-module follows from the fact that ${\mathfrak m}$ is equivariant. The action of $U$ is defined by the periodicity operator in the definition of instanton $SO(3)$-flow category. For the purposes of this paper, it is convenient to restrict attention to the $\mathbb Z/8$-graded dg-module 
\begin{equation}\label{eqn:quotienting-U-action}
	\widetilde C(\mathcal M) = \widetilde C(\mathcal M; \Delta_{\mathbb F_2}) \otimes_{\mathbb F_2[U, U^{-1}]} \mathbb F_2,
\end{equation} 
which is essentially equivalent to quotienting $\mathsf{orb}(\mathcal M)$ by the $U$-action. 

We may make the differential more explicit, as follows. Given a cellular instanton $SO(3)$-flow category, write $Z(\mathcal M)$ for the graded $\mathbb F_2$-vector space generated by central orbits, $A(\mathcal M)$ for the graded vector space generated by abelian orbits, and $C(\mathcal M)$ for the graded vector space generated by irreducible orbits, all considered modulo the periodicity operator $U$. Then
\begin{align}
\widetilde C(\mathcal M) &= C(\mathcal M) \otimes_{\mathbb F_2} \Lambda(\chi_1, \chi_2) \; \oplus \; A(\mathcal M) \otimes_{\mathbb F_2} \Lambda(\chi_2) \; \oplus \; Z(\mathcal M) 
\nonumber 
\end{align}
This induces the following direct sum decomposition of $\widetilde C(\mathcal M)$ as a vector space over $\F$:
 \begin{equation}\label{eqn:dg-module-decomposition}
	C(\mathcal M) \oplus \chi_1 C(\mathcal M) \oplus \chi_2 C(\mathcal M) \oplus \chi_3 C(\mathcal M) \oplus A(\mathcal M) \oplus \chi_2 A(\mathcal M) \oplus Z(\mathcal M),
 \end{equation}
where $\chi_3 = \chi_1 \chi_2$. In terms of this direct sum decomposition, the differential takes the form:
\begin{equation}\label{eq:so3complexdtildeshape}
\widetilde d  = \left[
\begin{array}{cccc|cc|c} 
d_1 & 0   & 0   & 0   & 0   & 0   & 0\\
d_2 & d_1 & 0   & 0   & e_2 & 0   & 0\\
d_3 & 0   & d_1 & 0   & 0   & 0   & 0\\
d_4 & d_3 & d_2 & d_1 & e_4 & e_2 & \delta_4\\
\hline
e_1 & 0   & 0   & 0   & r_1 & 0   & 0\\
e_3 & 0   & e_1 & 0   & r_3 & r_1 & s_3\\
\hline
\delta_1 & 0 & 0 & 0 & s_1 & 0 & t_1\\
\end{array} \right]
\end{equation}
Here, for instance, the map $d_i: C_*(\mathcal M) \to C_{*-i}(\mathcal M)$ is defined as
 \[\langle d_i \alpha, \beta\rangle = \langle {\mathfrak m}_{\alpha \beta}(e^0),e^{i-1}\rangle , \quad \quad \alpha, \beta \text{ irreducible},\] 
 where $e^{i-1}$ denotes the $(i-1)$-dimensional cell of $SO(3)$.
The components $r_1, r_3, s_1, s_3, t_1$ can only be nonzero if some point in $\eqm(\alpha, \beta)$ has nontrivial stabilizer; for the case of $r_3$, compare \cite[Proposition 7.2.4]{DMES}.

Instanton $SO(3)$-flow categories that arise in nature are not typically cellular, but we can make them so: 

\begin{prop}\label{prop:flowcat-cell-approx}
Suppose $(\mathcal M, p)$ is a pointed instanton $SO(3)$-flow category for which the $SO(3)$-action on each $\eqm(\alpha, \beta)$ is free. Then there exists a coherent homotopy of the right endpoint maps $t_{\alpha \beta}^r: [0,1] \times \eqm(\alpha, \beta) \to \lift(\beta)$ so that $(\eqm(\alpha, \beta), s_{\alpha \beta}, t^1_{\alpha \beta})$ are the correspondences of a cellular instanton $SO(3)$-flow category.
\end{prop}

We sketch the proof of this fact below; full details are provided in \cite[\S7.4]{DMES}. A similar claim holds for uniqueness: two cellular approximations are themselves coherently homotopic by a cellular homotopy. 

\begin{proof}[Sketch of proof.]
A flow category also includes the data of composition maps, which are used in the induction procedure. The domain of the composition map is given by the fiber product of $t_{\alpha \gamma}$ and $s_{\gamma \beta}$. In modifying $t_{\alpha \gamma}$ by a homotopy, the domain varies. Thus, one should also inductively construct a map \[\circ^r: \eqm(\alpha, \gamma) \mathbin{{}_{t^r_{\alpha \gamma}}\!\times_{s_{\gamma \beta}}} \eqm(\gamma, \beta)\to\partial\eqm(\alpha,\beta),\;\;r\in[0,1]\] so that, for each fixed time parameter $r$, the map $\circ^r$ is an embedding which determines a diffeomorphism as in \eqref{eqn:flowcat-boundary-relation}. Furthermore, $\circ^r$ should restrict to the existing composition maps on lower strata, so that the map is already specified for $r = 0$ and along the boundary of the domain. Because $s_{\gamma \beta}$ is a fiber bundle projection, the domain is a fiber bundle over $[0,1]$. As a consequence, the desired extension can be constructed by trivializing this bundle.

Write $F(\alpha, \beta) = s_{\alpha \beta}^{-1}(p_\alpha)$; this space carries an action of $\Gamma_\alpha$, the stabilizer of $p_\alpha$. Inducting on the dimension of $F(\alpha, \beta)$, we seek to find an equivariant homotopy from the map $t_{\alpha \beta}: F(\alpha, \beta) \to \lift(\beta)$ to a cellular map, equal to a homotopy on its boundary determined by the inductive hypothesis and the identification of $\partial F(\alpha, \beta)$ with the disjoint union of fiber products of $F(\alpha, \gamma)$ with $\eqm(\gamma, \beta)$, the fiber product being taken with respect to the composition maps $\circ^r$ described above. (That $t^r$ is prescribed in this way on the boundary is precisely what we mean by the term `coherent homotopy'.) When $\Gamma_\alpha$ is trivial, this is classical cellular approximation. When $\Gamma_\alpha = SO(3)$, the claim is tautological for reasons of dimension: $\dim F(\alpha, \beta) \ge 3$, while the codomain has dimension at most $3$, so any extension of the given homotopy will do. When $\Gamma_\alpha = SO(2)$ and $\dim F(\alpha, \beta) \ge 3$, the result is again tautological. If $\dim F(\alpha, \beta) \le 2$, the domain may be identified with $SO(2) \times C$ for a compact manifold $C$ of dimension at most $1$; by inductive hypothesis and the fact that $SO(2) \cdot \lift(\beta)^{(0)} = \lift(\beta)^{(1)}$, this decomposition may be chosen so that $t_{\alpha \beta}(\partial C) = p_\beta$. Applying classical cellular approximation to $t_{\alpha \beta}$ restricted to $C$, and extending in the obvious way to $SO(2) \times C$, gives the desired homotopy. 
\end{proof}

In particular, if the $SO(3)$-action on each $\eqm(\alpha, \beta)$ is free, there is a corresponding dg-module $\widetilde C(\mathcal M)$ for each pointing data $p$, well-defined up to equivariant homotopy equivalence, for which $r_1, r_3, s_1, s_3, t_1$ are all zero. Moreover, the homotopy equivalence class of $\widetilde C(\cM)$ is independent of the choice of basepoints on the irreducible orbits: if $p$ and $p'$ agree except possibly on free orbits, then any cellular approximations of $(\cM,p)$ and $(\cM,p')$ can be connected by a coherent cellular homotopy. However, because $S^2$ has exactly two $SO(2)$-invariant points, the choice of basepoints on each abelian orbit is a nontrivial extra piece of information. We will return to this point soon.
	
One may relax the hypothesis that the $SO(3)$-action is free in Proposition \ref{prop:flowcat-cell-approx}. The only obstruction to achieving cellularity arises in the case that $\alpha, \beta$ are abelian, $\eqm(\alpha, \beta)$ contains an abelian trajectory and that $\dim\eqm(\alpha,\beta)\le\dim\lift(\alpha)+1$. If no such trajectory exists, or if the obstruction vanishes for each such trajectory, then the $SO(3)$-flow category may be made cellular.\\

Suppose $Y$ is a rational homology sphere with a chosen basepoint $y$, equipped with a regular perturbation $\pi$ of the Chern--Simons functional associated with the trivial $SU(2)$-bundle. Then $(Y, \pi)$ gives rise to an instanton $SO(3)$-flow category as follows. 

The group $\mathcal G(Y)$ acts freely on $SU(2) \times \mathcal A(Y)$ via its standard action on $\mathcal A(Y)$, and on the $SU(2)$-factor by left multiplication by the value of elements of $\mathcal G(Y)$ at the basepoint $y$. Its identity component, denoted $\mathcal G_0(Y)$, consists of gauge transformations $\sigma: Y \to SU(2)$ of degree zero. We write $\underline{\widetilde {\mathcal B}}(Y)$ for the quotient space by the action of $\mathcal G_0(Y)$, called the {\it configuration space of framed connections}. The right action of $SU(2)$ on the first factor of $SU(2) \times \mathcal A(Y)$ commutes with the action of $\mathcal G(Y)$ and induces an action of $SO(3)=SU(2)/{\pm I}$ on $\underline{\widetilde {\mathcal B}}(Y)$. With respect to this action, irreducible orbits correspond to connections with nonabelian holonomy group, abelian orbits correspond to connections with abelian but noncentral holonomy group, and central orbits correspond to connections with holonomy in $\{\pm I\}$. 
	
The objects $\mathsf{Ob}(Y,\pi)$ of the $SO(3)$-flow category associated with $(Y,\pi)$ are the critical points of $CS+\pi$ on $\underline{\widetilde{\mathcal B}}(Y)$. The interior of the manifold $\eqm(\alpha, \beta)$ is the space of pairs $(g, A)$, where $A$ is a (perturbed) ASD connection on $\mathbb R \times Y$ asymptotic to the (perturbed) flat connections $\alpha, \beta$ at $\pm \infty$ and $g \in SU(2)$. These are considered modulo the action of asymptotically parallel gauge transformations, with $\sigma$ acting on the first factor by the value $\sigma(y,0)$, and modulo translation. The source and target maps are given by sending $(g, A)$ to the pair $(g_\pm, A_\pm)$, where $g_\pm$ is the $A$-parallel transport of $g$ along $(-\infty, 0] \times \{y\}$ or $[0,\infty) \times \{y\}$ and $A_\pm$ is the restriction of $A$ to $\{\pm \infty\} \times Y$. The moduli spaces $\eqm(\alpha, \beta)$ are obtained as a compactification by broken trajectories. The lower strata are determined by a sequence of flat connections $\gamma_i$, and are given by trajectories which break along the $\gamma_i$. We denote these moduli spaces $\eqm(Y; \alpha, \beta)$ when we wish to emphasize the $3$-manifold $Y$.  

There is a grading function $|\alpha|$ on the set of critical orbits, defined using a convention similar to Section \ref{Fro-approach}; it is chosen so that both of the following hold: \[\dim \eqm(\alpha,\beta) = \dim \lift(\alpha) + |\alpha| - |\beta| - 1, \quad \quad	|\theta| = 0.\] 
The periodicity isomorphism $U: \mathsf{Ob}(Y,\pi)_i \to \mathsf{Ob}(Y,\pi)_{i+8}$ is given by the action of a generator $U \in \pi_0 \mathcal G(Y)$, or equivalently, a degree-$1$ gauge transformation $\sigma: Y \to SU(2)$. 

In general, these moduli spaces need not be compact because of bubbling phenomena. But assuming the inequality
\begin{equation}\label{eqn:bubbling-dim-res}
	|\alpha| - |\beta| + \dim \lift(\alpha) < 8 + \max(\dim \lift(\alpha), \dim \lift(\beta)),
\end{equation}
which is satisfied in the degree range of \eqref{eqn:dim-restriction}, the moduli spaces $\eqm(\alpha, \beta)$ of possibly broken framed ASD connections are compact. 

Finally, because a reducible ASD connection on $\mathbb R \times Y$ is gauge-equivalent to a constant trajectory, the moduli spaces $\eqm(\alpha, \beta)$ consist of irreducible connections, and therefore carry a free $SO(3)$-action.

It will be useful to explicitly enumerate the abelian and central flat connections modulo the periodicity isomorphism $U$, denoted $\mathfrak A(Y)$ and $\mathfrak Z(Y)$, respectively. A reducible connection gives rise to a splitting of the trivial $SU(2)$-bundle into a direct sum of complex line bundles, and taking the first Chern class of this decomposition gives a bijection 
\begin{align*} 
	\mathfrak A(Y) &\cong \left\{\{x, -x\} \subset H^2(Y;\mathbb Z) \mid 2x \ne 0\right\} \\ \mathfrak Z(Y) &\cong \{x \in H^2(Y;\mathbb Z) \mid 2x = 0\}.
\end{align*}

In addition to the choice of perturbation, to construct the dg-module of $Y$, we also need to choose a basepoint on each orbit. In the next section, we will discuss the precise dependence on this data. It turns out that the complex depends up to equivariant homotopy only on the following:

\begin{definition}
Let $Y$ be a rational homology sphere. \textit{Abelian data} on $Y$ consists of a pair $\mathfrak a = (\sigma, p)$, where \[\sigma: \mathfrak A(Y) \to 2\mathbb Z, \quad p: \mathfrak A(Y) \to H^2(Y;\mathbb Z)\] are functions satisfying the following hypotheses: 
\begin{itemize}
\item For $\alpha = \{x, -x\} \in \mathfrak A(Y)$, we demand $\sigma(\alpha) \equiv \dim_{\mathbb R} H^1(Y; \mathbb C_{2x}) \mod 4$. 
\item For $\alpha = \{x, -x\} \in \mathfrak A(Y)$, we demand $p(\alpha) \in \{x, -x\}$.
\end{itemize}
A class $z\in H^2(Y;\mathbb Z)$ determines a flat complex line bundle on $Y$, and hence a local system $\mathbb C_z$; here $\mathbb C_{2x}$ denotes the bundle corresponding to $2x$.
\end{definition}

The function $\sigma$ is called the \emph{signature data}. Given a regular perturbation $\pi$ of the Chern--Simons functional, there is an associated signature data function $\sigma_\pi$, which records spectral flow information at each abelian orbit. The precise definition is such that when $\pi$ is a small perturbation, $\sigma_\pi(\alpha)$ is the signature of $\text{Hess}_\alpha(\pi)$ restricted to $\ker \text{Hess}_\alpha(CS)$ \cite[Definition 3.2]{DME1}. The second component of the abelian data corresponds to a choice of $SO(2)$-fixed basepoint on each abelian orbit.

Because the $SO(3)$-action on $\eqm(\alpha, \beta)$ is free, Proposition \ref{prop:flowcat-cell-approx} applies, and we may modify the endpoint maps to obtain a cellular instanton $SO(3)$-flow category $\cM(Y, \pi, p)$. Thus, a choice of regular perturbation $\pi$ and function $p$ gives rise to a dg-module, well-defined up to equivariant homotopy equivalence, in which the components $r_1, r_3, s_1, s_3, t_1$ of the differential are identically zero. It will be argued in the next section that the resulting homotopy type depends on $\pi$ only up to $\sigma_\pi$, so we will indicate this complex by $\widetilde C(Y, \mathfrak a)$. 

\begin{remark}\label{QHS^3-to-ZHS^3}
For an integer homology sphere$Y$, the complex $\widetilde C(Y, \mathfrak a)$ discussed above is the complex constructed in Example \ref{S-cplx:integer-sphere}. Note that in this case $\mathfrak a$ is vacuous because $\mathfrak A(Y)$ is empty. This explains the notation chosen in that example, especially for the operators $\chi_1$ and $\chi_2$. A detailed proof is given in \cite[\S8.3]{DMES}. A weaker version of this claim, which does not address the operator $d_4$, is established as Proposition \ref{prop:froy-vs-cell} below.
\end{remark}

\subsection{Bimodules and dg-module maps}\label{subsec:bimod}
To study the dependence of $\widetilde C(Y,\fa)$ on different abelian data, we need to consider chain maps between these dg-modules. To do so, and to upgrade these dg-modules to functorial invariants under cobordism maps, we need the following definition \cite[Chapter 6]{DMES}.

\begin{definition}
Suppose $\mathcal M, \mathcal M'$ are instanton $SO(3)$-flow categories. An \textit{instanton $SO(3)$-bimodule} $W: \mathcal M \to \mathcal M'$ of degree $d$ consists of the following data. For each $\alpha \in \mathsf{orb}(\mathcal M)$ and $\alpha' \in \mathsf{orb}(\mathcal M')$ with \begin{equation}\label{eqn:dim-res-2}|\alpha| - |\alpha'| + d \le \dim \lift(\alpha') + 1,\end{equation} 
there is a compact stratified-smooth $SO(3)$-manifold of dimension $|\alpha| - |\alpha'| + \dim \lift(\alpha)+ d$, denoted $\eqm(W;\alpha, \alpha')$. As in the case of flow categories, we demand that each element of $\eqm(W; \alpha, \alpha')$ has stabilizer conjugate to one of $\{I\}, SO(2)$, or $SO(3)$. This manifold is equipped with equivariant maps 
\[\lift(\alpha) \xleftarrow{s_{\alpha\alpha'}} \eqm(W;\alpha, \alpha') \xrightarrow{t_{\alpha\alpha'}} \lift(\alpha')\]
called the source and target maps. Whenever $\alpha, \alpha'$ satisfy \eqref{eqn:dim-res-2}, there are equivariant composition maps 
\begin{align*}
\circ: \eqm(\alpha, \beta) \times_{\beta} \eqm(W; \beta, \alpha') &\to \partial\eqm(W; \alpha, \alpha') \\
\circ: \eqm(W; \alpha, \beta') \times_{\beta'} \eqm(\beta', \alpha') &\to \partial\eqm(W; \alpha, \alpha')
\end{align*}
which are compatible with the source and target maps, and associate with the composition maps of $\mathcal M$ and $\mathcal M'$. We also require that these induce a diffeomorphism 
\begin{align}
\partial \eqm(W; \alpha, \alpha') \cong &\bigsqcup_{\substack{\beta \in \mathsf{orb}(\mathcal M)}} \eqm(\alpha, \beta) \times_{\beta} \eqm(W; \beta, \alpha') \label{eqn:bimodule-boundary-1}\\
&\bigsqcup_{\substack{\beta' \in \mathsf{orb}(\mathcal M')}} \eqm(W;\alpha, \beta') \times_{\beta'} \eqm(\beta', \alpha').\label{eqn:bimodule-boundary-2}
\end{align}
Finally, there is an equivariant diffeomorphism $U: \eqm(W; \alpha, \alpha') \to \eqm(W; U\alpha, U\alpha')$ which is compatible with the source, target, and composition maps defined above.
\end{definition}

Cellular $SO(3)$-bimodules can be defined in the same spirit as cellular $SO(3)$-flow categories. Suppose $W:\mathcal M \to \mathcal M'$ is an $SO(3)$-bimodule between cellular $SO(3)$-flow categories. We say $W$ is \textit{cellular} if, for any cell $e_\alpha$ of $\lift(\alpha)$ and $F(W; e_\alpha, \alpha') = e_\alpha \times_{\alpha} \eqm(W; \alpha, \alpha')$, the map $t^W_{\alpha \alpha'}: F(W; e_\alpha, \alpha') \to \lift(\alpha')$ is cellular. This follows for all cells as soon as it holds for the $0$-cell $p_\alpha$.

Given an instanton $SO(3)$-bimodule $W$, we define a collection of maps $\fm^W_{\alpha\alpha'}$ as follows. For any $\alpha\in\mathsf{orb}(\cM),\; \alpha'\in\mathsf{orb}(\cM')$, set
\begin{align*}
	\fm^W_{\alpha\alpha'}\colon C^\text{cell}(\lift(\alpha))&\to C^\text{cell}(\lift(\alpha')) \\
	\fm^W_{\alpha\alpha'}(e_\alpha) &= (t_{\alpha\alpha'})_*[F(W;e_\alpha,\alpha')].
\end{align*}
These maps are $SO(3)$-equivariant in the sense that $\fm^W_{\alpha\alpha'}(\chi\cdot p_\alpha) = \chi\cdot\fm^W_{\alpha\alpha'}(p_\alpha)$ for any $\chi\in C_*^\text{cell}(SO(3))$, and they satisfy the following boundary relation:
\begin{equation}\label{eqn:bimodule-boundary-relation}
	\partial^\text{cell}\fm^W_{\alpha \alpha'} + \fm^W_{\alpha\alpha'}\partial^\text{cell} = \sum_{\beta\in\mathsf{orb}(\cM)}\fm^W_{\beta\alpha'}\fm_{\alpha\beta}+\sum_{\beta'\in\mathsf{orb}(\cM')}\fm_{\beta'\alpha'}\fm^W_{\alpha\beta'}
\end{equation}
The first and second terms on the right are related to boundary components \eqref{eqn:bimodule-boundary-1} and \eqref{eqn:bimodule-boundary-2}, respectively. Since we are working over the field $\Bbb F_2$ throughout, the left-hand side of this equation is identically zero. 

The dg-module homomorphism associated with $W:\cM\to\cM'$ is defined by 
\begin{align*}
	\widetilde\lambda_W: \widetilde C(\cM)&\to \widetilde C(\cM') \\
	\widetilde\lambda_W(e_\alpha) &= \sum_{\alpha'}\fm^W_{\alpha\alpha'}(e_\alpha),
\end{align*}
and in terms of the decomposition \eqref{eqn:dg-module-decomposition}, takes the following form:
\begin{equation}\label{eq:so3morphism-shape}
\widetilde \lambda_W  = \left[
\begin{array}{cccc|cc|c} 
\lambda_0 & 0   & 0   & 0   & 0   & 0   & 0\\
\lambda_1 & \lambda_0 & 0   & 0   & \mu_1 & 0   & 0\\
\lambda_2 & 0   & \lambda_0 & 0   & 0   & 0   & 0\\
\lambda_3 & \lambda_2 & \lambda_1 & \lambda_0 & \mu_3 & \mu_1 & \Delta_3\\
\hline
\mu_0 & 0   & 0   & 0   & \nu_0 & 0   & 0\\
\mu_2 & 0   & \mu_0 & 0   & \nu_2 & \nu_0 & \xi_2\\
\hline
\Delta_0 & 0 & 0 & 0 & \xi_0 & 0 & \epsilon_0\\
\end{array} \right]
\end{equation}
In this matrix, an operator with subscript $i$ is induced by correspondences $\eqm(W;\alpha,\alpha')$ of dimension $\dim\lift(\alpha)+i$; equivalently, the dimension of $s_{\alpha \alpha'}^{-1}(p_\alpha) = F(W;\alpha,\alpha')$ is equal to $i$. The component $\varepsilon_0$ counts isolated points in $0$-dimensional moduli spaces $\eqm(W;\theta,\theta')$ with central orbits on both ends, and will play a special role later. The fact that $\widetilde\lambda_W$ is a dg-module homomorphism consists of two relations:
\[
	\widetilde d'\widetilde\lambda_W = \widetilde\lambda_W\widetilde d, \qquad \chi\widetilde\lambda_W = \widetilde\lambda_W\chi, \;\;\; \text{ for any }\chi\in C_*^\text{cell}(SO(3)),
\]
which follow immediately from relation \eqref{eqn:bimodule-boundary-relation} and $SO(3)$-equivariance of $\fm$.

\bigskip

Under certain circumstances, the preceding machinery can be applied to instanton Floer theory to obtain dg-module morphisms between the dg-modules associated with rational homology spheres. Suppose $(Y, \pi, p)$ is a rational homology sphere with a chosen basepoint $y$, a regular perturbation of the Chern--Simons functional $\pi$, and a choice of basepoint on each abelian orbit. Let $(Y', \pi', p')$ be another such tuple. Suppose $W: Y \to Y'$ is an oriented cobordism with $b_1(W) = 0$, equipped with a path $\gamma$ connecting the basepoints $y$ of $Y$ and $y'$ of $Y'$, cylindrical in a neighborhood of the boundary. Let $W^+$, $\gamma^+$ be defined as in Section \ref{Fro-approach} by adding cylindrical ends to $W$, $\gamma$.

We move on to constructing a bimodule between $\cM(Y, \pi, p)$ and $\cM(Y', \pi', p')$. Choose an extension of $\pi, \pi'$ to a perturbation $\pi_W$ over $W$. Given (perturbed) flat connections $\alpha$ and $\alpha'$, the interior of $\eqm(W;\alpha,\alpha')$ consists of pairs $(g, A)$, where $A$ is a (perturbed) ASD connection on $E = SU(2)\times W^+$ asymptotic to $\alpha, \alpha'$ at the corresponding ends and $g \in SU(2)$. These pairs are considered modulo the action of asymptotically parallel gauge transformations, with $\sigma$ acting on the first factor by the value $\sigma(\gamma(0))$. The source and target maps are determined by the $A$-parallel transport of $g$ along $\gamma$, and the space $\eqm(W; \alpha, \alpha')$ is the broken-trajectory compactification of this space. Its lower strata are determined by a sequence of flat connections on $Y$ and a sequence of flat connections on $Y'$, and consists of trajectories which break at these flat connections. The periodicity isomorphism $U$ is determined by a choice of gauge transformation $\sigma: W^+ \to SU(2)$ which is asymptotically equal to a degree-$1$ gauge transformation on each end.

If we can choose $\pi_W$ so that all of the instanton moduli spaces are cut out transversely, this construction produces an instanton $SO(3)$-bimodule, and $W$ is said to be \textit{unobstructed}. It is straightforward to arrange that the irreducible locus is cut out transversely, but there are topological obstructions to ensuring the reducible locus is cut out transversely. For central ASD connections, we must have $b^+(W) = 0$. If $A$ is an abelian connection, the adjoint bundle $\mathfrak{su}(2) \times W^+$ admits an $A$-parallel splitting $\mathfrak{su}(2) \cong \mathbb R \oplus L$, and we require that the index of the deformation operator restricted to $L$ is nonnegative. This \textit{normal index} can be computed in terms of the cohomology ring of $W$, the Atiyah--Patodi--Singer $\rho$ invariants of the restrictions of $A$ to the ends, and the signature data $\sigma_\pi$ of these restrictions \cite[Proposition 3.23]{DME1}. 

If we allow $\sigma_\pi$ to vary, this quantity can be ensured nonnegative except when $A$ is a flat abelian connection whose restriction to the ends is central \cite[Lemma 3.32]{DME1}, which necessarily has normal index equal to $-2$. Such abelian connections are called \emph{pseudocentral}. Pseudocentral connections are in one-to-one correspondence with elements of \[\{x\in \text{Tors}\;H^2(W;\Bbb Z)\mid 2x \neq 0 \text{ and } 2x|_{\partial W} = 0\},\] hence are unavoidable even for $\Bbb F_2$-homology cobordisms. The above discussion leads to the following definition.

\begin{definition}
A cobordism $W: (Y,\pi,p)\to (Y', \pi', p')$ with $b_1(W) = b^+(W) = 0$ is said to be \textit{pseudo-unobstructed} if there exists a perturbation $\pi_W$ so that all abelian ASD connections in $E$ are cut out transversely except for the pseudocentral ones.
\end{definition}

For a pseudo-unobstructed cobordism, one can choose a perturbation $\pi_W$ so that the moduli spaces $\eqm(W;\alpha, \alpha')$ have the structure of stratified-smooth manifolds away from the strata consisting of trajectories broken through pseudocentral connections; however, they are not manifolds near the pseudocentral strata. This issue is already observed in \cite[Proposition 2.13]{Don}. It can be remedied using the \textit{obstructed gluing theory} pioneered by Taubes \cite{Taubes-indefinite}; this is carried out in \cite[Section 6.1]{DME1}. To do so, one carefully deletes a neighborhood of these obstructed abelian connections from each moduli space and glues back additional pieces with desirable boundaries. This procedure, as is carried out in \cite[Proposition 6.3]{DME1}, gives rise to the desired bimodule between instanton $SO(3)$-flow categories.

\begin{remark}
	The spaces obtained by the procedure above are not necessarily stratified-smooth manifolds. They fall in the broader class of stratified-smooth spaces, whose definition will be explained in Section \ref{subsec:susp-chains}. The bimodule of \cite[Proposition 6.3]{DME1}, and thus the resulting chain map, is constructed at the level of geometric chains. To obtain a construction at the level of cellular chains, one may follow the recipe presented in Section \ref{sec:susp}.
\end{remark}

When $W$ is pseudo-unobstructed, we obtain an instanton $SO(3)$-bimodule between the flow categories $\cM(Y, \pi, p)$ and $\cM(Y', \pi', p')$. However, there may be obstructions to making this bimodule cellular. Suppose $\alpha, \alpha'$ are abelian flat connections on $Y$, $Y'$ with chosen basepoints $p_\alpha$, $p_{\alpha'}$ and $|\alpha| - |\alpha'| \le 1$, so that $\dim F(W; \alpha, \alpha') \le 1$. If $A \in F(W; \alpha, \alpha')$ is an abelian ASD connection with $s^W_{\alpha \alpha'}(A) = p_\alpha$, then we must have $t^W_{\alpha \alpha'}(A) = p_{\alpha'}$, or it is impossible to make the map $F(W; \alpha, \alpha') \to \lift(\alpha')$ cellular by any $SO(2)$-equivariant homotopy. Such ASD connections are called \emph{small antipodal flowlines}. As discussed after the proof of Proposition \ref{prop:flowcat-cell-approx}, they are the only obstruction to cellularity: if $W$ is pseudo-unobstructed and admits no small antipodal flowlines, then by a coherent homotopy of the endpoint maps we obtain a cellular instanton $SO(3)$-bimodule $W: \cM(Y, \pi, p) \to \cM(Y', \pi', p')$, well-defined up to cellular homotopy \cite[\S7.4]{DMES}.

\begin{definition}
A cobordism $W: (Y, \pi, p) \to (Y', \pi', p')$ is called \textit{nice} if $b_1(W) = b^+(W) = 0$, $W$ is pseudo-unobstructed, and $W$ admits no small antipodal flowlines. 
\end{definition}

A nice cobordism gives rise to a dg-module homomorphism $\widetilde C(Y, \pi, p)\to\widetilde C(Y', \pi', p')$, well-defined up to equivariant homotopy, and these morphisms compose functorially up to equivariant homotopy \cite[\S8.2]{DMES}. The following computations are straightforward:

\begin{itemize}
	\item The cobordism $I\times Y: (Y,\pi, p)\to (Y, \pi', p')$ is nice if and only if $\fa\le\fa'$, in the sense that $\sigma_\pi\le\sigma_{\pi'}$ and $\sigma_\pi(\alpha)=\sigma_{\pi'}(\alpha)$ implies $p(\alpha)=p'(\alpha)$.
	\item For any cobordism $W: Y\to Y'$ with $b_1(W)=b^+(W)=0$, there is a choice of auxiliary data $(\pi, p), (\pi', p')$ so that $W: (Y, \pi, p)\to (Y', \pi', p')$ is nice.
\end{itemize}

As a consequence, we find that the equivariant homotopy type of $\widetilde C(Y, \pi, p)$ depends on $(\pi, p)$ only up to the underlying abelian data $\fa = (\sigma_\pi, p)$.\bigskip

The final item of business is the geometric description of $\epsilon_0$ in matrix \eqref{eq:so3morphism-shape}. Recall that $\mathfrak Z(Y)$ is identified with the set of $2$-torsion elements in $H^2(Y;\mathbb Z)$, which can further be identified with $H^1(Y;\mathbb F_2)$ via the integral Bockstein map. The following is the special case of \cite[Proposition 6.6]{DME1} in which the bundle data $c$ is empty.

\begin{lemma}\label{lemma:component-epsilon}
If $W: (Y, \mathfrak a) \to (Y', \mathfrak a')$ is a nice cobordism, and $\theta \in \mathfrak Z(Y) \cong H^1(Y;\mathbb F_2)$, the map $\epsilon_0$ is given by 
\begin{equation}
	\epsilon_0(\theta) = \sum_{\theta' \in \mathfrak Z(Y')} \left(\# \{\Theta \in H^1(W;\mathbb F_2) \mid \;\; \Theta|_Y = \theta, \;\; \Theta|_{Y'} = \theta'\}\right) \theta'.\label{eqn:epsilon-description}
\end{equation}
\end{lemma}

In particular, the map $\widetilde \lambda_{I \times Y}$ has $\epsilon_0$ given by the identity map. In the special case that $Y$, $Y'$ are $\mathbb F_2$-homology spheres, $\epsilon_0 \in \mathbb F_2$ is equal to $|H_1(W;\mathbb F_2)| \mod 2$, hence $1$ if $H_1(W;\mathbb Z)$ consists of odd torsion and $0$ if $H_1(W;\mathbb Z)$ has even torsion.

\subsection{Fr{\o}yshov's invariant}\label{subsec:QHS-Scpx}
In this section we apply the algebraic package of $\cS$-complexes reviewed in Section \ref{sec:S-complexes} to the dg-modules $\widetilde C(\mathcal M)$ and thus $\widetilde C(Y, \fa)$ constructed in the preceding sections. As a consequence, we define an extension of the invariant $q_3(Y)$ to rational homology spheres.\bigskip 

Suppose first that $\mathcal M$ is a pointed instanton $SO(3)$-flow category. Recall that we have a dg-module $\widetilde C(\mathcal M)$, well-defined up to equivariant homotopy. Since $\widetilde C(\mathcal M)$ is a finite-dimensional dg-module over $\Lambda(\chi_1, \chi_2)$, it is also an $\cS$-complex with respect to the action of $\chi_i, i \in \{1,2\}$. The direct sum decomposition \eqref{eqn:dg-module-decomposition} gives rise to a canonical splitting for the $\cS$-complex $(\widetilde C(\mathcal M), \chi_2)$ as $\widetilde C(\mathcal M) = C^{(3)} \oplus \chi_2 C^{(3)} \oplus Z$, where
\[
C^{(3)}(\mathcal M) = C(\mathcal M)\oplus\chi_1C(\mathcal M)\oplus A(\mathcal M),\qquad Z = Z(\mathcal M).
\]

In terms of the given splitting, the components of $\widetilde d$ as in \eqref{chi-dtilde-special-form} are given by
\[
	d = \begin{bmatrix}
		d_1 & 0 & 0\\
		d_2 & d_1 & e_2\\
		e_1 & 0 & 0
	\end{bmatrix},\;
	v = \begin{bmatrix}
		d_3 & 0 & 0\\
		d_4 & d_3 & e_4\\
		e_3 & 0 & 0
	\end{bmatrix},\;
	\delta = \begin{bmatrix}
		\delta_1 & 0 & 0
	\end{bmatrix},\;
	\delta' = \begin{bmatrix}
		0\\ \delta_4\\ 0
	\end{bmatrix},\; r=0.
\]
It is straightforward to check that $\delta v^{i-1}\delta' = 0$ for each $i\ge 1$, so we obtain a perfect $\cS$-complex in the sense of Definition \ref{def:perfectness}.

In Section \ref{sec:equivariant-homology-groups} we introduced the equivariant homology groups $\medhat{H}(\widetilde C), \medcheck H(\widetilde C), \medbar H(\widetilde C)$ of a perfect $\cS$-complex $\widetilde C$. These groups fit in the exact triangle \eqref{exact-tri-equiv}, and we use the homomorphism $\mathfrak i:\medhat H(\widetilde C)\to\medbar H(\widetilde C)$ to define a filtration $\{J_n(\widetilde C)\}_{n\in\Bbb Z}$ on $Z(\widetilde C)$. This is used to define, for each $\vartheta \in Z(\widetilde C)$,
\[
h(\widetilde C;\vartheta) = \max\{n\in\Bbb Z\mid \vartheta\in J_n(\widetilde C)\}.
\] 
For $\vartheta \ne 0$, this quantity is finite. Specializing to the case $(\widetilde C(\mathcal M), \chi_2)$, we write 
\[
	J^{(3)}_n(\mathcal M) = J_n(\widetilde C(\mathcal M), \chi_2), \quad q_3(\mathcal M; \vartheta) = h(\widetilde C(\mathcal M), \chi_2; \vartheta).
\]
In the special case that $\mathfrak Z(\mathcal M)$ is a singleton, there is a unique nonzero $\vartheta_0 \in Z(\mathcal M)$, and we set $q_3(\mathcal M) = q_3(\mathcal M; \vartheta_0)$.\\

We now attend to the case of rational homology spheres. A pair $(Y, \mathfrak a)$ of a rational homology sphere and abelian data gives rise to an instanton $SO(3)$-flow category $\mathcal M(Y,\mathfrak a)$, well-defined up to homotopy equivalence. Before proceeding, observe that the sets $\mathfrak A(Y)$ and $\mathfrak Z(Y)$ of abelian and central connections on $Y$ are independent of the choice of abelian data.

The preceding discussion immediately gives rise to several invariants. First, we define the homology group \[\widehat I^{(3)}(Y,\fa) = \medhat H(\widetilde C(Y,\fa), \chi_2),\] and similarly $\overline I^{(3)}(Y,\fa)$ and $\widecheck I^{(3)}(Y,\fa)$. These are related by a long exact sequence, which is used to define the filtration $J^{(3)}_n(Y,\fa)$ and integer quantities $q_3(Y,\fa; \vartheta)$ for $\vartheta \in Z(Y)$. 

In the special case that $Y$ is an $\Bbb F_2$-homology sphere, $Z(Y)$ is generated by the trivial connection $\vartheta_0$ of $Y$, and we simplify the notation $q_3(Y,\fa) = q_3(Y, \fa; \vartheta_0)$. 

We first address the case of integer homology spheres, for which there is no auxiliary data $\fa$, the vector space $A(Y)$ is zero, the vector space $Z(Y)$ is $1$-dimensional, and the splitting of $\widetilde C(Y)$ discussed above is precisely the one given in Example \ref{S-cplx:integer-sphere}. In this case, in addition to the invariant $q_3(Y)$ defined above, we had previously defined an invariant by the same name in Section \ref{Fro-approach}. They coincide:
	
\begin{prop}\label{prop:froy-vs-cell}
	If $Y$ is an integer homology sphere, the invariant $q_3(Y)$ defined above coincides with the invariant $q_3(Y)$ defined in Section \ref{Fro-approach}.
\end{prop}
\begin{proof}
Lemma \ref{lemma:explicit-h-copy} makes clear that it suffices to show, for an appropriate choice of sections and cellular approximations, that the maps $d_1, d_2, d_3, \delta_1, \delta_4$ defined in Section \ref{Fro-approach} coincide with the maps of the same name defined in Section \ref{subsec:flowcat}. Note that when $\alpha$ is irreducible, $F(\alpha, \beta) \cong \eqm(\alpha, \beta)/SO(3) = \breve M(\alpha, \beta)$ is identified with the moduli spaces considered in Section \ref{Fro-approach}, and when $\beta$ is also irreducible, the endpoint map $t_{\alpha \beta}$ is identified with the map $h_{\alpha \beta}: \breve M(\alpha, \beta) \to SO(3)$ given by taking the parallel transport along $\mathbb R \times \{y\}$ of the chosen framing $p_\alpha$, and comparing it to the chosen framing $p_\beta$; this is a holonomy map. The operators $d_1$ and $\delta_1$ in both cases count unframed instantons in $0$-dimensional moduli spaces. The operator $\delta_4$ in both approaches agree because when $\alpha$ is the trivial connection and $\beta$ is irreducible, $F(\alpha, \beta)=\eqm(\alpha, \beta)$ admits a free $SO(3)$ action, and the degree of  $t_{\alpha\beta}:\eqm(\alpha, \beta)\to SO(3)$ equal the number of elements in $\eqm(\alpha, \beta)/SO(3) = \breve M(\alpha, \beta)$. In the approach of Section \ref{Fro-approach}, the operators $d_2, d_3$ are constructed in terms of sections of the vector bundle associated with the basepoint fibration, and in the approach of this section in terms of cellular approximations of the holonomy maps. Thus, we are comparing the Stiefel--Whitney classes of $SO(3)$-bundles over $\mathbb R \times \breve M(\alpha, \beta)$ to the degrees of their cellular clutching functions $\breve M(\alpha, \beta) \to SO(3)$. These coincide, by an argument akin to the computation of \cite[Section 7.3.2]{DonBook}.
\end{proof}

Returning to the case of rational homology spheres, we next show that $q_3(Y, \fa; \vartheta)$ is independent of the choice of abelian data, so it can be understood as an extension of $q_3(Y)$ to the case of rational homology spheres.

\bigskip

Given a nice cobordism $W:(Y,\fa)\to(Y',\fa')$, we have a dg-module homomorphism $\widetilde\lambda_W:\widetilde C(Y,\fa)\to\widetilde C(Y', \fa')$. This gives rise to a height-$0$ morphism between $\cS$-complexes $(\widetilde C, \chi_2)$, well-defined up to $\chi_2$-equivariant homotopy. In particular, the component $\varepsilon: Z(Y)\to Z(Y')$ is well-defined. It was discussed under the name $\varepsilon_0$ in the previous sections, and is determined by homological information by Lemma \ref{lemma:component-epsilon}. Moreover, for the product cobordism $I \times Y:(Y,\fa)\to(Y,\fa')$ with $\fa\le\fa'$, the cobordism map $\widetilde\lambda_{I \times Y}$ induces an isomorphism on instanton homology groups, and in fact on all variations of equivariant instanton homology $I^\bullet$. This was established in \cite[Theorem 5.26]{DME1}, using an alternate model for $\widetilde C(Y, \fa)$ defined in terms of geometric chains. The resulting equivariant homology groups are canonically isomorphic to those of the cellular model $\widetilde C(Y, \fa)$. This is discussed in \cite[Chapter 8]{DMES}; a self-contained argument that $\lambda^\bullet_{I \times Y}$ is an isomorphism, presented in the language of cellular chains, is given in \cite[Chapter 9]{DMES}. This is closely related to the material we will discuss in Section \ref{sec:susp}.
	
Applying these facts to the product cobordism $I\times Y$, we have the following.
\begin{lemma}\label{lemma:abelian-data-independence}
If $Y$ is a rational homology sphere, the filtration $J_n^{(3)}(Y,\fa)$ on $Z(Y)$ is independent of the choice of $\fa$.
\end{lemma}

\begin{proof}
Given two choices of abelian data on $Y$ which satisfy $\mathfrak a \le \mathfrak a'$, the cobordism $I \times Y: (Y, \mathfrak a) \to (Y, \mathfrak a')$ is nice and has $\varepsilon = \id$. Because this cobordism induces an isomorphism on all equivariant homology groups, and by \eqref{tau-i-formulas} we have $\tau_0 = \varepsilon = \id$, it follows from Proposition \ref{prop:tau-filtered} that the identity map induces a filtration-preserving isomorphism $Z(Y) \to Z(Y)$. To complete the proof, it suffices to note that for every two choices of abelian data on $Y$, there is another choice which is larger than both.
\end{proof}

We thus omit $\fa$ from notation and write $\{J^{(3)}_n(Y)\}_{n\in\Bbb Z}$ for the filtration on $Z(Y)$, which is an invariant of $Y$. As an immediate consequence, the value $q_3(Y,\fa; \vartheta)$ is independent of $\fa$, and will be simply denoted by $q_3(Y;\vartheta)$. A monotonicity result generalizing Proposition \ref{lemma:monotonic} can be stated as follows.

\begin{prop}\label{prop:generalized-monotonicity}
	If $W:Y\to Y'$ is a cobordism with $b_1(W)=b^+(W)=0$, then the map $\varepsilon_0:Z(Y)\to Z(Y')$ defined as in \eqref{eqn:epsilon-description} satisfies $\varepsilon_0(J_n^{(3)}(Y)) \subset J_n^{(3)}(Y')$. In particular, for each $\vartheta\in Z(Y)$, one has $q_3(Y;\vartheta) \le q_3(Y';\varepsilon_0(\vartheta))$.
\end{prop}
\begin{proof}
	One may always choose abelian data $\fa, \fa'$ so that $W:(Y,\fa)\to (Y',\fa')$ is nice. Because the filtration $J_n^{(3)}(Y,\fa)$ is independent of $\fa$, the result follows from Proposition \ref{prop:tau-filtered}.
\end{proof}

We end this section with a discussion of duality. Suppose $Y$ is equipped with a perturbation $\pi$ of its Chern--Simons functional. If $-Y$ is equipped with the perturbation $-\pi$, the instanton moduli spaces on $\R \times Y$ are canonically identified with those on $\R \times -Y$ by time-reversal in the first factor, and the correspondences $\lift(\alpha) \leftarrow \eqm(Y; \alpha, \beta) \to \lift(\beta)$ are identified with $\lift(\beta) \leftarrow \eqm(-Y; \beta, \alpha) \to \lift(\alpha)$. If $Y$ is further equipped with a choice of basepoints $s$ and cellular approximation of its $SO(3)$-flow category, we may use the same data to furnish $-Y$ with a cellular approximation of its $SO(3)$-flow category. Under these identifications, the maps $d_i$ for $-Y$ are identified with the duals $d_i^\dagger$ of the maps for $Y$, while $e_1, e_3$ are identified with $e_2^\dagger, e_4^\dagger$; similarly $\delta_1, \delta_4$ are exchanged by duality. Thus, writing $-\fa = (-\sigma_\pi, s)$ for the dual abelian data on $-Y$, it follows that $\widetilde C(Y, \fa)^\dagger$ is identified with $\widetilde C(-Y, -\fa)$; this is presented as \cite[Theorem 7.3.3]{DMES}. This gives rise to an isomorphism $(\widetilde C(-Y,-\fa),\chi_2)\cong (\widetilde C(Y,\fa), \chi_2)^\dagger$ on the level of $\cS$-complexes. Note that $Z(Y)$ is equipped with a canonical basis $\fZ(Y)$, and the duality isomorphism identifies $Z(-Y)=Z(Y)$ with $Z(Y)^\dagger$ by sending each basis vector to its dual vector. The following result is then the specialization of Proposition \ref{prop:filtration-under-duality} and Corollary \ref{cor:h-under-duality}.

\begin{prop}\label{prop:J-3mfd-duality}
The filtration $J_n^{(3)}(-Y)$ satisfies 
\[
	J_n^{(3)}(-Y) =\{\psi\in Z(Y)^\dagger\mid \psi(J^{(3)}_{-n+1}(Y))=0\}.
\]
For each $\vartheta^\dagger\neq 0\in Z(Y)^\dagger$, $q_3(-Y;\vartheta^\dagger)$ satisfies
\begin{equation}\label{eqn:q3-under-duality}
		q_3(-Y;\vartheta^\dagger) = -\max\{q_3(Y;\vartheta')\mid\langle\vartheta^\dagger, \vartheta'\rangle = 1\}.
\end{equation}
\end{prop}
If $Y$ is an $\Bbb F_2$-homology sphere, the above formula implies $q_3(-Y)=-q_3(Y)$.

\begin{remark}
As discussed in Remark \ref{rmk:Scpx-q2}, the dg-module $\widetilde C(Y, \mathfrak a)$ can also be understood as an $\mathcal S$-complex over $\chi_1$. In this case, $Z(\widetilde C) = A(Y) \oplus \chi_2 A(Y) \oplus Z(Y)$. The discussion above goes through with little change in the case that $A(Y) = 0$, or equivalently that $H_1(Y;\mathbb Z) \cong (\mathbb Z/2)^m$. As mentioned in Remark \ref{rmk:q2}, the corresponding $h$-invariant for integer homology spheres is called $q_2(Y)$, and was introduced and studied by Fr{\o}yshov in \cite{Fr:q2}. When $A(Y) \ne 0$, it is more difficult to obtain a well-defined filtration, as in this case the map \begin{align*}\tau_0: A(Y) \oplus \chi_2 A(Y) &\oplus Z(Y) \to A(Y) \oplus \chi_2 A(Y) \oplus Z(Y) \\
	\tau_0 &= \left[\begin{array}{cc|c} \nu_0 & 0 & 0 \\ \nu_2 & \nu_0 & \xi_2 \\ \hline \xi_0 & 0 & 1\end{array}\right]\end{align*} induced by the cobordism $I \times Y: (Y, \mathfrak a) \to (Y, \fa')$ is not the identity when $\mathfrak a \ne \mathfrak a'$. 
\end{remark}


\section{Suspension and indefinite cobordism maps}\label{sec:susp}

In this section, we prove the main inequality \eqref{b+-ineq} in full generality. By the behavior of $q_3$ under orientation-reversal, it suffices to prove the inequality $-b^+(W) \le q_3(Y') - q_3(Y)$. In fact, we prove Theorem \ref{thm:main-ineq-QHS}, as a generalization of this inequality which holds for cobordisms between rational homology spheres. 

When $b^+(W) = 0$ and $H^1(W;\mathbb F_2) = 0$, this inequality is proved by showing the existence of an $\mathcal S$-complex morphism $\widetilde \lambda_W: \widetilde C(Y) \to \widetilde C(Y')$ of height zero. In the case $b^+(W) > 0$, there is no such morphism. Instead, in Section \ref{subsec:susp-chains} we construct a new $\mathcal S$-complex, the {\it suspension complex}; this construction shifts $q_3$ up by $1$. When $b^+(W) = 1$, in Section \ref{subsec:susp-cobmap} we construct a morphism of height $0$ from $\widetilde C(Y)$ to $S \widetilde C(Y')$. The inequality is then established in general by decomposing an arbitrary cobordism into pieces with $b^+ = 1$.

\subsection{Suspension complexes}\label{subsec:susp-chains}
In this subsection, we introduce a definition of \emph{suspension complex} of a flow category. This will be used to define obstructed cobordism maps in the case that $b^+(W) = 1$. To define the suspension complex, we need the following extension of the notion of stratified-smooth manifolds from \cite{DME1}.

\begin{definition}
	A \textit{stratified-smooth space} $(P,\Delta)$ (or simply $P$) of dimension $n$ is the following data: 
	\begin{itemize}
		\item a partially ordered set $\Delta$ equipped with a dimension function $d: \Delta \to \mathbb N$ with maximum value $n$, 
		\item a compact topological space $P$ equipped with a partition into locally-closed subsets $P_\sigma$ for $\sigma \in \Delta$, satisfying $\overline P_\sigma = \bigcup_{\tau \le \sigma} P_\tau$, 
		\item the structure on each stratum $P_\sigma$ of a smooth manifold of dimension $d(\sigma)$.
	\end{itemize}
	We demand that the following {\it locally standard} hypothesis holds: for each $p \in P_\sigma$, there are non-negative integers $k$, $l$, an isomorphism of posets 
	$\bar \phi: \Delta_{\ge \sigma} \cong \{0,1\}^k$, an open neighborhood $U$ of $p$ in $P$, an open neighborhood $V$ of the origin in $[0,\infty)^k \times \mathbb R^{n+l-k}$, 
	a smooth map $\psi: V\to \mathbb R^l$, and a homeomorphism $\phi:U\to \psi^{-1}(0)$ with the following properties. For any stratum $Q_\sigma$ of $[0,\infty)^k \times \mathbb R^{n+l-k}$,
	the restriction of the map $\psi$ to $V\cap Q_\sigma$ is smooth and transverse to $0$. Furthermore, $\phi$ restricted to $P_\sigma$ takes values in 
	$V\cap Q_{\overline \phi(\sigma)}$ and this restriction is a diffeomorphism.
	
	A \textit{stratified-smooth $SO(3)$-space} is a stratified-smooth space $P$ equipped with a continuous $SO(3)$-action, for which each $P_\sigma$ is $SO(3)$-invariant, and the action is smooth on each stratum.

	A \textit{diffeomorphism} of stratified-smooth spaces $(\phi,\bar \phi): (P,\Delta)\to (P',\Delta')$  is the pair of an isomorphism of partially-ordered sets 
	$\bar \phi:\Delta \to \Delta'$ and a homeomorphism $\phi:P\to P'$ that restricts to a diffeomorphism from $P_\sigma$ to $P'_{\bar \phi(\sigma)}$ for each 
	$\sigma\in \Delta$.
\end{definition}

We refer the reader to \cite[Appendix A]{DME1} for more details on properties of stratified-smooth spaces. While the notion of a stratified-smooth space is weaker than that of a stratified-smooth manifold, it is adequate for the constructions with cellular homology from Section \ref{subsec:flowcat}. Let $P$ be a stratified-smooth  space of dimension $n$ and let $X$ be a CW-complex equipped with a smooth structure such that all cells are submanifolds. Following Definition \ref{cellular-fundamental-class}, a continuous map $f: P \to X$ is cellular if $f(P_\sigma)\subset X^{(d(\sigma))}$ for each stratum $\sigma$. The fundamental class $f_*[P]$ of a cellular map $f:P\to X$ is the element of $C_n^{\rm cell}(X)$ with the same formula as in \eqref{degree-str-smooth-space-map}:
\begin{equation}\label{degree-str-smooth-space-map-2}
	\langle f_*[P], e_\alpha\rangle = \deg(f: (P, \partial P) \to (e_\alpha/\partial e_\alpha, *)) \mod 2.
\end{equation}
Here $e_\alpha$ denotes an $n$-cell in $X$, and we use the second interpretation of the degree in Subsection \ref{subsec:flowcat} to make sense of the above expression. We define the boundary of a stratified-smooth space $(P,\Delta)$ as in Definition \ref{defn:bdry-stratsmooth}, and the second proof of Lemma \ref{Lemma:bdry-relation} implies that the following relation still holds for a cellular map $f: P \to X$:
\begin{equation}\label{eqn:cellular-fundamental-class-2}
	\partial^{\rm cell} f_*[P] = f_*[\partial P].
\end{equation}

Now fix an instanton $SO(3)$-flow category $\mathcal M$. Throughout this section, we make two standing assumptions on $\mathcal M$: 
\begin{itemize}
	\item We assume that for any $ \alpha, \beta \in \mathsf{orb}(\mathcal M)$, the elements of the stratified-smooth $SO(3)$-manifold $\eqm(\alpha, \beta)$ have trivial stabilizer. This assumption is stronger than necessary for most of what follows, but it suffices for our applications and simplifies the discussion; in particular, we may apply Proposition \ref{prop:flowcat-cell-approx}.
	\item We assume that the compact spaces $\eqm(\alpha, \beta)$, together with the source, target, and composition maps in the definition of a flow category, are provided through the dimension range \eqref{eqn:bubbling-dim-res}. This is used in the proof of Lemma \ref{lemma:zb-bdry-relns}.
\end{itemize} 
Both assumptions hold for the flow category of a rational homology sphere.

For any central orbit $\theta \in \mathsf{orb}(\mathcal M)$ and any orbit $\beta \in \mathsf{orb}(\mathcal M)$, let $\psi_{\theta \beta}: \eqm(\theta, \beta) \to \mathbb R^3$ be an $SO(3)$-equivariant map, where $\mathbb R^3$ is equipped with the standard $SO(3)$ action. These are required to be continuous and smooth on each stratum. They are furthermore required to be transverse to zero on each stratum and satisfy the following compatibility relation on the image of the composition maps: 
\[([A], [A']) \in \eqm(\theta, \alpha) \times_{\alpha} \eqm(\alpha, \beta) \subset \eqm(\theta, \beta) \quad \implies \quad \psi_{\theta \beta}([A], [A']) = \psi_{\theta \alpha}([A]).\]
We define $\eqz(\theta, \beta) = \psi_{\theta \beta}^{-1}(0)$. We also define $\eqb(\theta, \beta)$, the {\it real blowup} of $\psi_{\theta \beta}$, as 
\[\eqb(\theta, \beta):=\{([A], \lambda, v) \in \eqm(\theta, \beta) \times [0,\infty) \times S^2 \mid \psi_{\theta \beta}([A]) = \lambda v\}.\] 
The subspace of $\eqb(\theta, \beta)$ given by $\lambda\neq 0$ projects homeomorphically to the complement of $\eqz(\theta, \beta)$ inside $\eqm(\theta, \beta)$, while the preimage of $\eqz(\theta, \beta)$ is $\eqz(\theta, \beta) \times S^2$. These spaces determine the following correspondences:
\begin{equation}\label{cor-eqz-eqb}
  \lift(\theta) \leftarrow \eqz(\theta, \beta) \to \lift(\beta),\hspace{2cm}S^2 \leftarrow\eqb(\theta, \beta) \to \lift(\beta).
\end{equation}
The source map in the second correspondence involving $\eqb(\theta, \beta)$ is induced by the projection map into $S^2$. Equivalently, it can be interpreted as $\psi/\|\psi\|$. We equip $S^2$ with the standard $SO(3)$ action, induced by the natural action of $SO(3)$ on $\mathbb R^3$. The spaces $ \eqz(\theta, \beta)$ and $ \eqb(\theta, \beta)$ inherit $SO(3)$ actions from $ \eqm(\theta, \beta)$, and the correspondences in \eqref{cor-eqz-eqb} are $SO(3)$-equivariant.

\begin{lemma}\label{lemma:zb-bdry-relns}
So long as $|\theta| - |\beta| < 8 + \dim \lift(\beta)$, the spaces $\eqz(\theta, \beta)$ and $\eqb(\theta, \beta)$ are stratified-smooth spaces of dimensions $n = |\theta| - |\beta| - 4$ and $n = |\theta| - |\beta| - 1$, respectively.
The boundary $\partial \eqz(\theta, \beta)$ is given by 
\begin{equation} \label{boundary-Z}
	\bigsqcup_{\alpha \in \mathsf{orb}(\mathcal M)} \eqz(\theta, \alpha) \times_{\alpha} \eqm(\alpha, \beta),
\end{equation}
and the boundary $\partial \eqb(\theta, \beta)$ is given by the union of 
\begin{equation} \label{boundary-B-1}
	\bigsqcup_{\alpha \in \mathsf{orb}(\mathcal M)} \eqb(\theta, \alpha) \times_{\alpha} \eqm(\alpha, \beta),
\end{equation} 
and 
\begin{equation} \label{boundary-B-2}
	S^2 \times \eqz(\theta, \beta).
\end{equation} 	
Furthermore, the restriction of the source map to \eqref{boundary-B-2} is projection to the first factor.
\end{lemma}
\begin{proof}
The locally standard hypothesis for $\eqz(\theta, \beta)$ follows immediately from the fact that $\eqm(\theta, \beta)$ is a stratified-smooth manifold and $\psi$ is continuous and transverse to zero on each stratum. The proof of the locally standard hypothesis for $\eqb(\theta, \beta)$ is presented in \cite[Construction A.32]{DME1}. That these spaces are compact when $|\theta| - |\beta| < 8 + \dim \lift(\beta)$ follows because $\eqm(\theta, \beta)$ is, which is true by the hypothesis that our moduli spaces are defined and compact when \eqref{eqn:bubbling-dim-res} holds.
\end{proof}

As was discussed in Section 4, there is a coherent equivariant homotopy of the endpoint maps so that we obtain a cellular instanton $SO(3)$-flow category. In particular, we may define the maps 
\begin{equation}\label{eqn:hab-2} 
\fm_{\alpha \beta}: C_*^{\text{cell}}(\lift(\alpha)) \to C_*^{\text{cell}}(\lift(\beta))
\end{equation}
satisfying the relations
\begin{equation}\label{eqn:cell-diff-bdry-2} 
\partial^{\text{cell}} \fm_{\alpha \beta} + \fm_{\alpha \beta} \partial^{\text{cell}} = \sum_{\gamma \in \mathsf{orb}(\mathcal M)}  \fm_{\gamma \beta} \fm_{\alpha \gamma}.
\end{equation} 

\begin{lemma}\label{B-Z-cellular}
	If $\mathcal M$ and $\psi_{\theta\beta}: \eqm(\theta, \beta) \to \mathbb R^3$ are as above, then there exists a cellular approximation of $\mathcal M$ so that $\eqz(\theta, \beta)$ and $\eqb(\theta, \beta)$ also define cellular correspondences.
\end{lemma}
\begin{proof}
The claim for $\eqz$ is tautologically true for dimension reasons, because its domain is central; the map to the codomain is a submersion on each stratum, hence cellular. As for $\eqb(\theta, \beta)$, the fiber space $F_B(\theta, \beta)$ is identified with the subset of $F(\theta, \beta)$ for which $\psi_{\theta \beta}(A) = (0, 0, t)$ for some $t \ge 0$. We wish to choose the coherent homotopy of Proposition \ref{prop:flowcat-cell-approx} so that the resulting map $F_B(\theta, \beta) \to \lift(\beta)$ is also cellular. This can be argued exactly as in the case of $F(\alpha, \beta)$ when $\Gamma_\alpha = SO(2)$, and then the homotopy admits a canonical extension to $F(\theta, \beta)$.
\end{proof}

Lemma \ref{B-Z-cellular} implies that we may define 
\[\fz_{\theta\beta}: C_*^{\text{cell}}(\lift(\theta))\cong \F \to C_*^{\text{cell}}(\lift(\beta)),\hspace{1cm}
\fb_{\theta\beta}: C_*^{\text{cell}}(S^2)\cong  \Lambda(\chi_2) \to C_*^{\text{cell}}(\lift(\beta)),\]
using a similar definition as \eqref{eqn:hab-2} and replacing $\eqm(\alpha, \beta)$ with the moduli spaces $\eqz(\theta, \beta)$ and $\eqb(\theta, \beta)$, respectively. The relation in \eqref{eqn:cellular-fundamental-class-2} and the description of $\partial \eqz(\theta, \beta)$ and $\partial \eqb(\theta, \beta)$ in Lemma \ref{lemma:zb-bdry-relns} give the following analogues of \eqref{eqn:cell-diff-bdry-2}:
\begin{align}
	\partial^{\text{cell}} \fz_{\theta\beta}+ \fz_{\theta\beta}\partial^{\text{cell}} &= \sum_\alpha \fm_{\alpha\beta} \fz_{\theta\alpha}, \label{eqn:Z-reln} \\
	\partial^{\text{cell}} \fb_{\theta\beta}+ \fb_{\theta\beta}\partial^{\text{cell}} &= \fz_{\theta\beta}\pi + \sum_\alpha \fm_{\alpha\beta} \fb_{\theta\alpha}, \label{eqn:B-reln}
\end{align}
where $\pi: C_*^{\text{cell}}(S^2) \to C_*^{\text{cell}}(\text{pt})$ is the map induced by the correspondence $S^2 \leftarrow S^2 \to \text{pt}$. In particular, we have $\pi(1)=1$ and $\pi(\chi_2)=0$, and $\pi$ is (tautologically) a chain map. We also remark that $\partial^{\text{cell}}$ in \eqref{eqn:cell-diff-bdry-2}, \eqref{eqn:Z-reln} and \eqref{eqn:B-reln} is again trivial because we work with coefficients in $\F$.
We will need one more relation. Write $\iota:C_*^{\rm{cell}}({\rm{pt}})\to C_*^{\rm{cell}}(S^2) $ for the map induced by the correspondence ${\rm pt} \leftarrow S^2 \to S^2$, which is given more specifically by $\iota(1)=\chi_2$. Notice again that $\iota$ is tautologically a chain map.
\begin{lemma}\label{iota-b=m}
For any object $\beta$ of $\mathcal M$, we have  \begin{equation}\label{eqn:BM-reln}\fb_{\theta\beta} \iota = \fm_{\theta\beta}.\end{equation}
\end{lemma}
\begin{proof}
The left-hand side is the map induced by the composite correspondence of $\theta \leftarrow S^2 \to S^2$ and $S^2 \leftarrow \eqb(\theta, \beta) \to \lift(\beta)$, which is equal to the induced map of $\theta \leftarrow \eqb(\theta, \beta) \to \lift(\beta)$. Equivalently, this map sends $1$ to the cellular fundamental class of $\eqb(\theta, \beta)$ in $\lift(\beta)$. Because the map $\eqb(\theta, \beta) \to \lift(\beta)$ factors through the projection to $\eqm(\theta, \beta)$, which is a diffeomorphism away from a positive-codimension subset, the two cellular fundamental classes coincide. 
\end{proof}

\begin{definition}\label{def:suspension-complex}
	Suppose we are given a cellular instanton $SO(3)$-flow category $\mathcal M$ and a collection of sections $\psi_{\theta \beta}$ as above. For each $\theta \in \mathfrak Z(\mathcal M)$, 
	let $\lift(\underline \theta) $ be a copy of $S^2$ with its standard $SO(3)$-action and cellular structure. Let also $ \mathsf{orb}_S(\mathcal M)$ be the disjoint union of 
	the following sets:
	\[
	  \mathsf{orb}(\mathcal M) \setminus \mathfrak Z(\mathcal M),\hspace{1cm} \mathfrak Z(\mathcal M) ,\hspace{1cm} \underline{\mathfrak Z}(\mathcal M),
	\]
	where $\underline{\mathfrak Z}(\mathcal M)$ is another copy of $\mathfrak Z(\mathcal M)$ and consists of the elements $\underline \theta$ for any
	$\theta\in \mathfrak Z(\mathcal M)$. Let $| \cdot |_S : \mathsf{orb}_S(\mathcal M) \to \mathbb Z$ be defined as follows. For any $\alpha\in \mathsf{orb}(\mathcal M) \setminus \mathfrak Z(\mathcal M)$, we have $|\alpha|_S = |\alpha|$, and for any $\theta \in \mathfrak Z(\mathcal M)$ we have
	\[
	| \theta |_S=| \theta |-3,\hspace{1cm}| \underline \theta |_S=| \theta |-2.
	\]
	The \textit{geometric suspension complex} $S \widetilde C(\mathcal M)$ is the dg-module over $\Lambda(\chi_1, \chi_2)$ with the underlying module 
	\begin{equation}\label{eqn:susp-complex-module}
	  \bigoplus_{\alpha \in \mathsf{orb}_S(\mathcal M)} C_*^{\rm cell}(\lift(\alpha))[-|\alpha|_S]
	\end{equation}
	and the differential
	\[
	  \hspace{2cm} \widetilde d_S(e_\alpha) = \partial^{\rm cell}(e_\alpha) + \sum_{\beta\in \mathsf{orb}_S(\mathcal M)} \fm^S_{\alpha \beta}(e_\alpha),\hspace{1cm} 
	  \text{for all $e_\alpha\in C_*^{\rm cell}(\lift(\alpha))$,}
	\]	
	where for $\alpha,\beta\in \mathsf{orb}(\mathcal M) \setminus \mathfrak Z(\mathcal M)$ and $\theta,\tau \in  \mathfrak Z(\mathcal M)$, we have
	\begin{align*}
	  &\fm^S_{\alpha \beta}=\fm_{\alpha \beta},\hspace{1cm} \fm^S_{\underline{\theta}\beta}=\fb_{\theta \beta},\hspace{1cm}\fm^S_{\theta\beta}=\fz_{\theta \beta},\\
	  &\fm^S_{\alpha \underline{\tau}}=\iota\fm_{\alpha \tau},\hspace{0.9cm}\fm^S_{\underline{\theta}\underline{\tau}}=\iota\fb_{\theta \tau},\hspace{0.9cm}\fm^S_{\theta\underline{\tau}}=\iota\fz_{\theta \tau},\\
	  &\fm^S_{\alpha {\tau}}=0,\hspace{1.52cm}\fm^S_{\underline{\theta}{\tau}}=\pi\delta_{\theta\tau},\hspace{0.83cm}\fm^S_{\theta{\tau}}=0. 
	\end{align*}	
	In the last row above, $\delta_{\theta\tau}$ is the identity map of $C_*^{\rm cell}(S^2)$ if $\theta=\tau$ and is zero otherwise.
\end{definition}

\begin{remark}
The sections $\psi$ are omitted from the notation. It is a consequence of Proposition \ref{two-suspensions-are-equivalent} below that the resulting dg-module is independent of this choice up to isomorphism.
\end{remark}

\begin{prop}\label{prop:Sg-is-complex}
	The complex $(S \widetilde C(\mathcal M),\widetilde d_S)$  is a dg-module over $\Lambda(\chi_1, \chi_2)$. 
\end{prop}

\begin{proof}
	Since all the operators $\fm^S_{\alpha,\beta}$ commute with the action of $C_*^{\rm cell}(SO(3)) = \Lambda(\chi_1, \chi_2)$, the differential
	$\widetilde d_S$ satisfies the Leibniz rule. It remains to show that this operator squares to zero. This is equivalent to 
	\[
	  \partial^{\rm cell} \fm^S_{\alpha\beta}+\fm^S_{\alpha\beta}\partial^{\rm cell}=\sum_{\gamma\in \mathsf{orb}_S(\mathcal M)}\fm^S_{\gamma\beta}\fm^S_{\alpha\gamma}
	\]
	where $\alpha$ (resp. $\beta$) is either in $\mathsf{orb}(\mathcal M) \setminus \mathfrak Z(\mathcal M)$ or is equal to $\theta$ or $\underline \theta$ 
	(resp. $\tau$ or $\underline \tau$) for some $\theta\in \mathfrak Z(\mathcal M)$ (or $\tau\in \mathfrak Z(\mathcal M)$). These nine relations are given as
	\begin{align}
		\partial^{\rm cell} \fm_{\alpha\beta}+\fm_{\alpha\beta}\partial^{\rm cell}&= \sum_{\gamma} \fm_{\gamma\beta} \fm_{\alpha\gamma} 
		+ \sum_{\rho} \fb_{\rho \beta} \iota \fm_{\alpha\rho}\label{module-11}\\
		\partial^{\rm cell} \fb_{\theta\beta}+\fb_{\theta\beta}\partial^{\rm cell}&=\fz_{\theta\beta} \pi +\sum_{\gamma} \fm_{\gamma\beta} \fb_{\theta\gamma} + \sum_{\rho} \fb_{\rho \beta} \iota \fb_{\theta\rho}\label{module-12}\\
		\partial^{\rm cell} \fz_{\theta\beta}+\fz_{\theta\beta}\partial^{\rm cell}&=\sum_{\gamma} \fm_{\gamma\beta} \fz_{\theta\gamma}+ \sum_{\rho} \fb_{\rho \beta} \iota \fz_{\theta \rho}\label{module-13}\\
		\partial^{\rm cell} \iota\fm_{\alpha\tau}+\iota\fm_{\alpha\tau}\partial^{\rm cell}&= \sum_{\gamma } \iota \fm_{\gamma\tau} \fm_{\alpha\gamma}+ \sum_{\rho} \iota\fb_{\rho \tau} \iota \fm_{\alpha\rho}\label{module-21}\\
		\partial^{\rm cell} \iota\fb_{\theta\tau}+\iota\fb_{\theta\tau}\partial^{\rm cell}&=\iota \fz_{\theta\tau} \pi +\sum_{\gamma} \iota \fm_{\gamma\tau} \fb_{\theta\gamma}+ \sum_{\rho} \iota\fb_{\rho \tau} \iota \fb_{\theta\rho}\label{module-22}\\
		\partial^{\rm cell} \iota\fz_{\theta\tau}+\iota\fz_{\theta\tau}\partial^{\rm cell}&=\sum_{\gamma} \iota\fm_{\gamma\tau} \fz_{\theta \gamma}+\sum_{\rho} \iota\fb_{\rho \tau} \iota \fz_{\theta \rho}\label{module-23}
	\end{align}
	together with the three relations
	\begin{equation}\label{last-row-rel}
		\pi \iota \fm_{\alpha\tau}=0,\hspace{1cm}\pi \iota \fb_{\theta\tau}=0,\hspace{1cm}\pi \iota \fz_{\theta\tau}=0.
	\end{equation}
	In \eqref{module-11}-\eqref{module-23}, $\gamma$ runs through $\mathsf{orb}(\mathcal M) \setminus \mathfrak Z(\mathcal M)$ and $\rho$ runs through $\mathfrak Z(\mathcal M)$.
	To verify these relations, we first use \eqref{eqn:BM-reln} to simplify, and then apply \eqref{eqn:cell-diff-bdry-2}, \eqref{eqn:Z-reln}, \eqref{eqn:B-reln}. The identities in 
	\eqref{last-row-rel} follow from $\pi \iota = 0$.
\end{proof}

\begin{remark}
By decomposing the direct sum \eqref{eqn:susp-complex-module} further into summands corresponding to $\mathsf{orb}(\mathcal M) \setminus \mathfrak Z(\mathcal M)$, $\underline{\mathfrak Z}(\mathcal M)$, and $\mathfrak Z(\mathcal M)$, the array of equations defining $\mathfrak m^S_{\alpha \beta}$ can be understood as the matrix coefficients of $\mathfrak m^S$ with respect to this decomposition, and the relations \eqref{module-11}-\eqref{module-23} can be understood as arising from matrix multiplication.
\end{remark}	

\begin{remark}
The geometric suspension complex is inspired by a construction from Morse theory. Suppose $X$ is an $SO(3)$-manifold equipped with an $SO(3)$-equivariant Morse--Smale function $f: X \to \mathbb R$. Let $\beta: X \to [-1,1]$ be an $SO(3)$-invariant bump function equal to $-1$ in a neighborhood of the fixed critical points and equal to $1$ outside a slightly larger neighborhood, and let $\epsilon > 0$ be small. Define a function on $X \times \mathbb R^3$ by \[f_S(x,v) = f(x) + (2\epsilon^2 \beta(x)\|v\|^2 + \|v\|^4).\] This Morse function gives rise to a flow category, which has associated dg-module $\widetilde C(\mathcal M_{f_S})$. This dg-module coincides with the geometric suspension complex $S\widetilde C(\mathcal M_f)$ associated with the flow category of $f$. For this reason, we expect that the complex $S\widetilde C(\mathcal M)$ is the dg-module associated with an $SO(3)$-equivariant flow category. We do not endeavor to prove this, as it is not needed in what follows.
\end{remark}

To give a more explicit description of $(S \widetilde C(\mathcal M),\widetilde d_S)$, let $C = C(\mathcal M)$ be the vector space freely generated by irreducible orbits of $\mathcal M$, $A$ be the vector space freely generated by the abelian orbits, and $Z$ be the vector space freely generated by central orbits. Then the underlying vector space of $S\widetilde C(\mathcal M)$ is given by 
\begin{equation}\label{suspension-splitting}
	C \oplus \chi_1 C \oplus \chi_2 C \oplus \chi_3 C \;\; \oplus \;\; A \oplus \chi_2 A \;\; \oplus \;\; Z[2] \oplus \chi_2 Z[2] \;\; \oplus Z[3],\end{equation} 
and the differential can be written explicitly as 
\[\widetilde d_S = \left[\begin{array}{cccc|cc|cc|c}
d_1 & 0 & 0 & 0 	&	0 & 0 &	0 & 0 & 	0 \\
d_2 & d_1 & 0 & 0 	&	e_2 & 0 &	\delta_4 & 0 & 0 \\
d_3 & 0 & d_1 & 0 	& 	0 & 0 &	0 & 0 &	0 \\
d_4 & d_3 & d_2 & d_1 & e_4 & e_2 & b_6 & \delta_4 & z_7 \\
\hline
e_1 & 0 & 0 & 0 	& 	0 & 0 & 	0 & 0 & 	0 \\
e_3 & 0 & e_1 & 0 	& 	0 & 0 & 	0 & 0 & 	0 \\
\hline
0 & 0 & 0 & 0 		& 	0 & 0 & 	0 & 0 & 	0 \\
\delta_1 & 0 & 0 & 0 &	0 & 0 & 	0 & 0 & 	0 \\
\hline
0 & 0 & 0 & 0 		& 	0 & 0 & 	1 & 0 & 	0
\end{array}\right].\]
Here $\delta_4, b_6$ are induced by the correspondence $S^2 \leftarrow \eqb(\theta, \beta) \to \lift(\beta)$ when $\beta$ is irreducible; similarly, $z_7$ is induced by the correspondence $\lift(\theta) \leftarrow \eqz(\theta, \beta) \to \lift(\beta)$ when $\beta$ is irreducible. Proposition \ref{prop:Sg-is-complex} ensures the above differential squares to zero. In addition to the relations entailed by $\widetilde d^2 = 0$, this implies the novel relation \begin{equation}\label{nove-rel}d_3 \delta_4 + d_1 b_6 = z_7,\end{equation} which can also easily be argued directly.

Most terms in the above expression for the differential $\widetilde d_S$ are already determined by the differential $\widetilde d$. Only the two terms $b_6$, $z_7$ depend on the choice of the sections $\{\psi_{\theta \beta}\}$. It is natural to ask to what extent $\widetilde d_S$ genuinely depends on the particular choice of sections. To address this, define another dg-module $\Sigma \widetilde C(\mathcal M)$ over $\Lambda(\chi_1, \chi_2)$ with the same underlying module as $S \widetilde C(\mathcal M)$ and the differential 
\[\widetilde d_\Sigma = \left[\begin{array}{cccc|cc|cc|c}
d_1 & 0 & 0 & 0 	&	0 & 0 &	0 & 0 & 	0 \\
d_2 & d_1 & 0 & 0 	&	e_2 & 0 &	\delta_4 & 0 & 0 \\
d_3 & 0 & d_1 & 0 	& 	0 & 0 &	0 & 0 &	0 \\
d_4 & d_3 & d_2 & d_1 & e_4 & e_2 & 0 & \delta_4 & d_3 \delta_4 \\
\hline
e_1 & 0 & 0 & 0 	& 	0 & 0 & 	0 & 0 & 	0 \\
e_3 & 0 & e_1 & 0 	& 	0 & 0 & 	0 & 0 & 	0 \\
\hline
0 & 0 & 0 & 0 		& 	0 & 0 & 	0 & 0 & 	0 \\
\delta_1 & 0 & 0 & 0 &	0 & 0 & 	0 & 0 & 	0 \\
\hline
0 & 0 & 0 & 0 		& 	0 & 0 & 	1 & 0 & 	0
\end{array}\right].\]
Note that $\widetilde d_\Sigma$ is determined by $\widetilde d$. The $\cS$-complex $(\Sigma \widetilde C(\mathcal M), \chi_2)$ is precisely the suspension complex of \cite[Definition 2.23]{DS:-unori-skein-tr}, applied to the $\cS$-complex $(\widetilde C(\mathcal M), \chi_2)$.

\begin{prop}\label{two-suspensions-are-equivalent}
The dg-modules $(S \widetilde C(\mathcal M), \widetilde d_{S})$ and $(\Sigma \widetilde C(\mathcal M), \widetilde d_{\Sigma})$ are isomorphic.
\end{prop}
\begin{proof}
	Let $\widetilde \phi:S \widetilde C(\mathcal M) \to \Sigma \widetilde C(\mathcal M)$ be the map that sends the summand $Z[3]$ in \eqref{suspension-splitting} to 
	the summand $\chi_3 C$ by the map $b_6$ and is zero on the other summands. Then \eqref{nove-rel} implies that the isomorphism 
	$1+\widetilde \phi:S \widetilde C(\mathcal M) \to \Sigma \widetilde C(\mathcal M)$ is a chain map, which is also clearly a module homomorphism over the ring 
	$\Lambda(\chi_1, \chi_2)$.	
\end{proof}

\begin{prop}\label{prop:filt-shift}
	For any integer $n$, we have $J^{(3)}_n(S \widetilde C(\mathcal M))=J^{(3)}_{n-1}(\mathcal M)$. Equivalently, for any $\vartheta \in Z(\mathcal M)$, we have 
	\[q_3(S \widetilde C(\mathcal M); \vartheta) = q_3(\mathcal M; \vartheta) + 1.\] 
\end{prop}
\begin{proof}
	Consider the morphism $\widetilde \lambda: \widetilde C(\mathcal M) \to S\widetilde C(\mathcal M)$ given as 
	\[\widetilde \lambda =
	\left[\begin{array}{cccc|cc|c}
		1 & 0 & 0 & 0 	& 	0 & 0 	& 	0 \\
		0 & 1 & 0 & 0	&	0 & 0 	& 	0 \\
		0 & 0 & 1 & 0 	& 	0 & 0 	& 	0 \\
		 0& 0 & 0 & 1	& 	0 & 0	& 	0 \\
		 \hline
		 0 & 0 & 0 & 0 	& 	1 & 0 	& 	0 \\
		 0 & 0 & 0 & 0 	& 	0 & 1		& 	0 \\
		 \hline
		 0 & 0 & 0 & 0 	& 	0 & 0	& 	0 \\
		 0 & 0 & 0 & 0	& 	0 & 0 	& 	1 \\
		 \hline
		 0 & 0 & 0 & 0 	& 	0 & 0 	& 	0
	\end{array}\right]\]
	defined with respect to the decompositions of $\widetilde C(\mathcal M)$ and $S\widetilde C(\mathcal M)$ as in \eqref{eqn:dg-module-decomposition} 
	and \eqref{suspension-splitting}. This is similar to the morphisms of \cite[Construction 5.19]{DME1} and \cite[Equation (2.34)]{DS:-unori-skein-tr}. It is straightforward to check that $\widetilde \lambda:(\widetilde C(\mathcal M),\chi_2)\to (S\widetilde C(\mathcal M),\chi_2)$ 
	is a strong height-$1$ morphism such that $\medhat \lambda: \medhat{H}(\widetilde C(\mathcal M),\chi_2)\to \medhat{H} (S\widetilde C(\mathcal M),\chi_2)$ is 
	an isomorphism; compare \cite[Proposition 5.20]{DME1} and \cite[Proposition 7.4]{DS:-unori-skein-tr}. Thus the claim follows from Proposition \ref{prop:tau-filtered} and the 
	discussion following its proof.
\end{proof}

Now, suppose that $(Y, \mathfrak a)$ is a rational homology sphere with abelian data and a choice of cellular approximation of the holonomy maps. For each $\theta \in \mathfrak Z(Y)$, suppose that $\psi_{\theta \beta}: \eqm(Y;\theta, \beta) \to \mathbb R^3$ are sections as above: continuous, smooth on each stratum, and compatible with gluing on the right. Then define $S\widetilde C(Y, \mathfrak a)$ to be the geometric suspension of $\widetilde C(Y, \mathfrak a)$ associated with this choice of sections. Similarly, let $\Sigma\widetilde C(Y, \mathfrak a)$ be the algebraic suspension of $\widetilde C(Y, \mathfrak a)$. These two dg-modules are isomorphic by Proposition \ref{two-suspensions-are-equivalent}, and Proposition \ref{prop:filt-shift} implies that suspension shifts $q_3$ by $1$.

\subsection{Obstructed gluing theory}\label{subsec:susp-bdry-relns}
Our goal in this section is to extend the functoriality of Section \ref{subsec:bimod} to those cobordisms $W:Y\to Y'$ satisfying $b_1(W) = b_1(\partial W) = 0$ and $b^+(W) = 1$. After a perturbation, it may be assumed that $W$ supports no abelian ASD connections, and that irreducible ASD connections on this cobordism are cut out transversely \cite[Proposition 2.12]{Don}.
The novel difficulty with such cobordisms is that the moduli spaces of ASD connections on $W$ are not guaranteed to be cut out transversely, even after a perturbation of the ASD equation. However, corresponding to any representation $\rho:\pi_1(W)\to \mathbb Z/2$ (or equivalently an element of $H^1(W;\mathbb Z/2)$), there is a central flat $SU(2)$-connection $\Theta$ on $W$. The connection $\Theta$ remains a solution of the ASD equation after any (gauge invariant) perturbation, but it is not cut out transversely. In fact, the cokernel of the linearization of the ASD equation can be identified with $H^+(W;\mathbb R^3)$, which is $3$-dimensional. We will write $\mathfrak Z(W)$ for the set of all such central connections on $W$ up to the action of the gauge group. 

In order to define cobordism maps associated with $W$, we need to use {\it obstructed gluing theory} to analyze neighborhoods of elements of $\eqm(W; \alpha, \alpha')$ where there is a contribution of an element of $\mathfrak Z(W)$. 
Fix auxiliary data $\mathfrak a$, $\mathfrak a'$ for $Y$, $Y'$, and following \eqref{eqn:dim-res-2}, let $\alpha\in \mathsf{orb}(Y,\mathfrak a)$ and $\alpha'\in \mathsf{orb}(Y',\mathfrak a')$ be such that 
\begin{equation}\label{eqn:dim-res-3}|\alpha| - |\alpha'| - 3 \le \dim \lift(\alpha') + 1.\end{equation} 
If the perturbation of the ASD equation on $W$ is chosen as in the last paragraph, then $\eqm(W; \alpha, \alpha')$ is a stratified-smooth manifold of dimension $|\alpha| - |\alpha'| + \dim \lift(\alpha)- 3$ away from the strata of trajectories broken through a connection $\Theta$ in $\mathfrak Z(W)$. In the following we write $\theta, \theta'$ for the restrictions of $\Theta$ to the ends. Each such $\Theta$ can appear in three types of boundary strata of $\eqm(W; \alpha, \alpha')$:
\begin{itemize}
\item \textbf{Type I.} If $\alpha' = \theta'$ but $\alpha \ne \theta$, then a neighborhood of the connections broken through $\Theta$ is modeled on the zero set of an $SO(3)$-equivariant map
\begin{equation}\label{obs-section-1}
	\psi_{\alpha \Theta}: (0, \infty] \times \eqm(Y; \alpha, \theta) \to \mathbb R^3,
\end{equation} 
which is transverse to zero over $(0, \infty)$ but zero at $\infty$.
\item \textbf{Type II.} If $\alpha = \theta$ but $\alpha' \ne \theta'$, a neighborhood of connections broken through $\Theta$ is modeled on the zero set of an $SO(3)$-equivariant map
\begin{equation}\label{obs-section-2}
\psi_{\Theta \alpha'}: [-\infty, 0) \times \eqm(Y'; \theta', \alpha') \to \mathbb R^3,
\end{equation} 
which is transverse to zero over $(-\infty,0)$ but zero at $-\infty$. 
\item \textbf{Type III.} If $\alpha \ne \theta$ and $\alpha' \ne \theta'$, a neighborhood of connections broken through $\Theta$ is modeled on the zero set of an $SO(3)$-equivariant map
\begin{equation}\label{obs-section-3}
	\psi_{\alpha \Theta \alpha'}: (0, \infty] \times [-\infty, 0) \times \eqm(Y; \alpha, \theta) \times \eqm(Y'; \theta', \alpha') \to \mathbb R^3,
\end{equation}
zero at $(\infty, -\infty)$ but transverse to zero over its complement.
\end{itemize}
These identifications can be made compatibly on moduli spaces satisfying \eqref{eqn:dim-res-3}. We refer to the maps in \eqref{obs-section-1}-\eqref{obs-section-3} as obstruction sections. The above claims can be verified as in \cite[Section 5]{DME1} (see also \cite[Section 4]{DS1}). The only difference between the present setup and that of  \cite{DME1} is that here $\Theta$ is central, whereas in  \cite{DME1} we consider abelian obstructed instantons.

Choosing $L$ generic and large, we truncate these moduli spaces by deleting the portion of Type I trajectories lying above $(L, \infty]$, the portion of Type II trajectories lying above $[-\infty, -L)$, and the portion of Type III trajectories lying above $\{(t,s) : 1/t - 1/s < 1/L\}$. For generic $L$, the resulting space is stratified-smooth, with a new boundary component corresponding to the zero set of a section over $\eqm(Y; \alpha, \theta)$, $\eqm(Y'; \theta', \alpha')$, or $I \times \eqm(Y; \alpha, \theta) \times \eqm(Y'; \theta', \alpha')$, respectively, where 
\[I = \{(t,s) \in (0, \infty] \times [-\infty, 0) : 1/t - 1/s = 1/L\},\]
which we identify with $[0,1]$ via the map $(t,s) \mapsto L/t$. The resulting subspace of $\eqm(W; \alpha, \alpha')$ is denoted by $\eqn(W; \alpha, \alpha')$. The following lemma gives a description of the boundary of this space. To state the lemma, let $\widetilde Z_{\alpha \Theta}$ be the zero set of $\psi_{\alpha \Theta}$ restricted to time $L$, and similarly $\widetilde Z_{\Theta \alpha'}$ and $\widetilde Z_{\alpha \Theta \alpha'}$.

\begin{lemma}\label{lemma:N-relations}
	The truncated moduli space $\eqn(W; \alpha, \alpha')$ is a stratified-smooth space satisfying 
	\begin{align*}
		\partial \eqn(W; \alpha, \alpha') =  &\bigsqcup_{\beta \in \mathsf{orb}(Y,\mathfrak a)} \eqm(Y; \alpha, \beta) \times_\beta \eqn(W; \beta, \alpha') \\
		& \bigsqcup_{\beta' \in \mathsf{orb}(Y',\mathfrak a')} \eqn(W;\alpha, \beta') \times_{\beta'} \eqm(Y';\beta',\alpha') \\
		& \left(\bigsqcup\limits_{\Theta \in \mathfrak Z(W)} \widetilde Z_{\alpha \Theta \alpha'}\right)\sqcup \left(\bigsqcup\limits_{\Theta|_{Y'} = \alpha'}  \widetilde Z_{\alpha \Theta}\right) 
		\sqcup \left(\bigsqcup\limits_{\Theta|_{Y} = \alpha} \widetilde Z_{\Theta \alpha'}\right)
\end{align*} 
\end{lemma}

We want to use the moduli spaces $\eqn(W; \alpha, \alpha')$ to define a chain map $\widetilde C(Y, \mathfrak a) \to S \widetilde C(Y', \mathfrak a')$. The construction of the suspension complex $S \widetilde C(Y', \mathfrak a')$ requires a collection of sections $\psi_{\theta'\alpha'}$ for each $\theta' \in \mathfrak Z(Y')$ and $\alpha'\in \mathsf{orb}(Y',\mathfrak a')$. Ideally, we could choose these to be the sections $\psi_{\Theta \alpha'}$ above restricted to $\{-L\} \times \eqm(Y'; \theta', \alpha')$. However, this is only appropriate if the restriction map $\mathfrak Z(W)\to \mathfrak Z(Y')$ is a bijection. To address this problem, we choose {\it privileged} sections $\psi^{\text{priv}}_{\alpha \theta}: \eqm(Y; \alpha, \theta) \to \mathbb R^3$ for each $\theta \in \mathfrak Z(Y)$ and $\alpha\in \mathsf{orb}(Y,\mathfrak a)$, and similarly $\psi^{\text{priv}}_{\theta' \alpha'}: \eqm(Y'; \theta', \alpha') \to \mathbb R^3$ for each $\theta' \in \mathfrak Z(Y')$ and $\alpha'\in \mathsf{orb}(Y',\mathfrak a')$. As in the previous section, these are continuous, smooth on each stratum, and satisfy
\begin{equation}\label{gluing-comp}
  \psi^{\text{priv}}_{\alpha \theta}(A \circ B) = \psi^{\text{priv}}_{\beta \theta}(B), \quad \quad \psi^{\text{priv}}_{\theta' \alpha'}(A \circ B) = \psi^{\text{priv}}_{\theta' \beta'}(A).
\end{equation}
We also require that these are transverse to zero on each stratum, and that on each stratum of $[0,1] \times \eqm(Y; \alpha, \theta) \times \eqm(Y'; \theta', \alpha')$ the map 
\[\psi^{\text{priv}}_{\alpha \alpha'}(t, A, B) = t \psi_{\alpha \theta}^{\text{priv}}(A) + (1-t) \psi_{\theta'\alpha'}^{\text{priv}}(B)\] is transverse to zero. We denote the zero sets of these maps by $\widetilde Z(Y; \alpha, \theta)$, $\widetilde Z(Y'; \theta', \alpha')$, and $\widetilde Z(\alpha, \Theta, \alpha')$, respectively. Note that the last space depends only on the restrictions of $\Theta$ to the two ends. As in the previous section, we may also define the real blowup spaces $\widetilde B(Y; \alpha, \theta)$ and $\widetilde B(Y'; \theta', \alpha')$. The first is equipped with the map $-\psi^{\text{priv}}_{\alpha \theta}/\|\psi^{\text{priv}}_{\alpha \theta}\|$ to $S^2 = \lift(\underline \theta)$, while the latter is equipped with the map $\psi^{\text{priv}}_{\theta'\alpha'}/\|\psi^{\text{priv}}_{\theta'\alpha'}\|$.

Lemma \ref{B-Z-cellular} shows that one may choose the cellular approximations for $\eqm(Y'; \alpha', \beta')$ such that the spaces $\eqz(Y';\theta', \alpha')$ and $\eqb(Y';\theta', \alpha')$ define cellular correspondences. Using the same argument we may find cellular approximations for $\eqm(Y; \alpha, \beta)$ such that $\eqz(Y;\alpha, \theta)$ and $\eqb(Y;\alpha, \theta)$ define cellular correspondences. In particular, we may define the module homomorphisms
\begin{equation}\label{zata'}
  \fz_{\alpha\theta}: C_*^{\text{cell}}(\lift(\alpha)) \to C_*^{\text{cell}}(\lift(\theta)),\hspace{1cm}
  \fz_{\theta'\alpha'}: C_*^{\text{cell}}(\lift(\theta'))\to C_*^{\text{cell}}(\lift(\alpha')),
\end{equation}
associated with $\eqz(Y;\alpha, \theta)$, $\eqz(Y';\theta', \alpha')$, and the module homomorphisms
\begin{equation}\label{bata'}
  \fb_{\alpha\theta}: C_*^{\text{cell}}(\lift(\alpha)) \to C_*^{\text{cell}}(\lift(\underline \theta)),\hspace{1cm}
  \fb_{\theta'\alpha'}: C_*^{\text{cell}}(\lift(\underline \theta'))\to C_*^{\text{cell}}(\lift(\alpha')),
\end{equation}
associated with $\eqb(Y;\alpha, \theta)$, $\eqb(Y';\theta', \alpha')$. We also use the operators in \eqref{zata'} to define
\begin{equation}\label{zeta-I-II-def}
{\mathfrak z}^{I}_{\alpha\alpha'} = \sum_{\Theta|_{Y'} = \alpha'}  \fz_ {\alpha\Theta|_Y},\hspace{1cm}
{\mathfrak z}^{I\!I}_{\alpha\alpha'} = \sum_{\Theta|_{Y} = \alpha} \fz_{ \Theta|_{Y'}\alpha'}
\end{equation}
By definition, ${\mathfrak z}^{I}_{\alpha\alpha'}$ (resp. $\mathfrak z^{I\!I}_{\alpha\alpha'}$) is zero unless $\alpha'$ (resp. $\alpha$) is central.

Next, we analyze the space $\eqz(\alpha, \Theta, \alpha')$ and the module homomorphism associated with this correspondence. Recall that this is the zero set of the map $[0,1] \times \eqm(Y; \alpha, \theta) \times \eqm(Y'; \theta', \alpha') \to \mathbb R^3$ defined by \[t \psi_{\alpha \theta}^{\text{priv}}(A) + (1-t) \psi_{\theta'\alpha'}^{\text{priv}}(B).\] The complement of the locus where $\psi_{\alpha \theta}^{\text{priv}}(A) = 0$ or $\psi_{\theta'\alpha'}^{\text{priv}}(B) = 0$ is equal to the set of points in $(0,1) \times \eqm(Y; \alpha, \theta) \times \eqm(Y'; \theta', \alpha')$ for which \[\psi_{\alpha \theta}^{\text{priv}}(A) = -\frac{1-t}{t} \psi_{\theta'\alpha'}^{\text{priv}}(B).\] The factor $\frac{1-t}{t}$ is uniquely determined by the values of the sections, and the content of this equation is that \[\frac{\psi_{\alpha \theta}^{\text{priv}}(A)}{\|\psi_{\alpha \theta}^{\text{priv}}(A)\|} = -\frac{\psi_{\theta'\alpha'}^{\text{priv}}(B)}{\|\psi_{\theta'\alpha'}^{\text{priv}}(B)\|}.\] That is, this locus is canonically identified with the locus of 
\begin{equation}\label{finer-prod-b}
	\eqb(Y; \alpha, \theta) \times_{S^2} \eqb(Y'; \theta', \alpha')
\end{equation}
 over which neither section is zero. 
 
 From the preceding analysis, one can easily see that $\eqz(\alpha, \Theta, \alpha')$ is a cellular correspondence and it can be used to define a module homomorphism. Adding up these module homomorphisms for all $\Theta\in \mathfrak Z(W)$ gives the map
\[
  {\mathfrak z}^{I\!I\!I}_{\alpha\alpha'}:C_*^{\text{cell}}(\lift(\alpha)) \to C_*^{\text{cell}}(\lift(\alpha')).
\]
Because we have an identification between $\eqz(\alpha, \Theta, \alpha')$ and \eqref{finer-prod-b} in the complement of subspaces of codimension $1$, the induced maps on cellular chains coincide. Consequently, we have
\begin{equation}\label{Z-III-rel}
	\fz^{I\!I\!I}_{\alpha\alpha'}=\sum_\Theta\fb_{\Theta|_{Y'}\alpha'}\fb_{\alpha\Theta|_Y
	}.
\end{equation}

Going back to the goal of defining maps using the truncated moduli spaces, we need to fix homotopies between privileged sections and the obstruction sections. For each $\Theta\in \mathfrak Z(W)$ and each $\alpha \in  \mathsf{orb}(Y,\mathfrak a)$ and $\alpha' \in  \mathsf{orb}(Y',\mathfrak a')$, we choose $SO(3)$-equivariant maps 
\begin{align*}
\Xi_{\alpha \Theta}&\colon [0,1] \times \eqm(Y; \alpha, \theta) \to \mathbb R^3, \\
\Xi_{\Theta \alpha'}&\colon [0,1] \times \eqm(Y'; \theta', \alpha') \to \mathbb R^3, \\
\Xi_{\alpha \Theta \alpha'}&\colon [0,1] \times [0,1] \times \eqm(Y; \alpha, \theta) \times \eqm(Y'; \theta', \alpha') \to \mathbb R^3.
\end{align*}
We demand of these the usual compatibility with respect to gluing as in \eqref{gluing-comp}. We demand that 
\begin{align*}
	\Xi_{\alpha \Theta}(0, A) = \psi_{\alpha \Theta}(L,A), \quad\quad &\Xi_{\alpha \Theta}(1, A) = \psi^{\text{priv}}_{\alpha \theta}(A),\\
	\Xi_{\Theta \alpha'}(0, A) = \psi_{\Theta \alpha'}(-L,A), \quad\quad &\Xi_{\Theta\alpha'}(1, A) = \psi^{\text{priv}}_{\theta'\alpha'}(A),\\
	\Xi_{\alpha \Theta \alpha'}(0,t,A,B) = \psi_{\alpha \Theta \alpha'}\left(\frac Lt, \frac{L}{t-1},A,B\right), \quad\quad &\Xi_{\alpha \Theta \alpha'}(1,t,A,B) = \psi^{\text{priv}}_{\alpha \alpha'}(t, A, B).
\end{align*}
The third map is related to the other maps as 
\[
  \Xi_{\alpha \Theta \alpha'}(s,0,A,B)=  \Xi_{\Theta \alpha'}(s,B),\hspace{1cm}\Xi_{\alpha \Theta \alpha'}(s,1,A,B)=  \Xi_{\alpha \Theta}(s,A).
\]

In all cases we demand these are continuous, smooth on each stratum, and transverse to zero on each stratum. Thus, their zero sets, written $\widetilde Z^\Xi_{\alpha \Theta}$, $\widetilde Z^\Xi_{\Theta\alpha'}$, and $\widetilde Z^\Xi_{\alpha \Theta\alpha'}$, are stratified-smooth spaces. Their boundary relations are as follows.

\begin{lemma}\label{lemma:Xi-relations}
For any $\Theta \in \mathfrak Z(W)$ with limits $\theta, \theta'$, the spaces $\widetilde Z^\Xi_{\alpha \Theta}$, $\widetilde Z^\Xi_{\Theta\alpha'}$, and $\widetilde Z^\Xi_{\alpha \Theta\alpha'}$ defined above are stratified-smooth spaces and their boundaries are described as follows:
\begin{enumerate}[label=(\roman*)]
	\item The boundary of $\widetilde Z^\Xi_{\alpha \Theta}$ with $\Theta|_{Y}=\theta$ is given as 
	\[
	  \partial \widetilde Z^\Xi_{\alpha \Theta} = \widetilde Z_{\alpha \Theta} \sqcup \widetilde Z(Y; \alpha, \theta) 
	  \sqcup \(\bigsqcup_{\beta\in \mathsf{orb}(Y,\mathfrak a)} \eqm(Y; \alpha, \beta) \times_\beta \widetilde Z^\Xi_{\beta \Theta}\).
	 \]
	\item The boundary of $\widetilde Z^\Xi_{\Theta \alpha'}$ is given as 
	\[
	  \partial \widetilde Z^\Xi_{\Theta \alpha'} =\widetilde Z_{\Theta \alpha'}\sqcup \widetilde Z(Y'; \theta', \alpha') \sqcup \(\bigsqcup_{\beta'\in \mathsf{orb}(Y',\mathfrak a')} \widetilde Z^\Xi_{\Theta \beta'} \times_{\beta'} \eqm(Y'; \beta', \alpha')\).
	\]
	\item The boundary of $\widetilde Z^\Xi_{\alpha \Theta \alpha'}$ is given as 
	\begin{align*}
		\partial &\widetilde Z^\Xi_{\alpha \Theta \alpha'}= \widetilde Z_{\alpha \Theta \alpha'} \sqcup \widetilde Z(\alpha, \Theta, \alpha') \sqcup 
		(\eqm(Y; \alpha, \theta) \times \widetilde Z^\Xi_{\Theta \alpha'}) \sqcup (\widetilde Z^\Xi_{\alpha \Theta} \times \eqm(Y';\theta', \alpha'))\\
		&\hspace{.2cm}\(\bigsqcup_{\beta\in \mathsf{orb}(Y,\mathfrak a)} \eqm(Y; \alpha, \beta) \times_\beta \widetilde Z^\Xi_{\beta \Theta \alpha'}\)\sqcup \(\bigsqcup_{\beta'\in \mathsf{orb}(Y',\mathfrak a')} \widetilde Z^\Xi_{\alpha \Theta \beta'} \times_{\beta'} \eqm(Y'; \beta', \alpha')\).
\end{align*}
\end{enumerate}
\end{lemma}

Arguing as in Proposition \ref{prop:flowcat-cell-approx} and Lemma \ref{B-Z-cellular}, we may coherently modify the endpoint maps to ensure that $\widetilde Z^\Xi_{\alpha \Theta}$, $\widetilde Z^\Xi_{\Theta\alpha'}$, and $\widetilde Z^\Xi_{\alpha \Theta\alpha'}$ are cellular correspondences. In particular, $\widetilde Z_{\alpha \Theta}$, $\widetilde Z_{\Theta\alpha'}$, and $\widetilde Z_{\alpha \Theta\alpha'}$ are cellular correspondences. We may then carry out a straightforward induction to ensure that the correspondence defined by $\eqn(W; \alpha, \alpha')$ is cellular for each $\alpha, \alpha'$, with the endpoint maps restricted to $\widetilde Z_{\alpha \Theta}$, $\widetilde Z_{\Theta\alpha'}$, and $\widetilde Z_{\alpha \Theta\alpha'}$ equal to those induced by the endpoint maps of $\widetilde Z^\Xi_{\alpha \Theta}$, $\widetilde Z^\Xi_{\Theta\alpha'}$, and $\widetilde Z^\Xi_{\alpha \Theta\alpha'}$. Therefore, each of these defines a cellular correspondence, and it can be used to define a module map $C_*^{\rm cell}(\lift(\alpha)) \to C_*^{\rm cell}(\lift(\alpha'))$ for appropriate $\alpha$, $\alpha'$. The module homomorphism associated with $\eqn(W; \alpha, \alpha')$ is denoted by 
\[
  {\mathfrak h}_{\eqn(W;\alpha,\alpha')}:C_*^{\rm cell}(\lift(\alpha)) \to C_*^{\rm cell}(\lift(\alpha')).
\]
We also have the maps 
\begin{align}
  \mathfrak z_{\alpha\Theta}^{\Xi}:C_*^{\rm cell}(\lift(\alpha)) \to &C_*^{\rm cell}(\lift(\Theta|_{Y'})),\hspace{1cm}
  \mathfrak z_{\Theta\alpha'}^{\Xi}:C_*^{\rm cell}(\lift(\Theta|_{Y})) \to C_*^{\rm cell}(\lift(\alpha'))\nonumber\\
  &\mathfrak z_{\alpha\Theta\alpha'}^{\Xi}:C_*^{\rm cell}(\lift(\alpha)) \to C_*^{\rm cell}(\lift(\alpha'))\nonumber
\end{align}
associated with the correspondences $\widetilde Z^\Xi_{\alpha \Theta}$, $\widetilde Z^\Xi_{\Theta\alpha'}$, and $\widetilde Z^\Xi_{\alpha \Theta\alpha'}$. Finally, we define
\begin{align*}
{\mathfrak n}_{\alpha\alpha'} &= {\mathfrak h}_{\eqn(W;\alpha,\alpha')} + \sum_{\Theta \in \mathfrak Z(W)} \mathfrak z_{\alpha\Theta\alpha'}^{\Xi}+ \sum_{\Theta|_{Y'} = \alpha'} \mathfrak z_{\alpha\Theta}^{\Xi} +\sum_{\Theta|_Y = \alpha} \mathfrak z_{\Theta\alpha'}^{\Xi}.
\end{align*}

The following lemma follows by examining the relations of Lemmas \ref{lemma:N-relations} and \ref{lemma:Xi-relations}.
\begin{lemma}\label{lemma:cob-op-relns}
The operator ${\mathfrak n}_{\alpha\alpha'}$ satisfies the following boundary relation:
\begin{align*}
\partial^{\rm cell} {\mathfrak n}_{\alpha\alpha'} + {\mathfrak n}_{\alpha\alpha'} \partial^{\rm cell} = & \,{\mathfrak z}^I_{\alpha\alpha'} + {\mathfrak z}^{I\!I}_{\alpha\alpha'} + {\mathfrak z}^{I\!I\!I}_{\alpha\alpha'}\\&+\sum_{\beta \in \mathsf{orb}(Y,\mathfrak a)} {\mathfrak n}_{\beta\alpha'} {\mathfrak m}_{\alpha\beta} + \sum_{\beta' \in \mathsf{orb}(Y',\mathfrak a')} {\mathfrak m}_{\beta'\alpha'} {\mathfrak n}_{\alpha\beta'}.
\end{align*}
\end{lemma}

\subsection{Obstructed cobordism maps and the main inequality}\label{subsec:susp-cobmap}
We now proceed to use the maps defined above to define the dg-module homomorphism $\widetilde \lambda_W: \widetilde C(Y, \mathfrak a) \to S \widetilde C(Y', \mathfrak a')$ associated with a cobordism $W: (Y, \mathfrak a) \to (Y', \mathfrak a')$ for which $b_1(W) = b_1(\partial W) = 0$ and $b^+(W) = 1$. For 
\[
  \alpha\in \mathsf{orb}(Y,\mathfrak a), 
  \hspace{0.8cm}\alpha'\in\mathsf{orb}(Y',\mathfrak a') \setminus \mathfrak Z(Y'),\hspace{0.8cm} \theta'\in \mathfrak Z(Y'),
\] 
we have
\begin{align*}
  \fm^W_{\alpha \alpha'}&={\mathfrak n}_{\alpha\alpha'},\\
  \fm^W_{\alpha \underline{\theta}'}&=\iota {\mathfrak n}_{\alpha\theta'} + \sum_{\Theta|_{Y'} = \theta'} {\mathfrak b}_{\alpha\Theta|_Y},\\
  \fm^W_{\alpha \theta'}&=\sum\limits_{\Theta|_Y = \alpha, \,\, \Theta|_{Y'} = \theta'} 1, 
\end{align*}	
where the sums in these expressions are over all $\Theta\in \mathfrak Z(W)$ satisfying the given constraints. Since any $\Theta\in \mathfrak Z(W)$ restricts to an element of $\mathfrak Z(Y)$ on the incoming end, the last term is trivial unless $\alpha\in \mathfrak Z(Y)$. In the case $\alpha\in \mathfrak Z(Y)$, the map $1$ denotes the identity map $C^{\rm cell}_*(\lift(\alpha)) \to C^{\rm cell}_*(\lift(\theta'))$. In particular, the map $\epsilon_0: Z(Y) \to Z(Y')$ introduced in Section \ref{subsec:bimod} satisfies \[\fm^W_{\alpha \theta'}(1) = \langle \epsilon_0(\alpha), \theta'\rangle.\] Throughout this section, to indicate the dependence on $W$, we will denote this map by $\epsilon_0^W$ and call it the central component of $W$. This map is nonzero if and only if $H^1(W;\mathbb F_2) \to H^1(\partial W; \mathbb F_2)$ is injective, which is equivalent to $\dim H^1(W, \partial W; \mathbb F_2) = 1$.  

Now we are ready to define $\widetilde \lambda_W$. For $\alpha\in \mathsf{orb}(Y, \mathfrak a)$ and $e_\alpha\in C^{\rm cell}_*(\lift(\alpha))$, let
\[
  \widetilde\lambda_W(e_\alpha) = \sum_{\alpha'}\fm^{W}_{\alpha\alpha'}(e_\alpha),
\]
where the sum is over all $\alpha'\in \mathsf{orb}_S(Y', \mathfrak a')$. 

\begin{prop}\label{prop:obs-cob-map}
	The map $\widetilde \lambda_W$ is a dg-module homomorphism.
\end{prop}
\begin{proof}
	This map is clearly a module homomorphism because all of the maps ${\mathfrak n}_{\alpha\alpha'}$, ${\mathfrak b}_{\alpha, \Theta|_Y}$, $\iota$, and 
	$1$ are module homomorphisms. 
	To show that $\widetilde \lambda_W$ is a chain map, we need to check that for any $\alpha\in \mathsf{orb}(Y, \mathfrak a)$ and $\alpha'\in \mathsf{orb}_S(Y', \mathfrak a')$
	the following relation holds:
	\begin{equation}\label{chain-map-ob-cob-map}
	  \partial^{\rm cell} \fm^W_{\alpha\alpha'}+\fm^W_{\alpha\alpha'}\partial^{\rm cell}=\sum_{\beta'\in \mathsf{orb}_S(Y', \mathfrak a')}\fm^S_{\beta'\alpha'}\fm^W_{\alpha\beta'}
	  +\sum_{\beta\in \mathsf{orb}(Y, \mathfrak a)}\fm^W_{\beta\alpha'}\fm_{\alpha\beta}
	\end{equation}
	\begin{itemize}
	\item If $\alpha'\in\mathsf{orb}(Y',\mathfrak a') \setminus \mathfrak Z(Y')$, then \eqref{chain-map-ob-cob-map} expands to 
		\begin{align*}
			\partial^{\rm cell} {\mathfrak n}_{\alpha \alpha'} + {\mathfrak n}_{\alpha\alpha'} \partial^{\rm cell} =& 
			 \sum_{\beta' \in \mathsf{orb}(Y',\mathfrak a') \setminus \mathfrak Z(Y')} {\mathfrak m}_{\beta' \alpha'} {\mathfrak n}_{\alpha \beta'}+\sum_{\theta' \in \mathfrak Z(Y')} 			{\mathfrak b}_{\theta' \alpha'} \iota {\mathfrak n}_{\alpha \theta'} \\
			&+  \sum_{\Theta \in \mathfrak Z(W)} {\mathfrak b}_{\Theta|_{Y'},\alpha'} {\mathfrak b}_{\alpha \Theta|_Y}+\sum\limits_{\Theta|_Y = \alpha, \,\, \Theta|_{Y'} = \theta'} \fz_{\theta' \alpha'}\\
			&+\sum_{\beta \in \mathsf{orb}(Y,\mathfrak a)} {\mathfrak n}_{\beta\alpha'} {\mathfrak m}_{\alpha\beta}.
		\end{align*} 
		Using the relation ${\mathfrak b}_{\theta' \alpha'} \iota= {\mathfrak m}_{\theta' \alpha'}$ provided by Lemma \ref{iota-b=m}, the first two terms combine as 
		\[
		  \sum_{\beta' \in \mathsf{orb}(Y',\mathfrak a')} {\mathfrak m}_{\beta'\alpha'} {\mathfrak n}_{\alpha\beta'}.
		\] 
		Now the desired relation follows immediately from \eqref{Z-III-rel} and Lemma \ref{lemma:cob-op-relns}.
	\item If $\alpha'=\underline \theta'$ for $\theta'\in \mathfrak Z(Y')$, then the right hand side of \eqref{chain-map-ob-cob-map} can be rewritten as 
		\begin{align*} 
			 \sum_{\beta' \in \mathsf{orb}(Y',\mathfrak a') \setminus \mathfrak Z(Y')} &\iota{\mathfrak m}_{\beta' {\theta}'} {\mathfrak n}_{\alpha \beta'}+\sum_{\tau' \in \mathfrak Z(Y')} 			\iota{\mathfrak b}_{\tau' {\theta}'} \iota {\mathfrak n}_{\alpha \tau'}+  \sum_{\Theta \in \mathfrak Z(W)} \iota{\mathfrak b}_{\Theta|_{Y'}{\theta}'} {\mathfrak b}_{\alpha \Theta|_Y} \\
			+&\sum\limits_{\Theta|_Y = \alpha, \,\, \Theta|_{Y'} = \tau'} \iota\fz_{\tau' {\theta}'}+\sum_{\beta \in \mathsf{orb}(Y,\mathfrak a)}\iota  {\mathfrak n}_{\beta{\theta}'} {\mathfrak m}_{\alpha\beta}+
			\sum_{\substack{\beta \in \mathsf{orb}(Y,\mathfrak a)\\\Theta|_{Y'} = \theta' }} {\mathfrak b}_{\beta \Theta|_Y}{\mathfrak m}_{\alpha\beta}.
		\end{align*} 
		We may simplify the above expression by combining the first two terms as in the previous case, applying \eqref{Z-III-rel} and \eqref{zeta-I-II-def} to the third and the fourth terms, and 
		obtain
		\[
		  \sum_{\beta' \in \mathsf{orb}(Y',\mathfrak a')} \iota{\mathfrak m}_{\beta'\theta'} {\mathfrak n}_{\alpha\beta'}+\iota {\mathfrak z}^{I\!I\!I}_{\alpha \theta'}+
		  \iota {\mathfrak z}^{I\!I}_{\alpha \theta'}+\sum_{\beta \in \mathsf{orb}(Y,\mathfrak a)}\iota  {\mathfrak n}_{\beta{\theta}'} {\mathfrak m}_{\alpha\beta}+
			\sum_{\substack{\beta \in \mathsf{orb}(Y,\mathfrak a)\\\Theta|_{Y'} = \theta' }} {\mathfrak b}_{\beta \Theta|_{Y}} {\mathfrak m}_{\alpha \beta}.
		\] 
		By Lemma \ref{lemma:cob-op-relns} and the fact that $\iota$ is a chain map, the above expression is equal to
		\begin{equation}\label{rel}
		  \partial^{\rm cell} \iota {\mathfrak n}_{\alpha \theta'} + \iota{\mathfrak n}_{\alpha \theta'} \partial^{\rm cell}+\iota {\mathfrak z}^{I}_{\alpha \theta'}
		  +\sum_{\substack{\beta \in \mathsf{orb}(Y,\mathfrak a)\\\Theta|_{Y'} = \theta' }} {\mathfrak b}_{\beta \Theta|_{Y}} {\mathfrak m}_{\alpha \beta}.
		\end{equation}
		The following relation is a counterpart of \eqref{eqn:B-reln}:
		\[\partial^{\rm cell} {\mathfrak b}_{\alpha \theta} + {\mathfrak b}_{\alpha \theta}  \partial^{\rm cell} = 
		\sum_{\beta \in \mathsf{orb}(Y,\mathfrak a)} {\mathfrak b}_{\beta \theta} {\mathfrak m}_{\alpha \beta} + \iota {\mathfrak z}_{\alpha \theta}.\]
		Summing this relation over all $\Theta\in \mathfrak Z(W)$ with the constraint that $\theta=\Theta|_{Y}$ and $\theta'=\Theta|_{Y'}$ gives 
		\[\partial^{\rm cell} \left(\sum_{\Theta|_{Y'} = \theta'} {\mathfrak b}_{\alpha \Theta|_Y}\right) + \left(\sum_{\Theta|_{Y'} = \theta'} {\mathfrak b}_{\alpha \Theta|_Y}\right) \partial^{\rm cell} = 
		\sum_{\substack{\beta \in \mathsf{orb}(Y,\mathfrak a)\\\Theta|_{Y'} = \theta' }} {\mathfrak b}_{\beta \Theta|_{Y}} {\mathfrak m}_{\alpha \beta} + \iota {\mathfrak z}^{I}_{\alpha \theta'}.\]
		Using this relation, we can see that \eqref{rel} is equal to the left hand side of \eqref{chain-map-ob-cob-map}, which completes the proof of the claim in this case.
	\item In the case that $\alpha'=\theta'$ for $\theta'\in \mathfrak Z(Y')$, the relation \eqref{chain-map-ob-cob-map} can be rewritten as 
	\[
	  0=
	  \pi \iota {\mathfrak n}_{\alpha\theta'} + 
	  \sum\limits_{\Theta|_Y = \theta,\,\Theta|_{Y'} = \theta'} \pi {\mathfrak b}_{\alpha\theta}+\sum\limits_{\Theta|_Y = \theta,\, \Theta|_{Y'} = \theta'} \fm_{\alpha\theta}.
	 \] 
	  This follows from the relations $\pi \iota = 0$ and $\pi {\mathfrak b}_{\alpha\theta} = {\mathfrak m}_{\alpha\theta}$. 
	  The latter relation can be verified with the same argument as in Lemma \ref{iota-b=m}. \qedhere
\end{itemize}
\end{proof}

\begin{cor}\label{cor:main-ineq-bp-1}
Suppose $W: Y \to Y'$ is a cobordism between rational homology spheres with $b^+(W) = 1$. Then the map $\varepsilon_0^W: Z(Y) \to Z(Y')$  sends $J_n^{(3)}(Y)$ into $J_{n-1}^{(3)}(Y')$. Equivalently, for any $\vartheta \in Z(Y)$ we have \[q_3(Y; \vartheta) \le q_3(Y'; \varepsilon_0^W(\vartheta)) + 1.\]
\end{cor} 

Note that the statement is vacuous if $\dim H^1(W, \partial W; \mathbb F_2) > 1$, as then $\varepsilon^W_0 = 0$ and $q_3(Y'; 0) = \infty$. In particular, we may assume that $b_1(W)=0$. The statement is only interesting for those $\vartheta$ with $\varepsilon_0^W(\vartheta) \ne 0$. When $Y, Y'$ are $\mathbb F_2$-homology spheres, the statement is that if $b^+(W) = 1$ and $H^1(W;\mathbb F_2) = 0$, then $q_3(Y) \le q_3(Y') + 1$. When the first cohomology group is nonzero, the statement is vacuous.

\begin{proof}
By Proposition \ref{prop:obs-cob-map}, there exists a morphism of $\mathcal S$-complexes $\widetilde \lambda_W:\widetilde C(Y) \to S \widetilde C(Y')$ with central component $\varepsilon_0^W$. By Proposition \ref{prop:filt-shift}, the suspension shifts the filtration $\{J^{(3)}_{n}(Y')\}$ on $Z(Y')$ down by $1$. By Proposition \ref{prop:tau-filtered}, the map $\epsilon_0^W: Z(Y) \to Z(Y')$ is filtered with respect to this shifted filtration on the codomain, proving the desired claim. 
\end{proof}

In fact, the preceding result has a natural extension to arbitrary cobordisms. The proof rests on the following elementary topological lemma:

\begin{lemma}\label{lemma:cob-decomp}
Suppose $W: Y \to Y'$ is a cobordism between rational homology spheres with $b^+(W) = k > 0$ and $b_1(W) = 0$. Then $W$ may be written as a composite $W = W_k \circ \cdots \circ W_{1}$, where each $W_i: Y_{i-1} \to Y_i$ is a cobordism between rational homology spheres with $b^+(W_i) = 1$ and $b_1(W_i) = 0$, where $Y = Y_0$ and $Y' = Y_k$. 
\end{lemma}
\begin{proof}
The cobordism $W$ may be obtained by sequentially attaching $1$-, $2$-, and $3$-handles to the outgoing end of $I \times Y$. Because $b_1(W) = b_1(Y) = 0$, we may attach all $1$-handles and a collection of $2$-handles to decompose $W$ as a composite of a rational homology cobordism $Y \to Y''$ and another cobordism $Y'' \to Y'$; similarly with the $3$-handles. Thus, it suffices to provide such a decomposition for a $2$-handle cobordism. 

Since $b^+(W)> 0$, there is a primitive linear combination $H = \sum_{i=1}^m a_iH_i$ with integer coefficients so that the self-intersection number $H\cdot H$ is a positive rational number. Performing handleslides if necessary, one may suppose $H = H_1$. Let $W_1 = (I\times Y)\cup H_1$ be the cobordism with a single $2$-handle $H_1$ attached. The outgoing boundary $Y_1$ is a rational homology sphere, and the cobordism satisfies $b_1(W_1) = 0$ and $b^+(W_1) = 1$. The embedding $Y_1 \subseteq W$ divides $W$ into two successive cobordisms $W_1:Y_0\to Y_1$ and $W':Y_1\to Y'$ so that $b^+(W') = k - 1$. The result now follows by induction.
\end{proof}

This in hand, we are able to prove a generalization of the main inequality to rational homology spheres.

\begin{theorem}\label{thm:main-ineq-QHS}
Suppose $W: Y \to Y'$ is a cobordism of rational homology spheres. Then the map $\varepsilon_0^W: Z(Y) \to Z(Y')$ sends $J_n^{(3)}(Y)$ into $J_{n-b^+(W)}^{(3)}(Y')$. Equivalently, for any $\vartheta \in Z(Y)$ we have \[q_3(Y; \vartheta) \le q_3(Y'; \varepsilon_0^W(\vartheta)) + b^+(W).\]
\end{theorem}

Again, this statement is vacuous unless $\dim H^1(W, \partial W; \mathbb F_2) = 1$.

\begin{proof}
If $b_1(W) > 0$, then in particular $\dim H^1(W, \partial W; \mathbb F_2) > 1$, so the statement is vacuous. If $b_1(W) = b^+(W) = 0$, then the statement was established as Proposition \ref{prop:generalized-monotonicity}. Thus, we may suppose $b_1(W) = 0$ and $b^+(W) > 0$, in which case Lemma \ref{lemma:cob-decomp} allows us to decompose $W = W_k \circ \cdots \circ W_1$, where each $W_i$ is a cobordism between rational homology spheres with $b_1(W_i) = 0$ and $b^+(W_i) = 1$. Write $\varepsilon_0^{(i)}: Z(Y) \to Z(Y_i)$ for the map induced by the cobordism $W_i \circ \cdots \circ W_1$, with $\varepsilon_0^{(0)} = 1 $ and $\varepsilon_0^{(k)} = \varepsilon_0^W$. Because the assignment $Y \mapsto Z(Y)$ and $W \mapsto \varepsilon_0^W$ is easily seen to be functorial, we have $\varepsilon_0^{W_i} \circ \varepsilon_0^{(i-1)} = \varepsilon_0^{(i)}$. By Corollary \ref{cor:main-ineq-bp-1}, we have for each $i$ that \[\varepsilon_0^{W_i}(J^{(3)}_{m}(Y_{i-1})) \subset J^{(3)}_{m-1}(Y_i).\] By induction, we therefore have that $\varepsilon_0^{(i)}(J^{(3)}_n(Y)) \subset J^{(3)}_{n-i}(Y_i)$ for each $0 \le i \le k$. Taking $i = k = b^+(W)$ gives the desired result.
\end{proof}

\begin{remark}
In fact, the proof establishes a stronger claim. A cobordism $W: Y \to Y'$ with $b^+(W) = 1$ induces a dg-module morphism $\widetilde C(Y, \fa) \to S \widetilde C(Y', \fa') \cong \Sigma \widetilde C(Y', \fa')$. The algebraic suspension $\Sigma \widetilde C$ is functorial up to homotopy for $\cS$-complexes \cite[Section 2.3]{DS:-unori-skein-tr}, and in fact the same formulas establish that $\Sigma \widetilde C$ is functorial up to homotopy for dg-modules over $\Lambda(\chi_1, \chi_2)$ of the form considered in this paper. Therefore, for each cobordism $W: Y \to Y'$ between rational homology spheres, composition gives a $\Lambda(\chi_1, \chi_2)$-equivariant map $\widetilde C(Y, \fa) \to \Sigma^{b^+(W)} \widetilde C(Y', \fa')$, whose central component is exactly $\epsilon_0^W$. 
\end{remark}

\begin{proof}[Proof of Theorem \ref{main-thm}(ii)]
If necessary, perform surgery on loops representing a basis of $H_1(W;\mathbb Z)/\text{Tors}$, so that $H_1(W;\mathbb F_2) = H_1(\partial W; \mathbb F_2) = 0$. Then the map $\varepsilon_0^W: Z(Y) \to Z(Y')$ is the identity map between two $1$-dimensional $\mathbb F_2$-vector spaces. Thus, $q_3(Y) \le q_3(Y') + b^+(W)$, giving the left-hand inequality of \eqref{b+-ineq}. The right-hand inequality follows by applying this to the same cobordism with reversed orientation, $\bar W: Y' \to Y$, and using that $q_3$ negates under orientation-reversal.
\end{proof}

We conclude by observing that the inequality \eqref{b+-ineq} implies the other half of the additivity formula Theorem \ref{main-thm}(iii). 

\begin{proof}[Proof of Theorem \ref{main-thm}(iii)]
By Proposition \ref{prop:froy-vs-cell}, the invariant $q_3$ discussed above coincides with the invariant discussed in Section \ref{background-proof-1}, and in particular we have the inequality $q_3(Y \# P) \ge q_3(Y) + 1$ of Proposition \ref{proposition:superadditivity}. Because $P$ is the $(-1)$-surgery on the left-handed trefoil, it bounds a manifold $X$ with $b^-(X) = 1$. Taking the boundary sum of $X$ with $I \times Y$ gives a cobordism $W: Y \to Y \# P$ with $H_1(W;\mathbb Z) = 0$ and $b^-(W) = 1$. It follows from \eqref{b+-ineq} that $q_3(Y \# P) - q_3(Y) \le 1$, establishing the desired equality.
\end{proof}

\addcontentsline{toc}{section}{References}

\bibliography{references}

\Addresses

\end{document}